\ifdefined\pdfoutput \pdfoutput=1 \fi
\documentclass[11pt]{amsart}

\usepackage[margin=1in]{geometry}

\usepackage{amsmath,amssymb,amsfonts,amsthm,mathtools}
\usepackage{bm}

\usepackage{graphicx}
\usepackage{subfigure}
\usepackage{booktabs}
\usepackage{multirow}
\usepackage{tikz}
\usetikzlibrary{arrows.meta, decorations.pathmorphing, calc}

\usepackage[hidelinks]{hyperref}

\usepackage{microtype}

\theoremstyle{plain}

\newtheorem{theorem}{Theorem}
\newtheorem{problem}{Problem}
\newtheorem{lemma}{Lemma}
\newtheorem{remark}{Remark}

\tikzset{
	wave/.style = {decorate, decoration={snake, amplitude=1.5pt, segment length=6pt}},
	>={Latex[length=2.5mm]},
}

\title[Time-domain iterative method for 3D inverse obstacle scattering]{A novel time-domain iterative method for a three-dimensional inverse acoustic obstacle scattering problem}

\author{Lu Zhao}
\address{Civil Aviation University of China, 2898 Jinbei Road, Tianjin 300300, China}
\email{zhaol@cauc.edu.cn}

\author{Heping Dong}
\address{Jilin University, Qianjin Street 2699, Changchun, Jilin 130012, China}
\email{dhp@jlu.edu.cn}
\thanks{Corresponding author: Heping Dong (dhp@jlu.edu.cn)}

\author{Zhiyong Cheng}
\address{Jilin University, Qianjin Street 2699, Changchun, Jilin 130012, China}
\email{chengzy22@mails.jlu.edu.cn}

\date{}


\begin{document}

\begin{abstract}
This paper concerns the three-dimensional forward and inverse acoustic obstacle scattering problem in the time domain. For the forward problem, a retarded potential formulation discretized by convolution quadrature and Galerkin methods is introduced. By introducing the retarded boundary integral defined on a homothetic surface, we propose a novel time-domain convolution quadrature based iterative method to reconstruct both the shape and location of a rigid obstacle. The retarded integral in the time domain is reformulated into a system of integrals in the $s$-domain. The resulting $s$-domain integrals are very fast to compute, as they only involve non-singular integrals over the homothetic surfaces. Moreover, the Fr\'echet derivative with respect to the boundary can be derived straightforwardly. We also prove that the scattered field generated by the homothetic surface converges to the exact field in the time domain. To improve the stability of the inversion algorithm, an incremental truncation technique is proposed, and numerical experiments confirm the effectiveness and robustness of our method.
\end{abstract}

\maketitle

\section{Introduction}
In this paper, we consider an inverse acoustic scattering problem of reconstructing a bounded obstacle from time-domain scattered field. Such inverse problems have attracted considerable attention due to their wide range of practical applications, including geophysical exploration~\cite{Borden}, biomedical imaging~\cite{Ammari1}, and nondestructive evaluation~\cite{Ammari2}.

We begin by formulating the mathematical model for the time-dependent acoustic obstacle scattering problem. Let \( D \subset \mathbb{R}^3 \) be a bounded open domain with a smooth boundary \( \Gamma_D \). The exterior region \( \mathbb{R}^3 \setminus \overline{D} \) is filled with homogeneous medium of unit mass density. We assume that the incident wave $u^{\rm inc}(\pmb x, t)=\varrho(t-|\pmb{x}-\pmb{x}_0|)/(4\pi|\pmb{x}-\pmb{x}_0|)$ emitted from a point $\pmb{x}_0\in \mathbb{R}^3\setminus \overline{D}$ with a causal and smooth signal $\varrho$. The forward scattering problem consists in determining the scattered field $u^{\rm sc}(\pmb{x},t)$ that satisfies the following initial-boundary problem
\begin{equation}\label{scattered field}
	\left\{
	\begin{aligned}
		&\partial_{t}^2u^{\rm sc}-\Delta u^{\rm sc}=0 &&{\rm in}~\left(\mathbb{R}^3\setminus \overline{D}\right)\times(0,\infty),\\
		&u^{\rm sc}=-u^{\rm inc}&&{\rm on}~\Gamma_D\times(0,\infty),\\
		&u^{\rm sc}(\cdot,0)=\partial_t u^{\rm sc}(\cdot,0)=0 &&{\rm in}~\mathbb{R}^3\setminus \overline{D}.\\
	\end{aligned}
	\right.
\end{equation}

Let \(R>0\) be a fixed constant and define the ball
\(
B := \left\{ \pmb{x} \in \mathbb{R}^3 : |\pmb{x}| < R \right\}
\)
such that \(\overline{D} \subset B\).
The boundary \(\Gamma_B\) of \(B\) is chosen as the observation surface, and 
let \(T>0\) be the truncated final time.
The near-field measurements are given by
\[
\Lambda := \left\{u^{\rm sc}(\pmb{x},t):(\pmb{x},t)\in \Gamma_B\times[0,T]\right\}.
\]
Figure \ref{illustration} provides a schematic illustration of the time-domain scattering problem in $\mathbb{R}^3$. With these settings, we formulate the inverse scattering problem as follows.
\begin{problem}[Inverse acoustic obstacle scattering]\label{prob:inverse}
Given the near-field measurements \(\Lambda\) on \(\Gamma_B\times[0,T]\), 
determine the shape and location of the obstacle \(D\subset\mathbb{R}^3\).
\end{problem}

\begin{figure}[h]
	\centering
	\begin{tikzpicture}[scale=0.9]
	
	\def\R{3.0} 
	\draw[line width=0.6pt] (0,0) circle (\R);
	\draw[dashed] (-\R,0) arc[start angle=180, end angle=360, x radius=\R, y radius=0.55*\R];
	\draw[dashed, opacity=.45] ( \R,0) arc[start angle=0,   end angle=180, x radius=\R, y radius=0.55*\R];
	\draw[dashed, rotate=35] (0,0) ellipse [x radius=0.55*\R, y radius=\R];
	\draw[dashed, rotate=-35] (0,0) ellipse [x radius=0.55*\R, y radius=\R];
	
	\node[above] at (0,1.12*\R) {observation surface $\Gamma_B$};
	
	
	\begin{scope}[shift={(0.25,0.35)}, rotate=12]
	\def\ax{1.05} 
	\def\ay{0.72} 
	\shade[ball color=blue!35, draw=black, line width=0.4pt, opacity=0.95]
	(0,0) ellipse [x radius=\ax, y radius=\ay];
	\draw[black, opacity=.8, line width=0.4pt]
	(0,0) ellipse [x radius=\ax, y radius=\ay];
	\draw[dashed, opacity=.35, line width=0.4pt, rotate=25]
	(0,0) ellipse [x radius=0.85*\ax, y radius=0.55*\ay];
	\draw[white, opacity=.35, line width=1.2pt]
	(-0.45*\ax, 0.35*\ay) .. controls (-0.15*\ax,0.55*\ay) and (0.15*\ax,0.45*\ay) ..
	(0.35*\ax, 0.25*\ay);
	\end{scope}
	
	\node at (0.4, -0.7) {\small obstacle $D$};

	\coordinate (S) at (-5.2,0.2);  

	\fill[red] ($(S)+(0.6,0)$) circle (2pt);
	\node[below left] at ($(S)+(0.65,0.5)$) {\small $\pmb{x}_0$};
	\draw[wave, line width=0.8pt] ($(S)+(0.6,0)$) -- ++(0.8,0);
	\draw[->, line width=0.8pt]    ($(S)+(1.3,0)$) -- (-\R +0.4,0.1);
	\node[below] at ($(S)+(-0,-0.35)$) {incident wave $u^{\rm inc}$};
	
	\coordinate (P) at ({0.9*\R},{0.6}); 
	\draw[wave, line width=0.8pt] ($(P)+(-0.7,0.3)$) -- ($(P)+(-0.2,0.1)$);
	\draw[->,  line width=0.8pt]  ($(P)+(-0.2,0.1)$) -- ($(P)+(0.4,0.3)$);
	\node[right] at ($(P)+(0.45,0.35)$) {scattered field $u^{\rm sc}$};
	
	
	\end{tikzpicture}
	\caption{An illustration of acoustic obstacle scattering.}\label{illustration}
\end{figure}
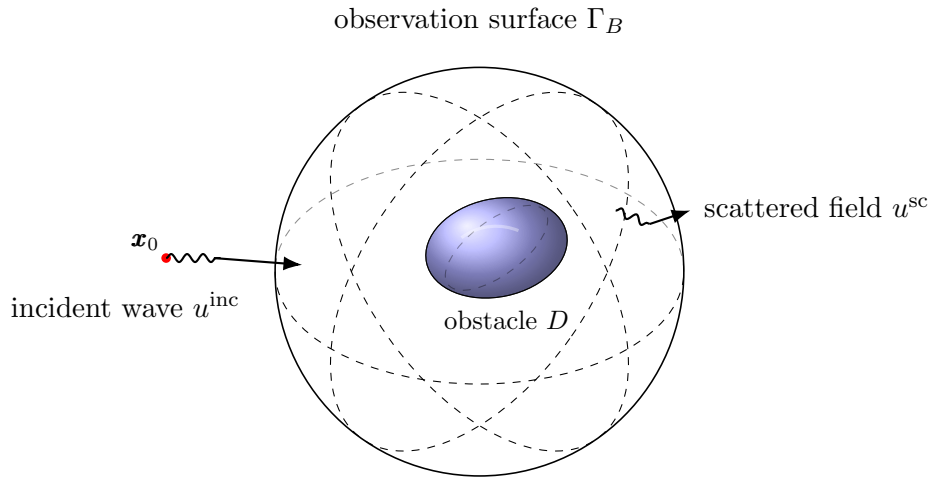

Time-domain broadband signals typically contain richer information and are often easier to capture in practice than frequency-domain data. In recent years, significant mathematical and computational progress has been achieved on time-domain acoustic obstacle scattering and the associated inverse scattering problems. Concerning the well-posedness analysis of the direct scattering problem, the theoretical foundations have been established in \cite{Cakoni2017,Guo2013}. 
Numerically, a common starting point is to formulate a boundary integral equation via retarded potential theory, which is subsequently discretized. 
Convolution quadrature is a classical and efficient technique for time discretization. In particular, the combination of convolution quadrature and boundary element methods provides an effective numerical framework for solving time-dependent acoustic scattering problems \cite{Lubich1994}. Moreover, by combining convolution quadrature with the scaled discrete Fourier transform, the original time-domain problem can be reformulated as a sequence of decoupled boundary integral equations in the \(s\)-domain, which can be efficiently solved using Galerkin spatial discretization \cite{CQ}. Furthermore, Runge-Kutta convolution quadrature combined with Galerkin spatial discretization leads to sparse convolution weights and hence improved computational efficiency \cite{Banjai2014}. More recently, a study has proposed an algorithm combining Runge-Kutta convolution quadrature in time with spectrally accurate spatial approximation \cite{Ganesh2023}. 
An alternative temporal discretization strategy is based on convolution splines. In particular, \cite{Davies2013} developed convolution-spline schemes for time discretization and combined them with Galerkin and collocation boundary element approximations for boundary integral formulations of time-dependent scattering problems. 
Building on this idea, high-order frameworks for three-dimensional acoustic scattering have been developed that couple convolution-spline time discretization with high-order quadrature rules in space \cite{Barnett2020,Epstein2016}. 
Moreover, a related combination of convolution splines in time and a Nystr\"om-type Galerkin method in space has been employed to solve time-domain acoustic scattering problems involving three-dimensional sound-soft and impedance obstacles \cite{Hou2024,Hou2025}. 
Beyond convolution-based approaches, frequency-time hybrid solvers have been developed that decompose the incident signal into wave packets and exploit their frequency-domain solutions to mitigate spatial dispersion \cite{Anderson2020,Bruno2024}. Other discretization strategies include coupling discontinuous Galerkin methods in space with explicit finite-difference in time \cite{Abboud2011}, space--time Galerkin methods \cite{Sauter2013}, and a space-time boundary element method for the wave equation in $3 + 1$ dimensions \cite{Polz2019}.

For time-domain inverse acoustic scattering problems, numerical algorithms are generally classified into two categories: qualitative and quantitative methods. On the qualitative side, the point source method \cite{Luke2006} uses the Fourier transform of time-dependent fields for pointwise reconstruction of the scatterer. Subsequently, the probe method \cite{Bukard2009} extends this idea to recover three-dimensional rough surfaces. For the direct inversion of obstacles from time-domain data, the linear sampling method has been applied to different boundary conditions \cite{Chen2,Haddar}, penetrable obstacles \cite{Guo2013}, and limited-aperture measurements \cite{Guo2016}. Detailed theoretical analyses can be found in \cite{Cakoni2021,Prunty2019}. The factorization method provides an explicit characterization of Dirichlet obstacles from causal far-field measurements \cite{Cakoni2019} and has been extended to impedance boundary conditions in \cite{Haddar2020}. Further developments include time-domain direct sampling methods, such as the total focusing method \cite{Guo20162}, which derives a sampling formula based on travel-time information to reconstruct both regular and extended obstacles with multiple components, as well as a novel method \cite{Guo2024} that establishes a connection between the time and frequency domains for reconstructing penetrable obstacles. The topological gradient method \cite{Kahlaoui2025} employs the leading term of the Kohn--Vogelius shape functional to enable efficient reconstruction for obstacles with Neumann boundary conditions. All of these methods are qualitative and require no a priori knowledge of the obstacle geometry or physical properties; instead, they rely on suitable indicator functions to determine whether a sampling point lies inside or outside the obstacle. However, they typically require scattered-field data generated by a large number of incident waves.

Quantitative methods are of particular interest in this work, as they typically provide more accurate and detailed reconstructions than qualitative approaches and, being based on linearization and optimization strategies, often require only scattered-field data from one or a few incident waves. An inverse scattering problem based on the domain derivative of the time-dependent wave equation is investigated in \cite{Knoller2025}. The resulting method linearizes the nonlinear operator via the Fr\'echet derivative, employs convolution quadrature for stable time discretization, and culminates in a Gauss--Newton scheme for boundary reconstruction. In \cite{Zhao2021}, a convolution quadrature-based nonlinear boundary integral equation method is proposed, which transforms time-domain boundary integrals into a system of decoupled equations in the \(s\)-domain, where the complex wavenumber depends on the temporal discretization, and an iterative strategy is developed for the \(s\)-domain field and data equations. Adjoint-based optimization has been used to recover the shape and location of a scatterer from transient wave data by minimizing a time-dependent misfit functional \cite{Sayag2022}, while a time-domain full-waveform inversion method based on a single level-set function has been proposed to represent multiple material interfaces \cite{Castro2024}. Related time-domain reconstruction methods for elastic and electromagnetic inverse scattering problems can be found in \cite{Ammari2013,Geng2025,Liu2023,Wang2021,Zhao2022}.

However, it is noteworthy that time-domain research has mainly focused on two-dimensional problems. Although some qualitative methods have been proposed for three-dimensional settings, they are rarely supported by extensive numerical experiments. More importantly, for the quantitative methods, classical nonlinear boundary integral equation methods alternate between solving a field equation for a boundary density and updating the boundary via a linearized data equation. The repeated treatment of weakly singular boundary integrals makes this approach efficient in two dimensions \cite{Zhao2021} but the high computational cost of fully three-dimensional simulations remains a major challenge. 
In this paper, we address the three-dimensional time-domain forward acoustic obstacle scattering problem using a convolution quadrature technique combined with a Galerkin scheme. Furthermore, we propose a novel iterative reconstruction algorithm for the corresponding inverse scattering problem, which completely avoids dealing with any singularity.
To the best of our knowledge, this is the first quantitative method of three-dimensional time-domain inverse scattering.

In this work, we first address the forward three-dimensional time-domain acoustic obstacle scattering problem. A retarded potential boundary integral equation is derived, which is then discretized in time using the convolution quadrature technique and in space via a Galerkin scheme based on spherical harmonics. Building upon this foundation, we focus on the inverse problem of reconstructing the unknown obstacle from time-dependent scattered field data. Motivated by the time-domain iterative framework for nonlinear integral equations in \cite{Zhao2021} and decomposition methods \cite{Colton}, we develop a novel and efficient iterative reconstruction algorithm. In particular, by employing analytic continuation, we establish a retarded boundary integral equation on a homothetic surface, which contains the complete geometric information of the obstacle boundary through a fixed contraction factor. By combining the convolution quadrature technique with a scaled discrete Fourier transform for time discretization, the retarded time-domain integral is reformulated as a system of decoupled equations in the \(s\)-domain. The inverse scheme then proceeds iteratively by alternately recovering the density function from an equivalent field equation and updating the boundary shape from an equivalent linearized data equation in the \(s\)-domain, until a stopping criterion is fulfilled.

%

The paper is organized as follows. Section 2 begins by deriving the time-domain boundary integral equations for the forward acoustic scattering problem and proposing an efficient Galerkin--Convolution quadrature (Galerkin-CQ) scheme for the numerical implementation. In Section 3, we propose a novel iterative approach based on the parametrization of the equivalent field and data equations, and establish the approximation property of the scattered field generated by the homothetic surface, together with the numerical implementation of the proposed iterative scheme. Section 4 presents a series of numerical experiments that demonstrate the effectiveness and robustness of the proposed time-domain forward and inverse approach. Finally, the paper is concluded in Section 5.

\section{The forward scattering problem}\label{direct scattering}
In this section, we study the forward time-domain acoustic scattering problem and develop a suitable functional framework for the analysis of numerical methods in subsequent inverse problems. We first introduce a suitable Laplace transform-based functional setting and derive a retarded potential boundary integral equation, then discretize it in time by convolution quadrature technique and in space by a Galerkin scheme based on spherical harmonics, yielding an efficient method for the forward problem.

\subsection{Laplace transform and theoretical analysis} 
In this subsection, we introduce appropriate Sobolev spaces involving both spatial and temporal variables. These spaces are chosen to ensure the well-posedness of the forward problem and to provide an appropriate functional space for the analysis of numerical methods in the context of inverse problems. The well-posedness theory for the direct scattering problem has been systematically developed in \cite{Chen2,Haddar,Lubich1994}, based on the following Sobolev spaces and the Laplace transform.

For a Hilbert space $X$, the spaces $\mathcal{D}'(\mathbb{R};X)$ and $\mathcal{S}'(\mathbb{R};X)$ denote the $X$-valued distributions and tempered distributions on the real line, respectively.
Define
\[
\mathcal{L}_{\sigma}'(\mathbb{R};X)=\left\lbrace f\in\mathcal{D}'(\mathbb{R};X),e^{-\sigma t}f(t)\in\mathcal{S}'(\mathbb{R};X) ~{\rm for}~ \sigma\in\mathbb{R}\right\rbrace ,
\]
the Fourier transform of $f\in\mathcal{L}_{\sigma}'(\mathbb{R};X)$ is defined by
\[
\mathcal{F}\left[ f\right] (\omega)=\int_{-\infty}^{\infty}e^{-{\rm i}\omega t}f(t)\mathrm{d}t,\quad \omega\in\mathbb{R},
\]
and the Laplace transform of $f$ is given by
\[
\mathcal{L}\left[ f\right] (s)=\int_{-\infty}^{\infty}e^{-st}f(t)\mathrm{d}t,\quad \Re(s)=\sigma>\sigma_0>0.
\]
This definition implies that for any $f \in \mathcal{L}_{\sigma}'(\mathbb{R};X)$, the Fourier transform of $e^{-\sigma t}f(t)$ is well-defined. Consequently, the Laplace transform of \( f \) exists and satisfies
\[
\mathcal{L}\left[ f\right] (s)=\mathcal{F}\left[ e^{-\sigma t}f\right] (\omega), \quad  s = \sigma +{\rm i} \omega.
\]%
The weighted Sobolev space $H_{\sigma}^p(\mathbb{R};X)$ is then defined by
{\small\[
H_{\sigma}^p(\mathbb{R};X):=\left\lbrace f\in\mathcal{L}_{\sigma}'(\mathbb{R};X):\int_{\sigma-\mathrm{i}\infty}^{\sigma+\mathrm{i}\infty}|s|^{2p}\left\| \mathcal{L}\left[ f\right] (s)\right\| _{X}^2\mathrm{d}s<\infty \right\rbrace ,\quad p\in\mathbb{R},~\sigma>0
\]}
with corresponding norm
\[
\left\| f\right\| _{H_{\sigma}^p(\mathbb{R};X)}=\left( \int_{\sigma-\mathrm{i}\infty}^{\sigma+\mathrm{i}\infty}|s|^{2p}\left\| \mathcal{L}\left[ f\right] (s)\right\| _{X}^2\mathrm{d}s\right) ^{1/2}.
\]
We denote by $H_{\sigma}^p(\mathbb{R}_{>a};X)$ the subspace of $H_{\sigma}^p(\mathbb{R};X)$ consisting of functions that vanish for $t < a$.

Applying the Laplace transform to the scattering problem \eqref{scattered field}, we obtain that $\breve{u}(s)=\mathcal{L}\left[ u\right] $ satisfies the following Helmholtz equation with complex wavenumber
\begin{equation*}\label{scattered field s-domain}
	\left\{
	\begin{aligned}
		&\Delta \breve{u}^{{\rm sc}}(\pmb{x},s)-s^2\breve{u}^{\rm sc}(\pmb{x},s)=0 &&{\rm in}~\mathbb{R}^3\setminus \overline{D},\\
		&\breve{u}^{\rm sc}(\pmb{x},s)=-\breve{u}^{\rm inc}(\pmb{x},s)&&{\rm on}~\Gamma_D.\\
	\end{aligned}
	\right.
\end{equation*}
Let $\Omega\subset\mathbb{R}^3$ be a bounded domain with boundary $\Gamma_{\Omega}$. We introduce the following frequency-dependent norm on the standard Sobolev space $H^1$($\Omega$)
\[
\left\| u\right\| _{1,s,\Omega}:=\left( \int_{\Omega}|\nabla u(\pmb{x},s)|^2+|su(\pmb{x},s)|^2\mathrm{d}\pmb{x}\right) ^{1/2},
\]
which is equivalent to the standard $H^1(\Omega)$ norm when $s\neq0$. 
We denote by $H^q(\Gamma_\Omega)$ the standard trace spaces on the boundary. It is well known that the dual space of $H^{1/2}(\Gamma_\Omega)$ is $H^{-1/2}(\Gamma_\Omega)$, under the $L^2$ inner product
\[
\left\langle u,v\right\rangle _{\Gamma_\Omega}=\int_{\Gamma_\Omega}u\overline{v}\mathrm{d}\gamma.
\]

\subsection{Retarded potential boundary integral equation method}
We provide some mathematical preparations for time-domain obstacle scattering problem in this subsection. For wave propagation and scattering, the fundamental solution to the three-dimensional wave equation is given by
\begin{equation*}
G(|\pmb{x}-\pmb{y}|,t):=\frac{\delta(t-|\pmb{x}-\pmb{y}|)}{4\pi|\pmb{x}-\pmb{y}|},\quad \pmb{x}\neq\pmb{y},
\end{equation*}
where $\delta$ denotes the Dirac delta distribution. The retarded single layer potential on $\Gamma_D$ is expressed as
{\small \begin{equation}\label{1}
	(S_D\psi)(\pmb{x},t):=\int_{0}^{t}\int_{\Gamma_D} G(|\pmb{x}-\pmb{y}|,t-\tau)\psi(\pmb{y},\tau)\mathrm{d}s_{\pmb{y} }\mathrm{d}\tau,~~ \pmb{x}\in\mathbb{R}^3\setminus\Gamma_D,~t\in(0,\infty),
\end{equation}}
where \(\psi\) denotes the density. Furthermore, the associated single layer boundary operator is defined as
{\small \begin{equation}\label{2}
	(\mathcal{S}_{D}\psi)(\pmb{x},t):=\int_{0}^{t}\int_{\Gamma_D}G(|\pmb{x}-\pmb{y}|,t-\tau)\psi(\pmb{y},\tau)\mathrm{d}s_{\pmb{y} }\mathrm{d}\tau,\quad \pmb{x}\in \Gamma_D,~t\in(0,\infty).
\end{equation}}

The following lemma presents the well-posedness result for the forward problem, which can be found in \cite{Chen2,Lubich1994}. 
\begin{lemma}\label{lemma31}
Let \( u^{\rm inc}(\pmb{x}, t) \in H_{\sigma}^p(\mathbb{R}_{>0}; H^{1/2}(\Gamma_D)) \) for arbitrary real $p$. Then there exists a unique solution \( u^{\rm sc}(\pmb{x}, t) \in H_{\sigma}^{p-3/2}(\mathbb{R}_{>0}; H^1(\mathbb{R}^3 \setminus \overline{D})) \) to the following exterior boundary value problem
\begin{equation*}
\left\lbrace 
\begin{aligned}
    &\partial_t^2 u^{\rm sc} - \Delta u^{\rm sc} = 0 && {\rm{in }}~ \left(\mathbb{R}^3 \setminus \overline{D}\right) \times \mathbb{R}, \\
    &u^{\rm sc} = 0 &&{\rm{in }}~  \left(\mathbb{R}^3 \setminus \overline{D}\right) \times (-\infty,0], \\
    &u^{\rm sc} = -u^{\rm inc} && {\rm{on }}~ \Gamma_D \times \mathbb{R}.
\end{aligned}
\right.
\end{equation*}
Furthermore, \( u^{\rm sc} \) satisfies the estimate
\[
\| u^{\rm sc} \|_{H_{\sigma}^{p-3/2}(\mathbb{R}_{>0}; H^1(\mathbb{R}^3 \setminus \overline{D}))}
\leq C(\sigma) \| u^{\rm inc} \|_{H_{\sigma}^p(\mathbb{R}_{>0}; H^{1/2}(\Gamma_D))},
\]
and the solution \( u^{\rm sc} \) can be represented as
\(
u^{\rm sc} = S_{D} \mathcal{S}_{D}^{-1} ( -u^{\rm inc} ),
\)
where \( S_{D}  \) and \( \mathcal{S}_{D} \) are defined in \eqref{1}-\eqref{2}.
\end{lemma}

The proof is completed by applying the following estimate~\cite[eqn.~(2.27)]{Lubich1994}
\[
\left\| \breve{u}^{\rm sc} \right\|_{1, s, \mathbb{R}^3 \setminus \overline{D}} 
\leq \frac{C(\sigma_0)}{\sigma}|s|^{3/2} \left\| \breve{u}^{\rm inc} \right\|_{H^{1/2}(\Gamma_D)},
\]
where \( \Re(s) = \sigma > \sigma_0 > 0 \).

To compute the scattered field $u^{\rm sc}$ described in \eqref{scattered field}, we choose an appropriate time interval $[0,T]$ such that the energy of the scattered field within the region of interest becomes negligible for $t > T$. Within this interval, the scattered field can be represented by the retarded single layer potential
\begin{equation}\label{single layer}
	u^{\rm sc}(\pmb{x},t)=\int_{0}^{t}\int_{\Gamma_D}G(|\pmb{x}-\pmb{y}|,t-\tau)\psi(\pmb{y},\tau)\,\mathrm{d}s_{\pmb{y}}\mathrm{d}\tau,~ (\pmb{x},t)\in(\mathbb{R}^3\setminus\overline{D})\times[0,T].
\end{equation}
A straightforward calculation shows that the representation \eqref{single layer} satisfies both the wave equation and the initial conditions given in \eqref{scattered field}. Furthermore, by applying the jump relation of the single layer potential as \( \pmb{x} \to \Gamma_D \), we obtain the following boundary integral equation 
\begin{equation}\label{boundary integral equation}
	\int_{0}^{t}\int_{\Gamma_D}G(|\pmb{x}-\pmb{y}|,t-\tau)\psi(\pmb{y},\tau)\,\mathrm{d}s_{\pmb{y}}\mathrm{d}\tau=-u^{\rm inc}(\pmb{x},t),\quad(\pmb{x},t)\in\Gamma_D\times[0,T].
\end{equation}

\subsection{Convolution quadrature}\label{3.2}
This subsection details the time discretization of the boundary integral equation via the convolution quadrature method. We first restate \eqref{boundary integral equation} in simplified form
\begin{equation}\label{CQfunction}
	(w*\psi)(t):=\int_{0}^{t}w(t-\tau)\psi(\tau)\mathrm{d}\tau,
\end{equation}
where the integral operator $w$ is defined as
\[
(w(t-\tau)\psi(\tau))(\pmb{x})=\int_{\Gamma_D}G(|\pmb{x}-\pmb{y}|,t-\tau)\psi(\pmb{y},\tau)\mathrm{d}s_{\pmb{y}},\quad \pmb{x}\in\Gamma_D.
\]
The Laplace transform of $w$ is given by
\[
\mathcal{W}(s)=\int_{0}^{+\infty}w(t)e^{-st}\mathrm{d}t,
\]
since $w = 0$ for $t < 0$. The explicit form of $W(s)$ is
\[
(\mathcal{W}(s)g)(\pmb{x})=\int_{\Gamma_D} G_s(|\pmb{x}-\pmb{y}|,s)g(\pmb{y})\mathrm{d}s_{\pmb{y}},\quad \pmb{x}\in\Gamma_D,
\]
where \(G_s(|\pmb{x} - \pmb{y}|, s) = \frac{e^{-s|\pmb{x} - \pmb{y}|}}{4\pi |\pmb{x} - \pmb{y}|}\) is the fundamental solution of the Helmholtz equation \(\Delta u + \kappa^2 u = 0\) with complex wave number \(\kappa = \mathrm{i}s\). By the inverse Laplace transform, \eqref{CQfunction} can be represented as
\begin{equation*}\label{convolution}
	\begin{aligned}
		(w*\psi)(t)=&\int_{0}^{t}\left( \frac{1}{2\pi \mathrm{i}}\int_{\sigma-\mathrm{i}\infty}^{\sigma+\mathrm{i}\infty}e^{s\tau}\mathcal{W}(s)\mathrm{d}s\right) \psi(t-\tau)\mathrm{d}\tau\\
		=&\frac{1}{2\pi \mathrm{i}}\int_{\sigma-\mathrm{i}\infty}^{\sigma+\mathrm{i}\infty}\mathcal{W}(s)\left( \int_{0}^{t}e^{s\tau}\psi(t-\tau)\mathrm{d}\tau\right) \mathrm{d}s.
	\end{aligned}
\end{equation*}
Note that the function \( y(t) := \int_0^t e^{s\tau} \psi(t - \tau)\,\mathrm{d}\tau \) is the unique solution to the differential equation
\begin{equation*}
	\left\lbrace \begin{aligned}
		&y'(t)=sy(t)+\psi(t),\quad~t\in[0,\infty),\\
		&y(0)=0.\end{aligned}\right. 
\end{equation*}

Many numerical schemes can be applied to the above differential equations, and here, we adopt the unconditionally stable third-order backward difference formula (BDF3) \cite{sayas}. The time interval $[0,T]$ is divided into $N+1$ equidistant time instants, where
\begin{equation*}
	t_n=n\Delta t,\quad n=0,1,\cdots,N,~\Delta t=T/N.
\end{equation*}
The discretization of equation \eqref{CQfunction} then becomes
\begin{equation}\label{equationend}
	(w* \psi)(t_n):=\sum_{j=0}^{N}\omega_{n-j}^{\Delta t,\lambda}(\mathcal{W})\psi_j=-u^{\rm inc}(\pmb{x},t_n)=-u_n^{\rm inc},\quad n=0,\cdots,N,
\end{equation}
where the convolution weights are given by
\begin{equation}\label{omega}
	\omega_j^{\Delta t,\lambda}(\mathcal{W}):=\frac{\lambda^{-j}}{N+1}\sum_{l=0}^{N}\mathcal{W}(s_l)\zeta_{N+1}^{lj}.
\end{equation}
Here $\lambda < 1$, $\zeta_{N+1} = e^{2\pi{\rm{i}}/(N+1)}$, $s_l =  \gamma(\lambda {\zeta_{N+1}^{-l}})/{\Delta t}$, $\gamma(\zeta) =\sum\limits_{k=1}^{3}\frac{1}{k}(1-\zeta)^k$ and $\psi_j = \psi^{\Delta t}(t_j)$ is the discrete approximation of $\psi(t_j)$.	
	
Substituting \eqref{omega} into \eqref{equationend}, yields for $l=0,\cdots,N$ that
\begin{equation}\label{field equation}
	(\mathcal{W}(s_l)\hat{\psi}_l)(\pmb{x})= \int_{\Gamma_D} G_{s}(|\pmb{x}-\pmb{y}|,s_l)\hat{\psi}_l(\pmb{y})\mathrm{d}s_{\pmb{y}} =-\hat{u}_l^{\rm inc}(\pmb{x}),~\pmb{x}\in\Gamma_D,
\end{equation}
where $\hat{\psi}_l$ and $\hat{u}_l^{\rm inc}$ are the scaled discrete Fourier transforms of $\psi_n$ and $u_n^{\rm inc}$ with respect to $n$, i.e.,
\[
\hat{\psi}_l:=\sum_{n=0}^{N}\lambda^n \psi_n\zeta_{N+1}^{-ln},\quad \hat{u}_l^{\rm inc}:=\sum_{n=0}^{N}\lambda^n u_n^{\rm inc}\zeta_{N+1}^{-ln}.
\]
Thus, we can obtain the density $\psi$ by the inverse transform
\[
\psi_n=\frac{\lambda^{-n}}{N+1}\sum_{j=0}^{N}\hat{\psi}_j\zeta_{N+1}^{nj}.
\]
The time discretization method described above for equation \eqref{CQfunction} is called the convolution quadrature method. For further details of the convolution quadrature method, we refer to \cite{CQ, sayas}.
	
Given the values of $\hat{\psi}_l$, the scaled discrete Fourier transform of $u_n^{\rm sc}(\cdot) = u^{\rm sc}(\cdot, t_n)$ with respect to $n$ can be computed by
\begin{equation}\label{data equation}
    \hat{u}_l^{\rm sc}(\pmb{x}) = \int_{\Gamma_D} G_s(|\pmb{x} - \pmb{y}|, s_l)\, \hat{\psi}_l(\pmb{y})\mathrm{d}s_{\pmb{y}}, \quad \pmb{x} \in \mathbb{R}^3 \setminus \overline{D},~l = 0, \cdots, N.
\end{equation}
Consequently, the scattered field $u^{\rm sc}$ can be derived via the inverse transform
\[
u_n^{\rm sc} = \frac{\lambda^{-n}}{N+1} \sum_{j=0}^{N} \hat{u}_j^{\rm sc} \zeta_{N+1}^{nj}.
\]

\subsection{Galerkin discretization}
The convolution quadrature method in time generates a system of boundary integral equations in the $s$-domain \eqref{field equation}-\eqref{data equation}. We now address the spatial discretization by adopting a Galerkin-type scheme based on spherical harmonics to approximate $\hat{\psi}_l$ in \eqref{field equation}. The scattered field $u^{\rm sc}$ is subsequently recovered via the inverse scaled discrete Fourier transform.

The boundary integral operator given in \eqref{field equation} can be divided into 
\begin{equation}\label{singular integral}
	(\mathcal{W}(s_l)\hat{\psi}_l)(\pmb{x})=\int_{\Gamma_D}\left( \frac{1}{|\pmb{x}-\pmb{y}|}w_1^l(\pmb{x},\pmb{y})+w_2^l(\pmb{x},\pmb{y})	\right) \hat{\psi}_l(\pmb{y})\mathrm{d}s_{\pmb{y}},
\end{equation}
where 
\begin{align*}
	w_1^l(\pmb{x},\pmb{y}) &= \frac{e^{-k_2^l|\pmb{x}-\pmb{y}|}\cos(\kappa_1^l|\pmb{x}-\pmb{y}|)}{4\pi}, \\
	w_2^l(\pmb{x},\pmb{y}) &=
	\begin{cases}
		\frac{e^{-k_2^l|\pmb{x}-\pmb{y}|}\mathrm{i}\sin(\kappa_1^l|\pmb{x}-\pmb{y}|)}{4\pi|\pmb{x}-\pmb{y}|}, & \pmb{x}\neq\pmb{y},\\[1ex]
		\frac{\mathrm{i}\kappa_1^l}{4\pi}, & \pmb{x}=\pmb{y}.
	\end{cases}
\end{align*}
with $\mathrm{i}s_l=\kappa_1^l+\mathrm{i}\kappa_2^l$. The operator \( \mathcal{W} \) involves a weakly singular kernel \( \frac{1}{|\pmb{x}-\pmb{y}|} \). 

We further assume that the boundary $\Gamma_D$ has a star-shaped representation
\[
	\Gamma_D=\left\lbrace \pmb p_D(\hat{\pmb{x}})=\pmb{c}+r(\hat{\pmb{x}})\hat{\pmb{x}}:\pmb{c}=(c_1,c_2,c_3)^\top,\hat{\pmb{x}}\in \mathbb{S}^2\right\rbrace,
\]
with $\mathbb{S}^2=\left\lbrace \hat{\pmb{x}}(\theta,\phi)=(\sin\theta\cos\phi,\sin\theta\sin\phi,\cos\theta)^\top:\theta\in[0,\pi],\phi\in[0,2\pi)\right\rbrace $ \\and $r(\hat{\pmb{x}})>0$ for all $\hat{\pmb{x}}\in\mathbb{S}^2$. Then the point on $\Gamma_D$ can be described as
\[
	\pmb p_D(\hat{\pmb{x}}):=(c_1,c_2,c_3)^\top+r(\hat{\pmb{x}})\hat{\pmb{x}}.
\] 
The observation surface $\Gamma_B$ is parametrized by $\pmb p_B(\hat{\pmb{x}})=R\hat{\pmb{x}}$.

Based on the parametrization map $\pmb{p}_D:\mathbb{S}^2\to\Gamma_D$, the boundary integral equations \eqref{field equation}-\eqref{data equation} defined on $\Gamma_D$ can be transformed on the unit sphere $\mathbb{S}^2$, which means
\[
	\int_{\Gamma_D}g(\pmb{x})\mathrm{d}s_{\pmb{x}}=\int_{\mathbb{S}^2}g(\pmb{p}_D(\hat{\pmb{x}}))J_{D}(\hat{\pmb{x}})\mathrm{d}s_{\hat{\pmb{x}}},
\]
where $J_D(\hat{\pmb{x}})$ is the Jacobian of the transformation $\pmb{p}_D$. Therefore, the singular integral equation \eqref{singular integral} can be transformed to 
\begin{equation}\label{parametrization singular}
	(\mathcal{W}(s_l)\hat{\varPsi}_l)(\hat{\pmb{x}})=\int_{\mathbb{S}^2}\left( \frac{1}{|\hat{\pmb{x}}-\hat{\pmb{y}}|}\tilde{w}_1^l(\hat{\pmb{x}},\hat{\pmb{y}})+\tilde{w}_2^l(\hat{\pmb{x}},\hat{\pmb{y}})	\right) \hat{\varPsi}_l(\hat{\pmb{y}})\mathrm{d}s_{\hat{\pmb{y}}},
\end{equation}
where $\hat{\varPsi}_l(\hat{\pmb{x}})=\hat{\psi}_l(\pmb{p}_D(\hat{\pmb{x}}))$ and 
\[
	\tilde{w}_1^l(\hat{\pmb{x}},\hat{\pmb{y}})=R(\hat{\pmb{x}},\hat{\pmb{y}})w_1^l(\pmb{p}_D(\hat{\pmb{x}}),\pmb{p}_D(\hat{\pmb{y}}))J_D(\hat{\pmb{y}}),\quad \tilde{w}_2^l(\hat{\pmb{x}},\hat{\pmb{y}})=w_2^l(\pmb{p}_D(\hat{\pmb{x}}),\pmb{p}_D(\hat{\pmb{y}}))J_D(\hat{\pmb{y}}),
\]
with $R(\hat{\pmb{x}},\hat{\pmb{y}})=\frac{|\hat{\pmb{x}}-\hat{\pmb{y}}|}{|\pmb{p}_D(\hat{\pmb{x}})-\pmb{p}_D(\hat{\pmb{y}})|}$.
\vspace{1ex}
	
From \cite{Ganesh,Ivanyshyn1}, the spherical harmonics
\begin{align}\label{spherical harmonics}
		Y_{k,j}(\hat{\pmb{x}})=Y_{k,j}(\hat{\pmb{x}}(\theta,\phi))=c_k^jP_k^{|j|}(\cos\theta)e^{\mathrm{i}j\phi},\quad k\in\mathbb{N}, ~|j|\leq k,
\end{align}
form a complete orthonormal system in $L^2(\mathbb{S}^2)$, where the coefficients are given by
\[ 
	c_k^j=(-1)^{(|j|-j)/2}\sqrt{\frac{2k+1}{4\pi}\frac{(k-|j|)!}{(k+|j|)!}}
\]
and \( P_k^{|j|} \) denotes the associated Legendre function with degree \( k \) and order \( |j| \). Thus, we choose the following $(\breve{n}+1)^2$-dimensional space to approximate the density $\hat{\varPsi}_l$ defined on the unit sphere
\begin{equation*}
	X_{\breve{n}}=\mathrm{span}\left\lbrace Y_{k,j}(\hat{\pmb{x}}):0\leq k\leq \breve{n},|j|\leq k\right\rbrace ,
\end{equation*}
that means 
\[
	\hat{\varPsi}_l(\hat{\pmb{x}})=\sum_{k=0}^{\breve{n}}\sum_{j=-k}^{k}a_{kj}^lY_{k,j}(\hat{\pmb{x}}),\quad a_{kj}^l\in\mathbb{C}.
\]
Accordingly, the determination of \(\hat{\varPsi}_l\) in~\eqref{parametrization singular} is reduced to the computation of the spherical harmonic coefficients \(a_{kj}^l\).

The standard Galerkin method is to seek an approximate solution $\hat{\varPsi}_l\in X_{\breve{n}}$ satisfying
\begin{equation*}
\left(\mathcal{W}	(s_l)\hat{\varPsi}_l, \chi_{\breve{n}}\right)=\left(-\hat{u}_l^{\rm inc},\chi_{\breve{n}}\right),\quad {\rm for~all}~ \chi_{\breve{n}}\in X_{\breve{n}}.
\end{equation*}
For $l=0,1,\cdots,N$, we compute $\hat{\varPsi}_l$ and $\hat{u}_l^{\rm sc}$ by using the Galerkin method proposed in \cite{Ganesh,Ivanyshyn1}.

\section{The inverse scheme}
 
The classical nonlinear integral equation method involves two steps: solving for the density function via the field equation at the current approximation, and linearizing the data equation to obtain an update. The treatment of singular integrals is unavoidable in the first step. Although the two-dimensional case has been effectively addressed in \cite{Zhao2021}, the three-dimensional case entails substantially higher computational and memory costs due to the need to solve $N+1$ singular integral equations. In the following subsection, we introduce a novel and more efficient computational approach for the time-domain inverse scattering problem.

\subsection{The Novel Iterative Approach}
Inspired by the iterative framework of the nonlinear integral equation method and the auxiliary surface idea of the decomposition method \cite{Colton}, we propose a novel iterative algorithm to reconstruct both the shape and location of the obstacle simultaneously. The key point technique is analytic continuation which avoids singular integrals in the field equation \eqref{field equation}, and thereby significantly simplifying the numerical implementation. In particular, we construct the single layer potential on a homothetic surface andprove that the scattered field generated by the homothetic surface can arbitrarily approximate the exact scattered field.

\subsubsection{Homothetic surface}
Based on the representation on $\Gamma_D$, we introduce a homothetic surface which captures the geometric information of the obstacle boundary.
Specifically, we fix a geometric contraction factor $0<\varsigma<1$ and define an interior domain $D'\subset D$ with the homothetic boundary
\[
\Gamma_{D'}=\left\{\pmb p_{D'}(\hat{\pmb{x}})=\pmb{c}+\varsigma\, r(\hat{\pmb{x}})\hat{\pmb{x}}:\ \hat{\pmb{x}}\in\mathbb{S}^2\right\}.
\]
The scattered field $u_{D',\mathfrak{g}}^{\rm sc}$ can be represented as 
\begin{equation}\label{data equation new}
	u_{D',\mathfrak{g}}^{\rm sc}(\pmb{x},t)=\int_{0}^{t}\int_{\Gamma_{D'}}G(|\pmb{x}-\pmb{y}|,t-\tau)\mathfrak{g}(\pmb{y},\tau)\mathrm{d}s_{\pmb{y}}\mathrm{d}\tau.
\end{equation}
It is easy to verify that the above single layer potential in \eqref{data equation new} satisfies the wave equation in $\mathbb{R}^3 \setminus \overline{D}$. According to the Dirichlet boundary condition, the unknown density $\mathfrak{g}$ should satisfy
\begin{equation}\label{field equation new}
	\int_{0}^{t}\int_{\Gamma_{D'}}G(|\pmb{x}-\pmb{y}|,t-\tau)\mathfrak{g}(\pmb{y},\tau)\mathrm{d}s_{\pmb{y}}\mathrm{d}\tau=-u^{\rm inc}(\pmb{x},t),\quad (\pmb{x},t)\in\Gamma_D\times[0,T].
\end{equation}

The following two theorems show that the scattered field generated by the homothetic surface can approximate the exact scattered field arbitrarily.
\begin{theorem}\label{4.1}
	The single layer operator $$\mathcal{S}_{D'}:H_{\sigma}^p(\mathbb{R}; H^{-1/2}(\Gamma_{D'})) \to H_{\sigma}^{p-1}(\mathbb{R}; H^{1/2}(\Gamma_D)),$$
	defined by 
	\[
	\left( \mathcal{S}_{D'}\mathfrak{g}\right) (\pmb{x},t)=\int_{0}^{t}\int_{\Gamma_{D'}}G(|\pmb{x}-\pmb{y}|,t-\tau)\mathfrak{g}(\pmb{y},\tau)\mathrm{d}s_{\pmb{y}}\mathrm{d}\tau,\quad \pmb{x}\in\Gamma_D,
	\]
is injective and has dense range.
\end{theorem}
\begin{proof}
	To prove injectivity of $\mathcal{S}_{D'}:H_{\sigma}^p(\mathbb{R}; H^{-1/2}(\Gamma_{D'}))\to H_{\sigma}^{p-1}(\mathbb{R}; H^{1/2}(\Gamma_D))$, assume  $\mathcal{S}_{D'}\mathfrak{g}=0$ in $H_{\sigma}^{p-1}(\mathbb{R}; H^{1/2}(\Gamma_D))$ for $\mathfrak{g}\in H_{\sigma}^p(\mathbb{R}; H^{-1/2}(\Gamma_{D'}))$. Applying the Laplace transform to $\mathcal{S}_{D'}\mathfrak{g}$ yields
	\begin{equation*}
		\begin{aligned}
			\mathcal{L}\left[ \mathcal{S}_{D'}\mathfrak{g}\right] (\pmb{x},s)=&\int_{-\infty}^{\infty}e^{-st}\int_{\Gamma_{D'}}\frac{\mathfrak{g}(\pmb{y},t-|\pmb{x}-\pmb{y}|)}{4\pi|\pmb{x}-\pmb{y}|}\mathrm{d}s_{\pmb{y}}\mathrm{d}t\\
			=&\int_{\Gamma_{D'}}\frac{1}{4\pi|\pmb{x}-\pmb{y}|}\int_{-\infty}^{\infty}e^{-st}\mathfrak{g}(\pmb{y},t-|\pmb{x}-\pmb{y}|)\mathrm{d}t\mathrm{d}s_{\pmb{y}}\\
			=&\int_{\Gamma_{D'}}\frac{e^{-s|\pmb{x}-\pmb{y}|}}{4\pi|\pmb{x}-\pmb{y}|}\left( \int_{-\infty}^{\infty}e^{-s\tau}\mathfrak{g}(\pmb{y},\tau)\mathrm{d}\tau\right) \mathrm{d}s_{\pmb{y}}\\
			=&\int_{\Gamma_{D'}}\frac{e^{-s|\pmb{x}-\pmb{y}|}}{4\pi|\pmb{x}-\pmb{y}|}\breve{\mathfrak{g}} (\pmb{y},s) \mathrm{d}s_{\pmb{y}},
		\end{aligned}
	\end{equation*}
	where $\breve{\mathfrak{g}}=\mathcal{L}\left[ \mathfrak{g}\right] $ denotes the Laplace transform of $\mathfrak{g}$ and $s\in\sigma+{\rm i}\mathbb{R}$ with $\sigma>\sigma_0>0$.
	
	Define 
	\[
	 v (\pmb{x},s):=\breve{\mathcal{S}}_{D'}\breve{\mathfrak{g}}(\pmb{x},s)=\int_{\Gamma_{D'}}\frac{e^{-s|\pmb{x}-\pmb{y}|}}{4\pi|\pmb{x}-\pmb{y}|}\breve{\mathfrak{g}} (\pmb{y},s) \mathrm{d}s_{\pmb{y}},\quad \pmb{x}\in  \mathbb{R}^3\setminus\Gamma_{D'}.
	\]
	Note that the kernel $\frac{e^{-s|\pmb{x}-\pmb{y}|}}{4\pi|\pmb{x}-\pmb{y}|}$ decays exponentially as $|\pmb{x}|\to\infty$.
	Since $ v (\pmb{x},s)$ satisfies $\Delta v +(\mathrm{i}s)^2 v =0$ in $\mathbb{R}^3\setminus\overline{D}$ and vanishes on $\Gamma_D$, the uniqueness result for the exterior problem of the Helmholtz equation with complex wavenumber \cite{Lubich1994} implies $ v =0$ in $\mathbb{R}^3\setminus\overline{D}$ for all $s\in\sigma+{\rm i}\mathbb{R}$. By the unique continuation property, $ v =0$ extends to $\mathbb{R}^3\setminus\overline{D'}$. Furthermore, the continuity of single layer potential gives $ v =0$ on $\Gamma_{D'}$. As $ v $ also satisfies $\Delta v +(\mathrm{i}s)^2 v =0$ in $D'$, the uniqueness of the interior Helmholtz equation problem yields $ v =0$ in $D'$. Thus $ v =0$ in $\mathbb{R}^3$. In this case, the jump relationship for single layer potential implies that $\breve{\mathfrak{g}}(\pmb{y},s)\in H^{-1/2}(\Gamma_{D'})$ vanishes for all $s\in\sigma+{\rm i}\mathbb{R}$. The injectivity of $\mathcal{S}_{D'}$ then follows from the inverse Laplace transform.
	
	To prove the denseness of the range of \[\mathcal{S}_{D'}:H_{\sigma}^p(\mathbb{R}; H^{-1/2}(\Gamma_{D'}))\to H_{\sigma}^{p-1}(\mathbb{R}; H^{1/2}(\Gamma_D)),\] it suffices to show that its $L^2$-adjoint 
	\[
	\mathcal{S}_{D'}^{*}:H_{-\sigma}^{1-p}(\mathbb{R}; H^{-1/2}(\Gamma_D))\to H_{-\sigma}^{-p}(\mathbb{R}; H^{1/2}(\Gamma_{D'})),
	\]
	given explicitly by
	\[
	\left( \mathcal{S}_{D'}^{*}\mathfrak{h}\right) (\pmb{y},t)=\int_{\Gamma_D}\frac{\mathfrak{h}(\pmb{x},t+|\pmb{x}-\pmb{y}|)}{4\pi|\pmb{x}-\pmb{y}|}\mathrm{d}s_{\pmb{x}},
	\]
	is injective.
	
	Suppose $\mathfrak{h}\in H_{-\sigma}^{1-p}(\mathbb{R}; H^{-1/2}(\Gamma_D))$ satisfies $\mathcal{S}_{D'}^{*}\mathfrak{h}=0$ in $H_{-\sigma}^{-p}(\mathbb{R}; H^{1/2}(\Gamma_{D'}))$, that means we can take Fourier transform to $e^{\sigma t}\mathcal{S}_{D'}^{*}\mathfrak{h}$, i.e.,
	\begin{equation*}
		\begin{aligned}
			\mathcal{F }\left[e^{\sigma t}\mathcal{S}_{D'}^{*}\mathfrak{h}\right]=&\int_{-\infty}^{\infty}e^{-\mathrm{i}\omega t}e^{\sigma t}\int_{\Gamma_D}\frac{\mathfrak{h}(\pmb{x},t+|\pmb{x}-\pmb{y}|)}{4\pi|\pmb{x}-\pmb{y}|}\mathrm{d}s_{\pmb{x}}\mathrm{d}t\\
			=&\int_{\Gamma_D}\frac{1}{4\pi|\pmb{x}-\pmb{y}|}\int_{-\infty}^{\infty}e^{-(-\sigma+\mathrm{i}\omega)t}\mathfrak{h}(\pmb{x},t+|\pmb{x}-\pmb{y}|)\mathrm{d}t\mathrm{d}s_{\pmb{x}}\\
			=&\int_{\Gamma_D}\frac{e^{s|\pmb{x}-\pmb{y}|}}{4\pi|\pmb{x}-\pmb{y}|}\left( \int_{-\infty}^{\infty}e^{-s\tau}\mathfrak{h}(\pmb{x},\tau)\mathrm{d}\tau\right) \mathrm{d}s_{\pmb{x}},
		\end{aligned}
	\end{equation*}
	where $s=-\sigma+\mathrm{i}\omega$. Define
	\[
	\breve{\mathfrak{h}}(\pmb{x},s):=\int_{-\infty}^{\infty}e^{-s\tau}\mathfrak{h}(\pmb{x},\tau)\mathrm{d}\tau=\mathcal{F}\left[ e^{\sigma t}\mathfrak{h}\right] ,
	\]
	and the potential
	\[
	 v ^*(\pmb{y},s):=\breve{\mathcal{S}}_{D'}^*\breve{\mathfrak{h}}(\pmb{y},s)=\int_{\Gamma_D}\frac{e^{s|\pmb{x}-\pmb{y}|}}{4\pi|\pmb{x}-\pmb{y}|}\breve{\mathfrak{h}}(\pmb{x},s)\mathrm{d}s_{\pmb{x}},\quad \pmb{y}\in\mathbb{R}^3\setminus \Gamma_D,
	\]
	for $s\in -\sigma+\mathrm{i}\mathbb{R}$ with $\sigma>\sigma_0>0$. 

	Since $ v ^*$ satisfies $\Delta v ^*+(-\mathrm{i}s)^2 v ^*=0$ in $D'$ and vanishes on $\Gamma_{D'}$, the uniqueness of the interior problem implies $ v ^*=0$ in $D'$ for all $s\in-\sigma+\mathrm{i}\mathbb{R}$. By the unique continuation property, $ v ^*=0$ extends to $D$. Furthermore, the continuity of single layer potential gives $ v ^*=0$ on $\Gamma_D$. Similar to $\breve{\mathcal{S}}_{D'}$, the kernel $\frac{e^{s|\pmb{x}-\pmb{y}|}}{4\pi|\pmb{x}-\pmb{y|}}$ of $\breve{\mathcal{S}}_{D'}^{*}$ also decays exponentially as $|\pmb{y}|\to\infty$ since $\Re(s)=-\sigma<0$. As $ v ^*$ satisfies Helmholtz equation $\Delta v ^*+(-\mathrm{i}s)^2 v ^*=0$ in $\mathbb{R}^3\setminus\overline{D}$,  the uniqueness of the exterior problem yields $ v ^* = 0$ in $\mathbb{R}^3$. Consequently, the jump relationship implies $\breve{\mathfrak{h}} = 0$ in $H^{-1/2}(\Gamma_D)$ for all $s \in -\sigma + \mathrm{i}\mathbb{R}$.
	Finally, applying the Parseval identity shows that the density $\mathfrak{h}$ vanishes on $\Gamma_D$ in the time domain, thus establishing the injectivity.
\end{proof}

\begin{theorem}\label{4.2}
	Under the assumptions of Lemma~\ref{lemma31}, let $u_D^{\rm sc}$ be the exact scattered field in $\mathbb{R}^3 \setminus \overline{D}$. 
	Then, for any $\varepsilon>0$, there exists $\mathfrak{g}\in H_{\sigma}^p(\mathbb{R}; H^{-1/2}(\Gamma_{D'}))$ such that
	\begin{equation*}
		\left\|u_{D',\mathfrak{g}}^{\rm sc}(\pmb{x},t)-u_D^{\rm sc}(\pmb{x},t)\right\|_{H_{\sigma}^{p-5/2}(\mathbb{R};H^1(\mathbb{R}^3\setminus\overline{D}))}
		\le \varepsilon.
	\end{equation*}
\end{theorem}

\begin{proof}
	By Theorem \ref{4.1}, given any $\varepsilon > 0$, there exists $\mathfrak{g} \in H_{\sigma}^p(\mathbb{R}; H^{-1/2}(\Gamma_{D'}))$ such that
	\[
	\left\| \left( \mathcal{S}_{D'}\mathfrak{g}\right) (\pmb{x},t)+u^{\rm inc}(\pmb{x},t)\right\| _{H_{\sigma}^{p-1}(\mathbb{R};H^{1/2}(\Gamma_D))}<\varepsilon.
	\]
	Since the scattered field depends continuously on the boundary data by Lemma \ref{lemma31}, it follows that
	\begin{align*}
		&\left\| u_{D',\mathfrak{g}}^{\rm sc}\left( \pmb{x},t\right)-u_D^{\rm sc}(\pmb{x},t)\right\| _{H_{\sigma}^{p-5/2}(\mathbb{R};H^1(\mathbb{R}^3\setminus\overline{D}))}\\
		\leq &~C\left\|u_{D',\mathfrak{g}}^{\rm sc}(\pmb x,t)|_{\Gamma_D}-u_D^{\rm sc}(\pmb x,t)|_{\Gamma_D}\right\|_{H_{\sigma}^{p-1}(\mathbb{R};H^{1/2}(\Gamma_D))}\\
		= &~C\left\| \left( \mathcal{S}_{D'}\mathfrak{g}\right) (\pmb{x},t)-u_D^{\rm sc}(\pmb{x},t)|_{\Gamma_D}\right\| _{H_{\sigma}^{p-1}(\mathbb{R};H^{1/2}(\Gamma_D))}.
	\end{align*}
	The proof is completed by noting that $u_D^{\rm sc} = - u^{\rm inc}$ on $\Gamma_D$.
\end{proof}
\begin{remark}
	Unfortunately, the proofs of Theorems \ref{4.1} and \ref{4.2} can be established only in the space $H_{\sigma}^{p}(\mathbb{R}; X)$, rather than in $H_{\sigma}^{p}(\mathbb{R}_{>0}; X)$, to ensure the density of the range of $\mathcal{S}_{D'}$. As shown in\cite{Chen2}, counterexamples can be readily constructed to demonstrate that the operator $\mathcal{S}_{D'}$ does not possess a dense range in the space $H_{\sigma}^{p}(\mathbb{R}_{>0}; X)$.
\end{remark}

\subsubsection{Iterative framework}
We have demonstrated the approximation property of the scattered field generated by the homothetic surface in Theorem \ref{4.2}. This result provides the mathematical foundation for the iterative approach developed below.

The time discretization of the time-domain equivalent field equation \eqref{field equation new} is performed using the convolution quadrature method described in Subsection \ref{3.2}. This discretization leads to a system of equivalent field equations in the $s$-domain
\begin{equation}\label{field equation new sdomain}
	(\mathcal{W}_{D'}(s_l)\hat{\mathfrak{g}}_l)(\pmb{x})=-\hat{u}_l^{\mathrm{inc}}(\pmb{x}),\quad l=0,\cdots,N,~\pmb{x}\in\Gamma_D,
\end{equation}
where the operator $\mathcal{W}_{D'}: H^{-1/2}(\Gamma_{D'}) \to H^{1/2}(\Gamma_D)$ is defined by
\[
(\mathcal{W}_{D'}(s_l)\hat{\mathfrak{g}}_l)(\pmb{x})=\int_{\Gamma_{D'}}G_s(|\pmb{x}-\pmb{y}|,s_l)\hat{\mathfrak{g}}_l(\pmb{y})\mathrm{d}s_{\pmb{y}}.
\]
Notably, the equivalent field equations \eqref{field equation new sdomain} in the $s$-domain still retain the boundary information of the unknown obstacle $D$. However, these equations are ill-posed. To address this issue, we employ Tikhonov regularization by solving the following regularized equivalent field equations for $l=0,\cdots,N$, i.e., 
\begin{equation}\label{regularization}
	 \left( \alpha \mathbb{I}+\mathcal{W}^*_{D'}(s_l)\mathcal{W}_{D'}(s_l)\right)\hat{\mathfrak{g}}_l^\alpha(\pmb{x})=-\left( \mathcal{W}^*_{D'} (s_l)\hat{u}_l^{\rm inc}\right) (\pmb{x}),\quad \pmb{x}\in\Gamma_D,
\end{equation}
where $\alpha>0$ is a regularization parameter, \(\mathbb{I}\) is an identity operator, and the $L^2$-adjoint operator $\mathcal{W}_{D'}^*: H^{1/2}(\Gamma_D) \to H^{-1/2}(\Gamma_{D'})$ is defined by
\[
(\mathcal{W}^*_{D'}(s_l)\psi)(\pmb{x})=\int_{\Gamma_D}\overline{G_s(|\pmb{x}-\pmb{y}|,s_l)}\psi(\pmb{y})\mathrm{d}s_{\pmb{y}}.
\]

The corresponding equivalent data equations in the $s$-domain are given by
\begin{equation}\label{data equation new s domain}
	\mathcal{W}_{D'}^{B}(s_l)\hat{\mathfrak{g}}_l^\alpha(\pmb{x}):=\int_{\Gamma_{D'}}G_s(|\pmb{x}-\pmb{y}|,s_l)\hat{\mathfrak{g}}_l^\alpha(\pmb{y})\mathrm{d}s_{\pmb{y}}=\hat{u}_{D,l}^{\rm sc}(\pmb{x}),\quad \pmb{x}\in\Gamma_B,
\end{equation}
for $l=0,\cdots,N$. Therefore, based on the regularized equivalent field equations \eqref{regularization} and the equivalent data equations \eqref{data equation new s domain} in the $s$-domain, we propose the iterative method as follows:
\begin{itemize}
	\item  Given an initial guess for $\Gamma_D$, solve the regularized equivalent field equation \eqref{regularization} to obtain the  regularized density $\hat{\mathfrak{g}}_l^\alpha$;
	\item For the fixed $\hat{\mathfrak{g}}_l^\alpha$, linearize the equivalent data equation \eqref{data equation new s domain} with respect to $\Gamma_D$, since the homothetic surface $\Gamma_{D'}$ carries the geometric information of $\Gamma_D$. The linearized equation yields a boundary update $\Delta \Gamma$, and the boundary is updated by \( \Gamma^{(k+1)}_D = \Gamma^{(k)}_D + \Delta \Gamma^{(k)}\).
\end{itemize}
These two steps are repeated alternately until the relative error of the scattered field meets a predetermined tolerance. A complete description of the algorithmic procedure is given in Table~\ref{algorithm} in the next subsection.

\subsection{Numerical  implementation}
In this section, we provide a detailed presentation of the iterative algorithm. In particular, we derive its parametrized and discretized formulations.
\subsubsection{Parametrization} 
For $l=0,\cdots,N$, we use $W_{D'}$ and $W_{D'}^{B}$ to represent the parametrized integral operators $\mathcal{W}_{D'}$ and $\mathcal{W}_{D'}^{B}$, respectively,
\begin{align*}
	\left(W_{D'}(s_l)\hat{\mathfrak{h}}_l^\alpha\right)(\hat{\pmb{x}})=&\int_{\mathbb{S}^2}\frac{e^{-s_l|\pmb p_D(\hat{\pmb{x}})-\pmb p_{D'}(\hat{\pmb{y}})|}}{4\pi|\pmb p_D(\hat{\pmb{x}})-\pmb p_{D'}(\hat{\pmb{y}})|}\hat{\mathfrak{h}}_l^\alpha(\hat{\pmb{y}})J_{D'}(\hat{\pmb{y}})\mathrm{d}s_{\hat{\pmb{y}}},\\
	\left({W}_{D'}^{B}(s_l)[\pmb{p}_D]\hat{\mathfrak{h}}_l^\alpha\right)(\hat{\pmb{x}})=&\int_{\mathbb{S}^2}\frac{e^{-s_l|\pmb p_B(\hat{\pmb{x}})-\pmb p_{D'}(\hat{\pmb{y}})|}}{4\pi|\pmb p_B(\hat{\pmb{x}})-\pmb p_{D'}(\hat{\pmb{y}})|}\hat{\mathfrak{h}}_l^\alpha(\hat{\pmb{y}})J_{D'}(\hat{\pmb{y}})\mathrm{d}s_{\hat{\pmb{y}}},
\end{align*}
where $\hat{\mathfrak{h}}_l^\alpha(\hat{\pmb{y}})=\hat{\mathfrak{g}}_l^\alpha(\pmb{p}_{D'}(\hat{\pmb{y}}))$ and $J_{D'}$ is the Jacobian of the parameterization.

Consequently, the regularized equivalent field equations \eqref{regularization} can be transformed as parameterized integral equations
\begin{equation}\label{field equation parama}
	(\alpha \mathbb{I}+{W}^*_{D'}(s_l){W}_{D'}(s_l))\hat{\mathfrak{h}}_l^\alpha={W}^*_{D'}(s_l)\hat{f}_l,\quad l=0,\cdots,N,
\end{equation}
with $\hat{f}_l=-\hat{u}_l^{\rm inc}(\pmb{p}_{D}(\hat{\pmb{x}}))$. The parametrized adjoint operator $\mathcal{W}^*_{D'}$ is given by
\[
\left({W}^*_{D'}(s_l)\varphi\right)(\hat{\pmb{y}})=\int_{\mathbb{S}^2}\frac{e^{-\overline{s_l}|\pmb p_D(\hat{\pmb{x}})-\pmb p_{D'}(\hat{\pmb{y}})|}}{4\pi|\pmb p_D(\hat{\pmb{x}})-\pmb p_{D'}(\hat{\pmb{y}})|}\varphi(\hat{\pmb{x}})J_{D}(\hat{\pmb{x}})\mathrm{d}s_{\hat{\pmb{x}}},
\]
where $\overline{s_l}$ denotes the complex conjugate of $s_l$ and $J_D$ is the Jacobian of the parametrization. Similarly, the equivalent data equations \eqref{data equation new s domain} in parametrized form are expressed as
\begin{equation}\label{data equation parama}
	W_{D'}^{B}(s_l)[\pmb{p}_D]\hat{\mathfrak{h}}_l^\alpha
	= \hat{w}_{D,l}^{\rm sc},
	\quad l=0,\cdots,N,
\end{equation}
with $\hat{w}_{D,l}^{\rm sc}(\hat{\pmb{x}})
	= \hat{u}_{D,l}^{\rm sc}(\pmb{p}_B(\hat{\pmb{x}}))$.

The linearization of \eqref{data equation parama} with respect to $\pmb p_D$ involves the Fr\'echet derivative of the integral operators $W_{D'}^{B}(s_l)[\pmb{p}_D]$. This derivative can be expressed explicitly as
\begin{equation*}
\left(\left({W_{D'}^{B}}(s_l)\right)'[\pmb{p}_D]\hat{\mathfrak{h}}_l^\alpha\right)\pmb{q}(\hat{\pmb{x}})=\int_{\mathbb{S}^2}\widetilde{S}_l(\hat{\pmb{x}},\hat{\pmb{y}})\hat{\mathfrak{h}}_l^\alpha(\hat{\pmb{y}})J_{D'}(\hat{\pmb{y}})\mathrm{d}s_{\hat{\pmb{y}}},
\end{equation*}
where
\[
\widetilde{S}_l(\hat{\pmb{x}},\hat{\pmb{y}})=(\pmb{p}_{D'}(\hat{\pmb{y}})-\pmb{p}_B(\hat{\pmb{x}}))\cdot\tilde{\pmb{q}}(\hat{\pmb{y}})\left(-s_l-\frac{1}{|\pmb{p}_B(\hat{\pmb{x}})-\pmb{p}_{D'}(\hat{\pmb{y}})|}\right)\frac{e^{-s_l|\pmb p_B(\hat{\pmb x})-\pmb p_{D'}(\hat{\pmb y})|}}{4\pi|\pmb p_B(\hat{\pmb{x}})-\pmb p_{D'}(\hat{\pmb{y}})|^2}.
\]
Here, $\pmb{q}(\hat{\pmb{x}})=\Delta \pmb{c}+\Delta r(\hat{\pmb{x}})\hat{\pmb{x}}$ and $\tilde{\pmb{q}}(\hat{\pmb{x}})=\Delta \pmb{c}+\varsigma\Delta r(\hat{\pmb{x}})\hat{\pmb{x}}$ represent the updates to $\pmb{p}_D$ and $\pmb{p}_{D'}$, respectively.

Let $\pmb p_{D}^{(ll)}$ denote the $ll$-th iteration of the parametrized boundary $\pmb p_{D}$.  The update $\pmb{q}$ is obtained by solving
\begin{equation}\label{update equation}
	\left(\left({W_{D'}^{B}}(s_l)\right)'[\pmb{p}_D^{(ll)}]\hat{\mathfrak{h}}_l^\alpha\right)\pmb{q}=\hat{w}_{D,l}^{\rm sc}-{W}_{D'}^{B}(s_l)[\pmb{p}_D^{(ll)}]\hat{\mathfrak{h}}_l^\alpha.
\end{equation}
In addition, we employ the following relative error estimator to terminate the iterative procedure
\begin{equation}\label{error}
	E_{ll}=\frac{\left\|\hat{w}_{D,l}^{\rm sc}-{W}_{D'}^{B}(s_l)[\pmb{p}_D^{(ll)}]\hat{\mathfrak{h}}_l^\alpha\right\| _{L^2 }}{\left\|\hat{w}_{D,l}^{\rm sc}\right\| _{L^2 }}\leq\epsilon,
\end{equation}
where $\epsilon$ is a small positive constant chosen based on the noise level.

For the boundary surface $\Gamma_D$, since the radial function $r$ is real-valued,
we employ the following space of spherical harmonics to approximate  its update $\Delta r$
\begin{equation*}
	\Delta r(\hat{\pmb x})=\sum_{{k}=1}^{M}\sum_{j=1}^{k}\alpha_{kj}\Im (Y_{k,j}(\hat{\pmb x}))+\sum_{k=0}^{M}\sum_{j=0}^{k}\beta_{kj}\Re (Y_{k,j}(\hat{\pmb x})),
\end{equation*}
where the integer $M\geq 0$ denotes the truncation number and the spherical harmonics $Y_{k,j}$ are defined in \eqref{spherical harmonics}. For the special case of $M=0$, the expression reduces to $\Delta r(\hat{\pmb x})=\beta_{00}\Re(Y_{0,0}(\hat{\pmb x}))$, corresponding to a constant radial perturbation.  
\begin{remark}[Incremental truncation technique]
It is observed in the numerical examples that the reconstructions are sensitive to the initial guess, since the proposed approach is intrinsically of Newton type. To mitigate this dependence, we propose an \emph{incremental truncation technique}: starting from $M=0$ to obtain an approximation of the obstacle's position and scale, we then gradually increase $M$ to $1,2,\cdots,M_{\max}$ to further refine the detailed shape until the stopping criterion is satisfied. This strategy allows the measured data to be repeatedly exploited during the iterations and leads to satisfactory reconstruction results.	
\end{remark}

Next, we present the iterative algorithm in Table \ref{algorithm} for the inverse problem.

\begin{table}[h]
\centering\caption{}\label{algorithm}
\begin{tabular}{cp{.8\textwidth}}
	\toprule
	\multicolumn{2}{l}{{\bf Algorithm:}\quad Iterative procedure for inverse obstacle scattering problem} \\
	\midrule
	
	{\bf Step 1} & Collect the scattered field data $u_D^{\rm sc}(\pmb{x},t)$ on $\Gamma_B$ over $[0,T]$ corresponding to the incident wave $u^{\rm inc}(\pmb{x},t)$ emitted from the source point $\pmb{x}_0$, and then take the scaled discrete Fourier transform of both $u^{\rm inc}$ and $u_D^{\rm sc}$;\\
	{\bf Step 2} & Select a suitable geometric contraction factor $\varsigma$, an initial star-shaped surface $\pmb p_{D}^{(0)}$ for the boundary $\Gamma_D$, the error tolerance $\epsilon$, a fixed constant $loop$, and the maximal truncation number $M_{\max}$. 
	Set $M=0$ and initialize $ll=0$; \\
	{\bf Step 3} & For the $ll$-th iteration $\pmb p_{D}^{(ll)}$, solve the regularized density $\hat{\mathfrak{h}}_l^\alpha$ on $\pmb p_{D'}^{(ll)}$ by \eqref{field equation parama} at $l=[ll/loop]$;\\
	{\bf Step 4} & Calculate the scattered field data ${W}_{D'}[s_l, \pmb{p}_D^{(ll)}]\hat{\mathfrak{h}}_l^\alpha$ for the $ll$-th iteration and evaluate the error $E_{ll}$ defined in \eqref{error};\\
	{\bf Step 5} & If $E_{ll}\geq\epsilon$, solve \eqref{update equation} to update the approximation $\pmb p_{D}^{(ll+1)}=\pmb p_{D}^{(ll)}+\pmb{q}$; otherwise, take $\pmb p_{D}^{(ll)}$ as the final reconstruction of $\Gamma_D$ and terminate the iteration. \\
	{\bf Step 6} & If $\mathrm{mod}\!\bigl(ll+1,(N+1)\,loop\bigr)=0$, and $M<M_{\max}$, then set $M=M+1$; otherwise keep $M$ unchanged. Then set $ll=ll+1$ and go to \textbf{Step 3}.\\

	\bottomrule
\end{tabular}
\end{table}

The notation $[x]$ in {\bf Step 3} represents the greatest integer less than or equal to the real number $x$. The purpose of {\bf Step 3} is to iterate over each fixed complex wavenumber $s_l$, meaning that both the inner iteration for each fixed $s_l$ and the outer loop over all $s_l$ values are included. In addition, $\mathrm{mod}(a,b)$ denotes the remainder of $a$ upon division by $b$; in particular, $\mathrm{mod}(a,b)=0$ means that $a$ is a multiple of $b$.

\begin{remark}[Multiple-incident extension]
	The iterative approach in Table \ref{algorithm} can be naturally extended to cases with multiple incident waves. For incident waves $u_1^{\rm inc}, \cdots, u_k^{\rm inc}$, we collect the scattered fields $u_{D,1}^{\rm sc},\cdots, u_{D,k}^{\rm sc}$  on $\Gamma_B$ over $[0,T]$. The iterative approach is
	modified as follows:
		
		\noindent{\bf (i). }Solve the parametrized densities $\hat{\mathfrak{h}}_{l,1}^\alpha, \cdots , \hat{\mathfrak{h}}_{l,k}^\alpha$ by equivalent field equation; 
		
		\noindent{\bf (ii). }Replace the calculation of $E_{ll}$ in equation \eqref{error} with the following vector form
			\begin{equation*}
				E_{ll}=\frac{\left\|\left(\hat{w}_{D,l,1}^{\rm sc},\cdots,\hat{w}_{D,l,k}^{\rm sc}\right)^\top -\left({W}_{D'}^{B}(s_l)[ \pmb{p}_D^{(ll)}]\hat{\mathfrak{h}}_{l,1}^\alpha,\cdots,{W}_{D'}^{B}(s_l)[\pmb{p}_D^{(ll)}]\hat{\mathfrak{h}}_{l,k}^\alpha\right)^\top\right\| _{L^2 }}{\left\|\left(\hat{w}_{D,l,1}^{\rm sc},\cdots,\hat{w}_{D,l,k}^{\rm sc}\right)^\top\right\| _{L^2 }}\leq\epsilon,
			\end{equation*}
	
		\noindent{\bf (iii). }Transform the linearized equivalent data equation \eqref{update equation} as the following vectorized system in Step 5:
	\begin{equation*}
				\left(\left({W_{D'}^{B}}(s_l)\right)'[\pmb{p}_D^{(ll)}]\hat{\mathfrak{h}}_{l,1}^\alpha,\cdots,\left({W_{D'}^{B}}(s_l)\right)'[ \pmb{p}_D^{(ll)}]\hat{\mathfrak{h}}_{l,k}^\alpha\right)^\top\pmb{q}=\left(\hat{f}_{D,l,1}^{\rm sc},\cdots,\hat{f}_{D,l,k}^{\rm sc}\right)^\top,
		\end{equation*}
where $\hat{f}_{D,l,j}^{\rm sc} = \hat{w}_{D,l,j}^{\rm sc}-{W}_{D'}^{B}(s_l)[\pmb{p}_D^{(ll)}]\hat{\mathfrak{h}}_{l,j}^\alpha$, $j=1,\cdots,k$.

	As demonstrated in Section~\ref{num ex}, employing multiple incident waves significantly improves the accuracy and robustness of the reconstruction.
\end{remark}

\begin{remark}[High computational efficiency]
    It is worth mentioning that the proposed method generally exhibits high computational efficiency. All numerical tests were carried out in MATLAB 2025b on a personal laptop (48 GB RAM; 4.50 GHz Apple M4 Pro) using Apple's BLAS library. For Examples~1--4 in subsection \ref{inve ex} with a single incident field, the reconstruction takes about $3.106762$ seconds (around $215$ iterations). In the multi-incident setting, the total computation time remains within $5$ seconds. 
    For more complex geometries (e.g., Example~5 in subsection \ref{inve ex}), the computational cost may increase due to a substantial growth in the number of computational nodes and iterative steps.
\end{remark}

\subsubsection{Discretization}
In this subsection, we describe the full discretization of the equivalent field equations \eqref{field equation parama} and linearized equivalent data equations \eqref{update equation}. For the continuous function $f$, the Gauss trapezoidal product rule is given by
\[
\int_{\mathbb{S}^2}f(\hat{\pmb{x}})\mathrm{d}s_{\hat{\pmb{x}}}\approx\frac{\pi}{n+1}\sum_{j=0}^{2n+1}\sum_{i=1}^{n+1}\varpi_i f(\hat{\pmb{x}}(\theta_i,\phi_j))
\]
with quadrature points
\(
\phi_j = {j\pi}/{(n+1)},~\theta_i = \arccos z_i
\) and weights 
\[
\varpi_i = \frac{2(1-z_i^2)}{(n+1)^2 [P_n(z_i)]^2},
\] where $z_i$ denotes the zeros of the Legendre polynomial $P_{n+1}$. For clarity and brevity, we simply denote $f(\hat{\pmb{x}}(\theta_i,\phi_j))$ as $f(\theta_i,\phi_j)$ in the following.

Notably, the equivalent field equations \eqref{field equation parama} do not contain singularities, which makes it possible to obtain approximate solutions efficiently through the Gauss quadrature rule.
Next, a discretization scheme for the linearized equation \eqref{update equation} is presented, and the corresponding update is derived via Tikhonov regularization \cite{Tikhonov}.

Define
\begin{align*}
	L_{1,l}^{k,q}(\zeta,\xi,\theta,\phi):=&\widetilde{K}_l(\zeta,\xi,\theta,\phi) \Im(Y_{k,q}(\theta,\phi)),\\
	L_{2,l}^{k,q}(\zeta,\xi,\theta,\phi):=&\widetilde{K}_l(\zeta,\xi,\theta,\phi) \Re(Y_{k,q}(\theta,\phi)),
\end{align*}
where 
\begin{align*}
	\widetilde{K}_l(\zeta,\xi,\theta,\phi):=&\mathcal{A}_{1,l}(\zeta,\xi,\theta,\phi)\sin\theta\cos\phi+\mathcal{A}_{2,l}(\zeta,\xi,\theta,\phi)\sin\theta\sin\phi\\
	&+\mathcal{A}_{3,l}(\zeta,\xi,\theta,\phi)\cos\theta,\\
	\mathcal{A}_{1,l}(\zeta,\xi,\theta,\phi):=&K_l(\zeta,\xi,\theta,\phi)(c_1+\varsigma r(\theta,\phi)\sin\theta\cos\phi-R\sin\zeta\cos\xi),\\
	\mathcal{A}_{2,l}(\zeta,\xi,\theta,\phi):=&K_l(\zeta,\xi,\theta,\phi)(c_2+\varsigma r(\theta,\phi)\sin\theta\sin\phi-R\sin\zeta\sin\xi),\\
	\mathcal{A}_{3,l}(\zeta,\xi,\theta,\phi):=&K_l(\zeta,\xi,\theta,\phi)(c_3+\varsigma r(\theta,\phi)\cos\theta-R\cos\zeta),
\end{align*}
and
{\small\[
K_l(\zeta,\xi,\theta,\phi)=\left(-s_l-\frac{1}{|\pmb p_B(\zeta,\xi)-\pmb p_{D'}(\theta,\phi)|}\right)\frac{e^{-s_l|\pmb p_B(\zeta,\xi)-\pmb p_{D'}(\theta,\phi)|}}{4\pi|\pmb p_B(\zeta,\xi)-\pmb p_{D'}(\theta,\phi)|^2}J_{D'}(\theta,\phi).
\]}

Let $\zeta_r:=\pi r/\tilde{n}$, $ r=0,\cdots,\tilde{n}-1$, $\xi_s:=\pi s/\tilde{n}$, $s=0,\cdots,2\tilde{n}-1$ be equidistant sets of quadrature knots, then 
$$
R(\sin\zeta_r\cos\xi_s,\sin\zeta_r\sin\xi_s,\cos\zeta_r)^\top,
$$
are equidistantly distributed nodes on $\Gamma_B$.

Substituting the above definitions into the equation \eqref{update equation} and using the Gauss trapezoidal product rule, we get the fully discrete linear system
{\small
\begin{equation}\label{discretization update equation}
		\sum_{{p}=1}^{3}\Delta c_{p} B_{p,l}(\zeta_r,\xi_s)+\sum_{k=1}^{M}\sum_{q=1}^{k}\alpha_{kq}B_{1,l}^{k,q}(\zeta_r,\xi_s)+\sum_{k=0}^{M}\sum_{q=0}^{k}\beta_{kq}B_{2,l}^{k,q}(\zeta_r,\xi_s)=\hat{f}_{D,l}^{\rm sc}(\zeta_r,\xi_s),
\end{equation}
}for $r=0,\cdots,\tilde{n}-1$, $s=0,\cdots,2\tilde{n}-1$, where the system matrices are defined by
\begin{align*}
	B_{p,l}(\zeta_r,\xi_s)=&\frac{\pi}{n+1}\sum_{j=0}^{2n+1}\sum_{i=1}^{n+1}\varpi_i \mathcal{A}_{p,l}(\zeta_r,\xi_s,\theta_i,\phi_j)\hat{\mathfrak{h}}_l^\alpha(\theta_i,\phi_j),\quad p=1,2,3,\\
	B_{p,l}^{k,q}(\zeta_r,\xi_s)=&\varsigma\frac{\pi}{n+1}\sum_{j=0}^{2n+1}\sum_{i=1}^{n+1}\varpi_i L_{p,l}^{k,q}(\zeta_r,\xi_s,\theta_i,\phi_j)\hat{\mathfrak{h}}_l^\alpha(\theta_i,\phi_j),\quad p=1,2.
\end{align*}

Obviously, the number of unknown coefficients $(M+1)^2+3\ll 2\tilde{n}^2$ . Due to the severely ill-posed nature of the problem, we solve the overdetermined system \eqref{discretization update equation} via Tikhonov regularization. Accordingly, the linear system \eqref{discretization update equation} is expressed as an optimization problem aimed at minimizing the Tikhonov functional \cite{Tikhonov}

{\small
\begin{equation}\label{minimizer}
	\begin{aligned}
		\sum_{s=0}^{2\tilde{n}-1}&\sum_{r=0}^{\tilde{n}-1}\Big|\sum_{p=1}^{3} \Delta c_p B_{p,l}(\zeta_r,\xi_s)+\sum_{k=1}^{M}\sum_{q=1}^{k}\alpha_{kq}B_{1,l}^{k,q}(\zeta_r,\xi_s)\\
		&+\sum_{k=0}^{M}\sum_{q=0}^{k}\beta_{kq}B_{2,l}^{k,q}(\zeta_r,\xi_s)-\hat{f}_{D,l}^{\rm sc}(\zeta_r,\xi_s)\Big| ^2\\
		&+\lambda \Bigg(|\Delta c_1|^2+|\Delta c_2|^2+|\Delta c_3|^2
		+\frac{1}{2}\sum_{k=1}^{M}\sum_{q=1}^{k}(1+l(l+1))^{\gamma}\alpha_{kq}^2+\beta_{00}^2\\&+\sum_{k=1}^{M}(1+l(l+1))^{\gamma}\beta_{k0}^2+\frac{1}{2}\sum_{k=1}^{M}\sum_{q=1}^{k}(1+l(l+1))^{\gamma}\beta_{kq}^2\Bigg),
	\end{aligned}
\end{equation}
}
where $\lambda$ is a positive regularization parameter and the last item is $H^{\gamma}$ penalty term. Through the least squares principle, the minimizer of \eqref{minimizer} is the solution of the following system 
\begin{equation*}
	\left( {\lambda}\tilde{\pmb{I}}+\Re\left( \pmb{B}_l^*\pmb{B}_l\right) \right) \pmb{\varUpsilon}=\Re\left( \pmb{B}_l^*\hat{\pmb{f}}_{D,l}^{\rm sc}\right) ,
\end{equation*}
where the system matrix $\pmb{B}_l$ is defined by
\begin{equation*}
	\pmb{B}_l=\left(\pmb{B}_{1,l},\pmb{B}_{2,l},\pmb{B}_{3,l},\pmb{B}_{1,l}^{1,1},\cdots,\pmb{B}_{1,l}^{M,M},\pmb{B}_{2,l}^{0,0},\cdots,\pmb{B}_{2,l}^{M,M}\right).
\end{equation*}
The column blocks of $\pmb{B}_l$ are given by
\begin{align*}
	\pmb B_{p,l}&=\left(B_{p,l}(\pmb\zeta,\xi_0),\cdots,B_{p,l}(\pmb\zeta,\xi_{2\tilde{n}-1})\right)^\top, \quad {p}=1,2,3,\\
	\pmb B_{1,l}^{k,q}&=\left(B_{1,l}^{k,q}(\pmb\zeta,\xi_0),\cdots,B_{1,l}^{k,q}(\pmb\zeta,\xi_{2\tilde{n}-1})\right)^\top,\quad k=1,\cdots,M,~ q=1,\cdots,k,\\
	\pmb B_{2,l}^{k,q}&=\left(B_{2,l}^{k,q}(\pmb\zeta,\xi_0),\cdots,B_{2,l}^{k,q}(\pmb\zeta,\xi_{2\tilde{n}-1})\right)^\top,\quad k=0,\cdots,M,~ q=0,\cdots,k,
\end{align*}
with
\begin{align*}
	B_{p,l}(\pmb{\zeta},\xi_s)=&(B_{p,l}(\zeta_0,\xi_s),\cdots,B_{p,l}(\zeta_{\tilde{n}-1},\xi_s)),\quad p=1,2,3,\\
	B_{p,l}^{k,q}(\pmb{\zeta},\xi_s)=&(B_{p,l}^{k,q}(\zeta_0,\xi_s),\cdots,B_{p,l}^{k,q}(\zeta_{\tilde{n}-1},\xi_s)),\quad p=1,2.
\end{align*}
The solution vector and regularization matrix are shown by
\begin{align*}
	&\pmb{\varUpsilon}=(\Delta c_1,\Delta c_2,\Delta c_3,\alpha_{11},\cdots,\alpha_{M1},\cdots,\alpha_{MM},\beta_{00},\cdots,\beta_{M0},\cdots,\beta_{MM})^\top,\\
	&\tilde{\pmb{I}}
	=\mathrm{diag}\big(
	\underbrace{1,1,1,}_{\Delta c_j,\;j=1,2,3}
	\underbrace{\vartheta_1,}_{\alpha_{11}}\;
	\underbrace{\vartheta_2,\vartheta_2}_{\alpha_{2j},j=1,2},\;\cdots,\;
	\underbrace{\vartheta_M,\cdots,\vartheta_M}_{\alpha_{Mj},j=1,\cdots,M},\\
	&\qquad\qquad\underbrace{1}_{\beta_{00}},\;
	\underbrace{2\vartheta_1,\vartheta_1}_{\beta_{1j},j=0,1},\;\cdots,\;
	\underbrace{2\vartheta_M,\cdots,\vartheta_M}_{\beta_{Mj},j=0,\cdots,M}
	\big),\\
	&\hat{\pmb{f}}_{D,l}^{\rm sc}=\left(\hat{f}_{D,l}^{\rm sc}(\pmb\zeta,\xi_0),\cdots,\hat{f}_{D,l}^{\rm sc}(\pmb\zeta,\xi_{2\tilde{n}-1})\right)^\top,
\end{align*}
where 
\begin{equation*}
	\begin{aligned}
&\vartheta_m=(1+m(m+1))^{\gamma}/2,\quad m=1, \cdots, M,\\
&\hat{f}_{D,l}^{\rm sc}(\pmb\zeta,\xi_s)=\left( \hat{f}_{D,l}^{\rm sc}(\zeta_0,\xi_s),\cdots,\hat{f}_{D,l}^{\rm sc}(\zeta_{\tilde{n}-1},\xi_s)\right).
\end{aligned}
\end{equation*}
Thus, we obtain the new approximation
\[
\pmb p_D^{new}=(\pmb{c}+\Delta \pmb{c})+(r(\hat{\pmb{x}})+\Delta r(\hat{\pmb{x}}))\hat{\pmb{x}},
\]
with $\pmb{c}=(c_1,c_2,c_3)^\top$ and $\Delta \pmb{c}=(\Delta c_1,\Delta c_2,\Delta c_3)^\top$.

\section{Numerical experiments}\label{num ex}
In this section, we first verify the accuracy and effectiveness of the proposed forward scheme for the time-domain scattering problem. 
We then present a series of numerical experiments for the inverse problem to confirm the effectiveness and robustness of the proposed reconstruction approach.

\subsection{Forward scattering problem}\label{4.1111}

To test the accuracy of the forward scheme proposed in Section \ref{direct scattering}, we construct a model problem whose exact solution is available in explicit form. Let
\begin{equation*}
	u_{true}^{\rm sc}(\pmb{x},t;\pmb{x}_0)
	=\left\{
	\begin{aligned}
	&\frac{\sin(0.3(t-|\pmb{x}-\pmb{x}_0|))e^{-(t-|\pmb{x}-\pmb{x}_0|-2)^2}}{|\pmb{x}-\pmb{x}_0|},
	&& t>|\pmb{x}-\pmb{x}_0|,\\[1.2ex]
	&0, && t\le |\pmb{x}-\pmb{x}_0|,
	\end{aligned}
	\right.
\end{equation*}
be the exact solution of \eqref{scattered field} with $\pmb{x}_0=(0.1,0,0)^\top$ locating inside the obstacle.
This function represents a smooth spherical wave emitted from a point source. For the Dirichlet boundary condition, the numerical density $\psi$ on $\Gamma_D$ over $[0,T]$ is obtained by solving the boundary integral equation \eqref{boundary integral equation} with the right hand enforced
\[
u^{\rm inc}(\pmb{x},t)=-u_{true}^{\rm sc}(\pmb{x},t;\pmb{x}_0),\quad \left(\pmb{x},t\right)\in\Gamma_D\times[0,T],
\]
and then the numerical solution $u^{\rm sc}$ can be calculated directly by \eqref{single layer}.

The time discretization is carried out using the convolution quadrature method based on the BDF3 scheme, while the spatial discretization of the boundary integral equations is performed by the Galerkin method using spherical harmonics. We evaluate the forward scheme on several obstacle geometries, including sphere-shaped, cushion-shaped, pinched ball-shaped and bean-shaped obstacles, whose boundary parametrizations are given in Table \ref{obscatle1}.

\begin{table}[!htpb]
	\caption{}\label{obscatle1}\centering
	\resizebox{\textwidth}{!}{
	\begin{tabular}{lll}
		\toprule  
		Type           & Parametrization of the boundary surfaces \\
		\midrule
		sphere & $\pmb p_D(\theta, \phi) = 0.6 \, \hat{\pmb{x}}(\theta,\phi)$ \\
~\\
		pinched ball& $\pmb p_D(\theta, \phi)=\sqrt{0.3+0.12 \cos 2 \phi(\cos 2 \theta-1)} \, \hat{\pmb{x}}(\theta,\phi)$ \\
~\\
		cushion& $\pmb p_D(\theta, \phi)=\sqrt{0.3+0.1(\cos 2 \phi-1)(\cos 4 \theta-1)} \, \hat{\pmb{x}}(\theta,\phi)$ \\
~\\
		bean& 
		$\pmb{p}_D(\theta,\phi) =
		\begin{pmatrix}
			0.7\sqrt{1-0.1\cos(\pi\cos\theta)} \sin\theta\cos\phi \\[1.5ex]
			0.7\sqrt{1-0.4\cos(\pi\cos\theta)} \sin\theta\sin\phi + 0.21\cos(\pi\cos\theta)\\[1.5ex]
			0.7\cos\theta
		\end{pmatrix}$ \\
		\bottomrule 
	\end{tabular}
	}
\end{table}

The numerical scattered field $u^{\rm sc}$ is collected on a spherical surface $\Gamma_B$ centered at the origin with radius $R = 1.2$, using 200 observation points equally distributed in both the polar angle $\theta$ and the azimuthal angle $\phi$. The time interval is set to $[0, T]$ with $T = 6$, discretized into $N+1$ temporal nodes using a time step $\Delta t=T/N$. To demonstrate the convergence of the numerical algorithm, the relative $L^2$ error is defined as follows:
\begin{equation*}
	\epsilon_{\rm rel}=\frac{\left\| u_{true}^{\rm sc}-u^{\rm sc}\right\| _{L^2([0,T];L^2(\Gamma_B))}}{\left\| u_{true}^{\rm sc}\right\|_{L^2([0,T];L^2(\Gamma_B))} }\approx\frac{\left( \sum\limits_{i=0}^{N}\sum\limits_{j=1}^{200}\left| u_{true}^{\rm sc}(\pmb{x}_j,t_i)-u^{\rm sc}(\pmb{x}_j,t_i)\right| ^2\right)^{1/2} }{\left( \sum\limits_{i=0}^{N}\sum\limits_{j=1}^{200}\left| u_{true}^{\rm sc}(\pmb{x}_j,t_i)\right| ^2\right)^{1/2} }.
\end{equation*}

\medskip
{\noindent\bf Example: Accuracy and convergence of the forward scheme.}

\medskip
\noindent\emph{(1) Approximation of the scattered field.}
We first evaluate the ability of the proposed Galerkin-CQ scheme to approximate the exact scattered field.  
To this end, we consider the cushion-shaped obstacle and compare the numerical solution $u^{\rm sc}$ with the exact solution $u_{true}^{\rm sc}$ at the fixed observation point
\(
\pmb{x}_{\star} = (1.2,\,0,\,0)^\top,
\)
over the entire time interval $[0,T]$ with $T=6$.  
Figure~\ref{ex1.1} (a) shows $u^{\rm sc}(\pmb{x}_{\star},t)$ and $u_{true}^{\rm sc}(\pmb{x}_{\star},t)$ at different time nodes, which demonstrates the forward scheme is efficient for calculating the scattered field.

\medskip
\noindent\emph{(2) Influence of the time step.}
Next, we investigate the influence of time discretization while fixing the spatial discretization on $\Gamma_D$.  We consider the cushion-shaped obstacle and evaluate the scattered field at the observation point 
\(
\pmb{x}_{\star} = (1.2,\,0,\,0)^\top\), \(t_{\star} = 4.2.
\)
Figure~\ref{ex1.1} (b) displays the numerical values of $u^{\rm sc}(\pmb{x}_{\star},t_{\star})$ for decreasing time step $\Delta t=T/N$.  
The result shows that the numerical solution converges toward the exact scattered field as $\Delta t$ decreases ($N$ increases).
It verifies that smaller time steps yield better approximation of the exact solution.

\medskip
\noindent\emph{(3) Convergence in the space-time norm.}
Finally, we investigate the convergence behavior using the relative error in the $L^2([0,T];L^2(\Gamma_B))$ norm introduced earlier. The relative errors for the sphere-shaped, pinched ball-shaped, cushion-shaped and bean-shaped obstacle are summarized in Table~\ref{ex1.2} for varying time steps $\Delta t$ and spatial nodes $N_{\text{node}}$ on $\Gamma_D$. The error consistently decreases as the number of temporal and spatial nodes increases, confirming the stability and convergence of the Galerkin-CQ scheme.

The numerical results reported below demonstrate that the proposed Galerkin-CQ method produces stable and accurate solutions for the forward scattering problem, ensuring the reliability of the data used in the subsequent inverse reconstructions.

\begin{figure}[!htpb]
    \centering
    \subfigure[ Comparison between the numerical and exact scattered fields at $\pmb{x}_{\star}=(1.2,0,0)^{\top}$ over \texorpdfstring{$t\in[0,6]$}{t\in[0,6]}, $N=60$.]{
        \includegraphics[width=0.46\textwidth]{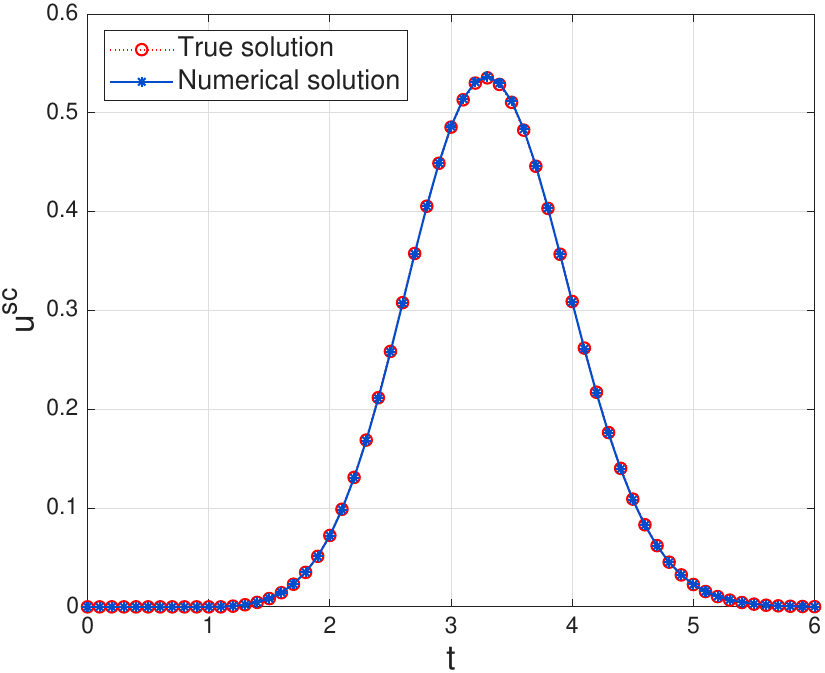}
    }
    \hfill
	\subfigure[ Convergence of $u^{\rm sc}(\pmb{x}_{\star},t_{\star})$ at $t_{\star}=4.2$ with respect to the time step $\Delta t=T/N$.]{
        \includegraphics[width=0.475\textwidth]{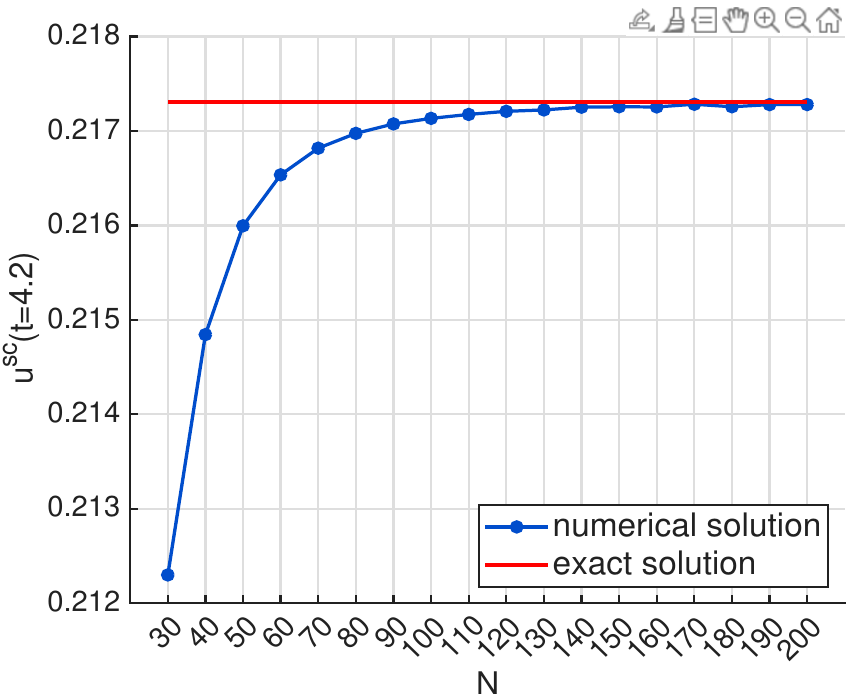}
    }
    \caption{Verification of the forward scheme using the cushion-shaped obstacle.}
    \label{ex1.1}
\end{figure}

\begin{table}[!htpb]
\centering
\caption{Relative error $\epsilon_{\rm rel}$ for different obstacles with various time steps.}\label{ex1.2}
\begin{tabular}{cccccc}
\toprule
time step & $N_{\text{node}}$ & Sphere & Pinched ball & Cushion & Bean \\
\midrule
\multirow{2}{*}{0.1} 
 & 200 & 2.12e-3 & 1.97e-3 & 1.80e-3 & 1.82e-3 \\
 & 800 & 1.91e-3 & 2.05e-3 & 1.69e-3 & 1.55e-3 \\
\midrule
\multirow{3}{*}{0.05} 
 & 200 & 1.45e-3 & 5.81e-4 & 1.13e-3 & 1.70e-3 \\
 & 800 & 4.18e-4 & 3.44e-4 & 3.66e-4 & 3.68e-4 \\
 & 1800 & 2.88e-4 & 2.86e-4 & 2.51e-4 & 2.29e-4 \\
\bottomrule
\end{tabular}
\end{table}

\begin{figure}[h]
	\centering 
	\subfigure[True shape with multiple views.]{
		\begin{tabular}{cccc}
			\includegraphics[width=0.21\textwidth]{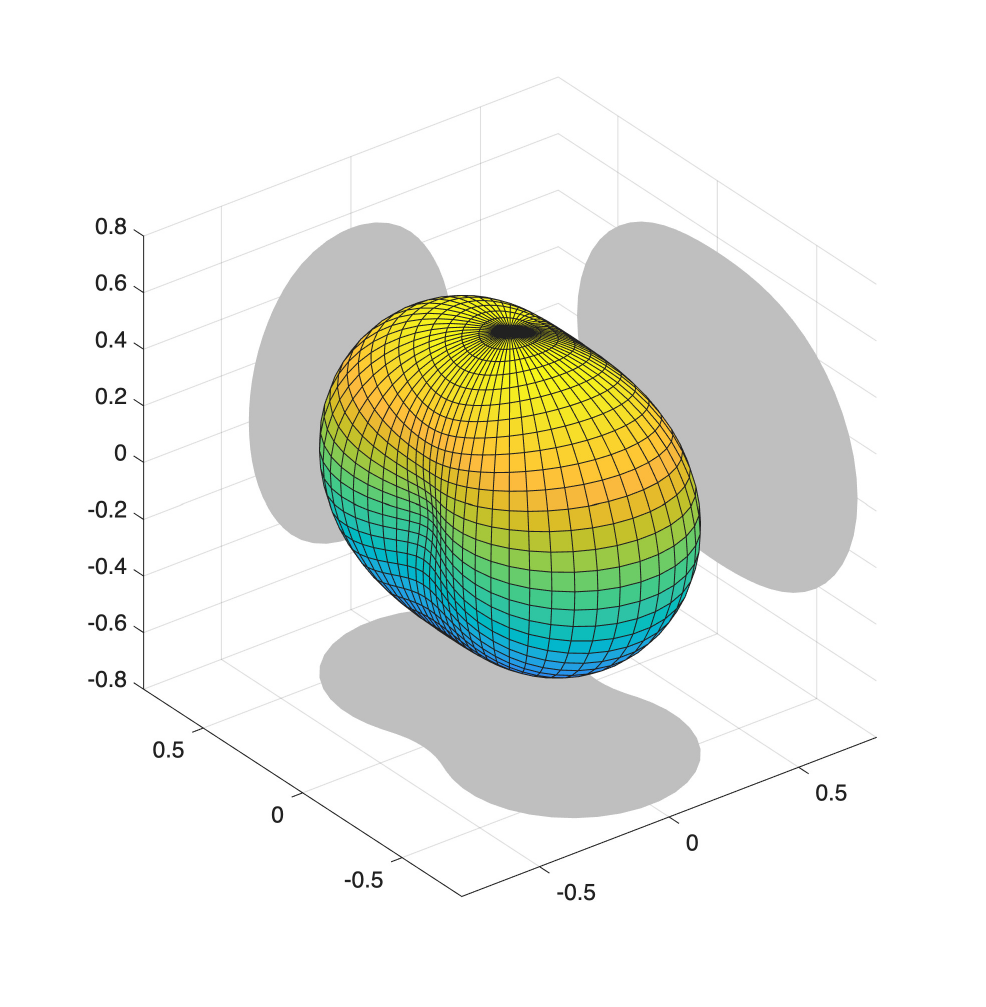} &
			\includegraphics[width=0.21\textwidth]{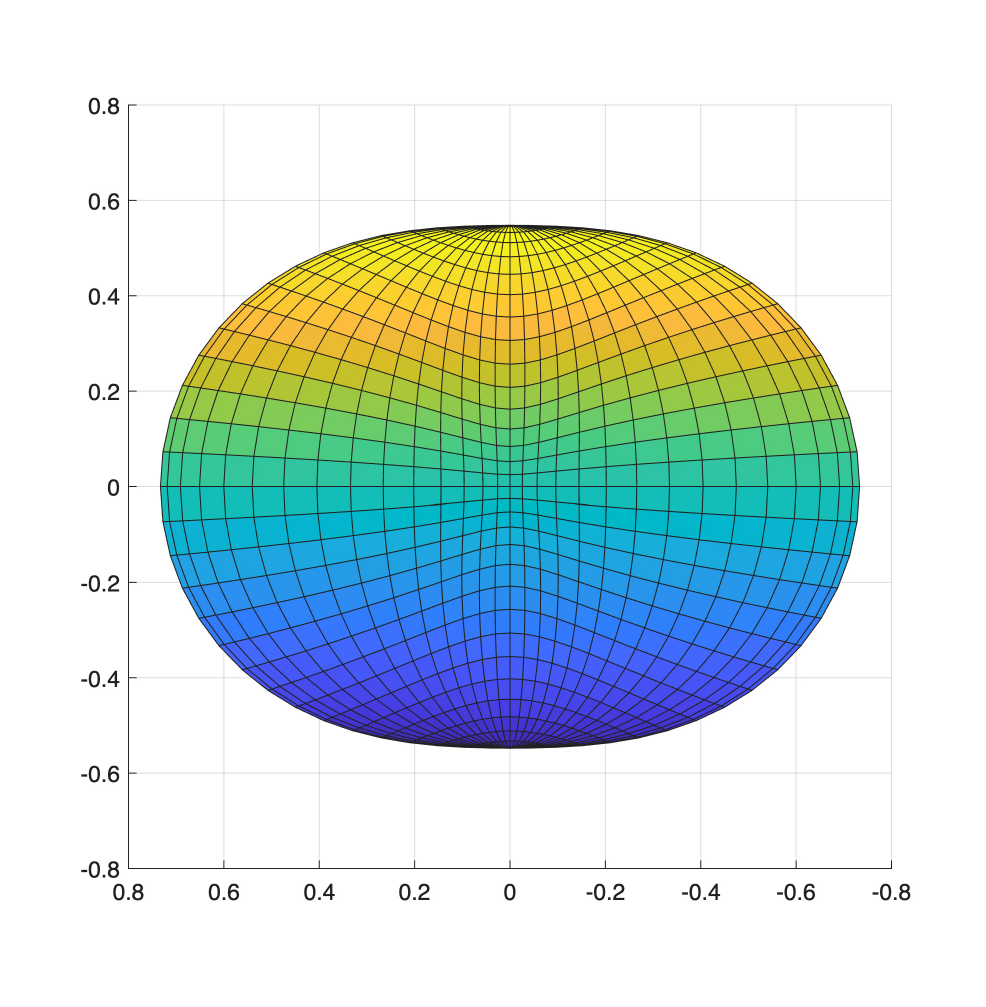} &
			\includegraphics[width=0.21\textwidth]{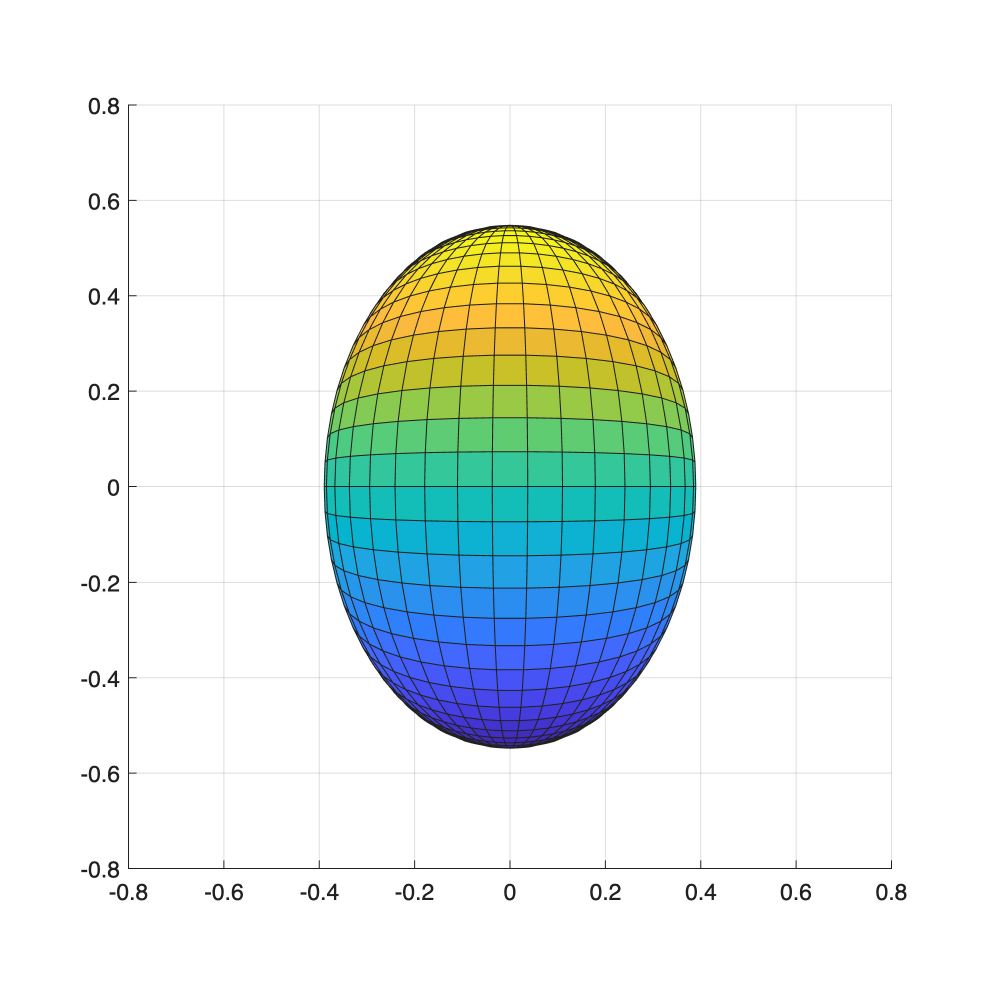} &
			\includegraphics[width=0.21\textwidth]{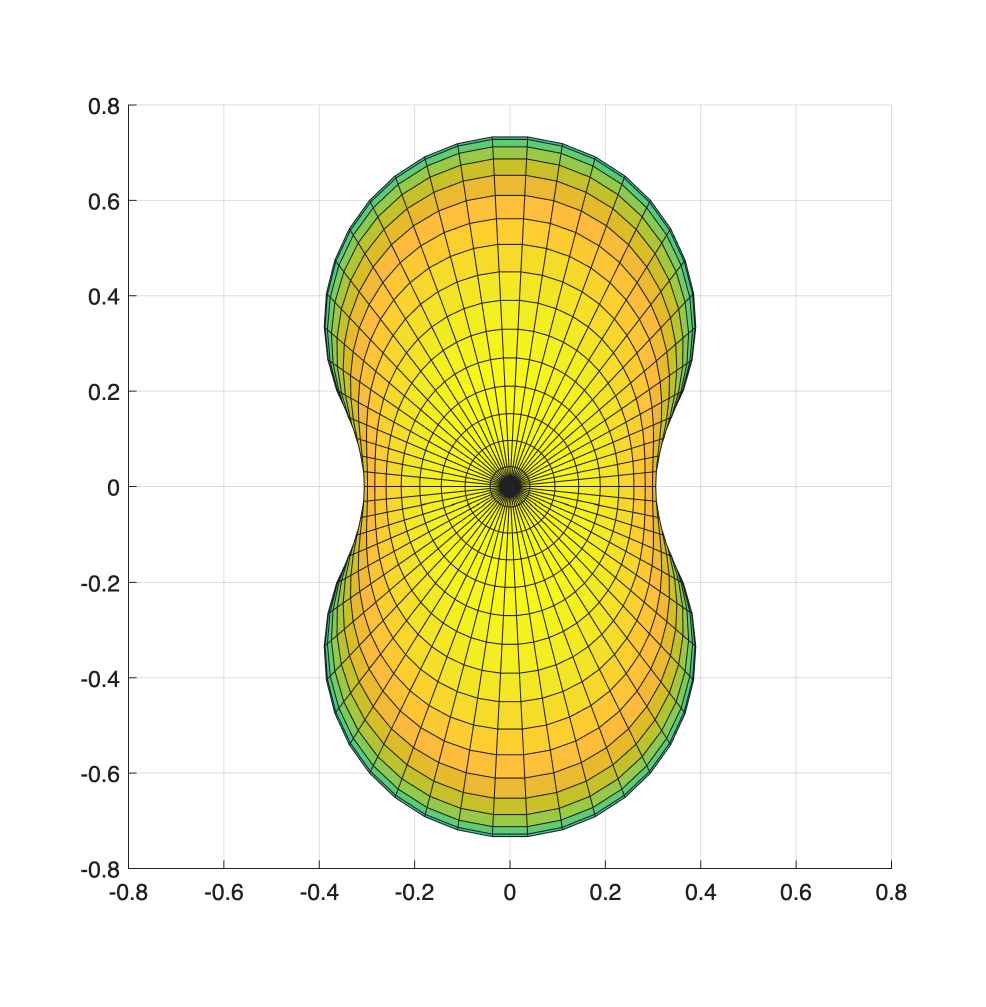}
		\end{tabular}
	} \\
	
	\subfigure[Reconstruction with 1$\%$ noise.]{
		\begin{tabular}{cccc}
			{\includegraphics[width=0.21\textwidth]{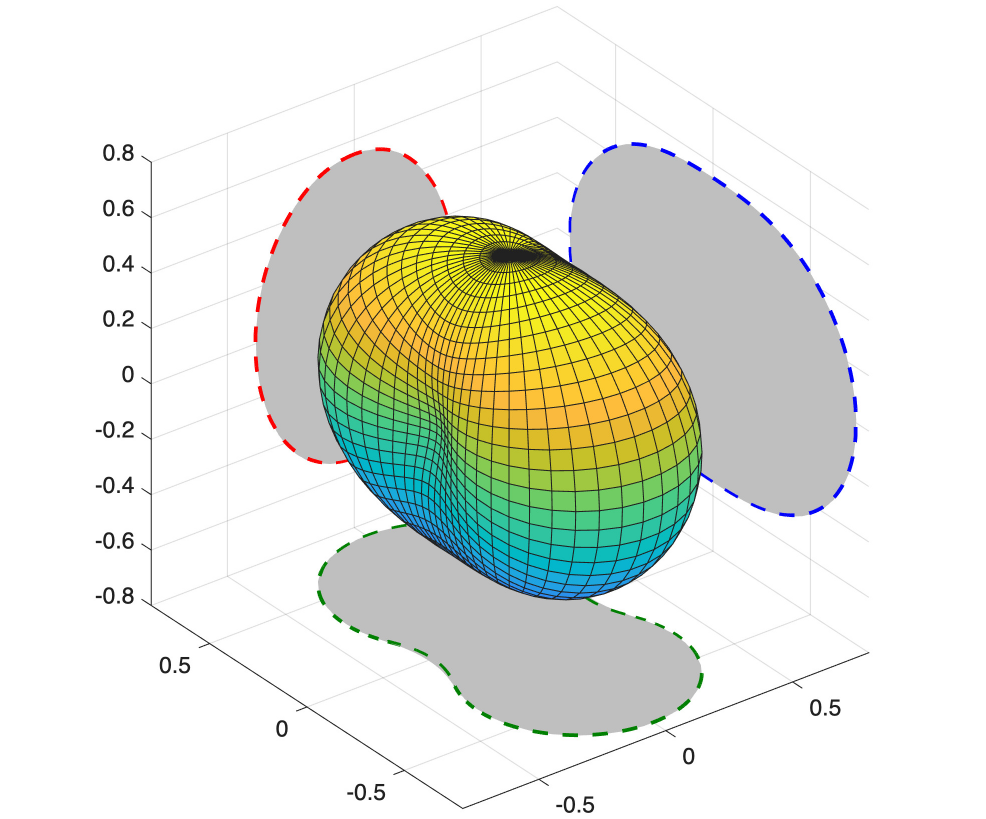}} &
			{\includegraphics[width=0.21\textwidth]{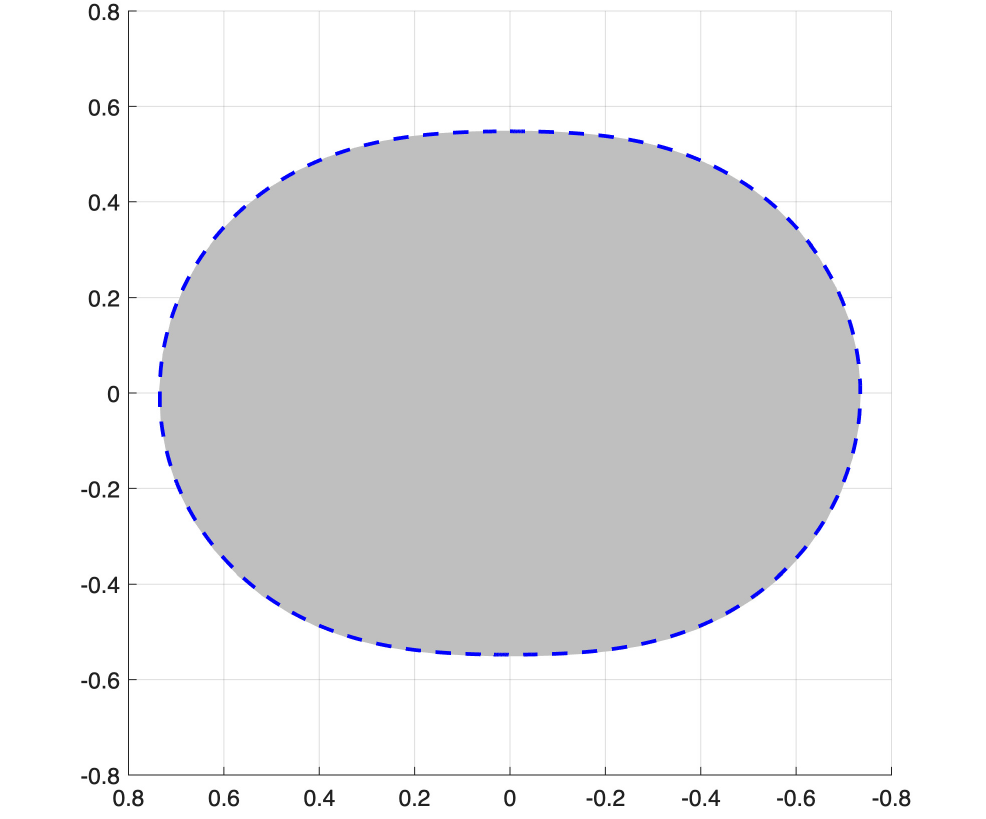}} &
			{\includegraphics[width=0.21\textwidth]{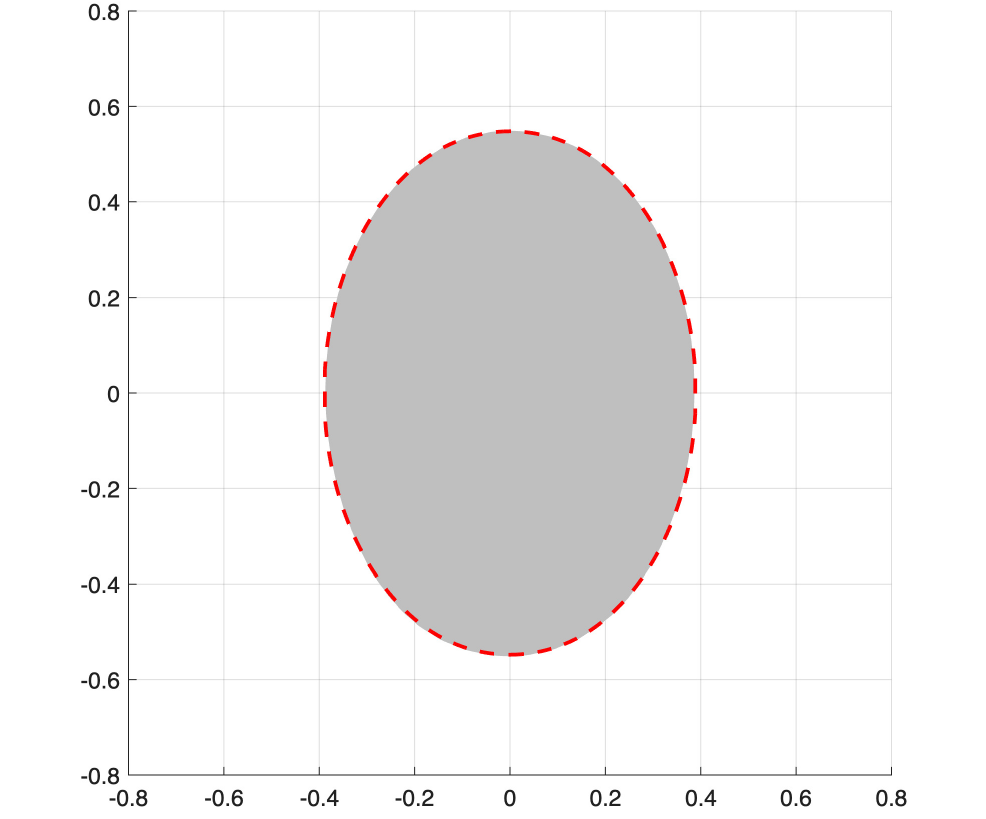}} &
			{\includegraphics[width=0.21\textwidth]{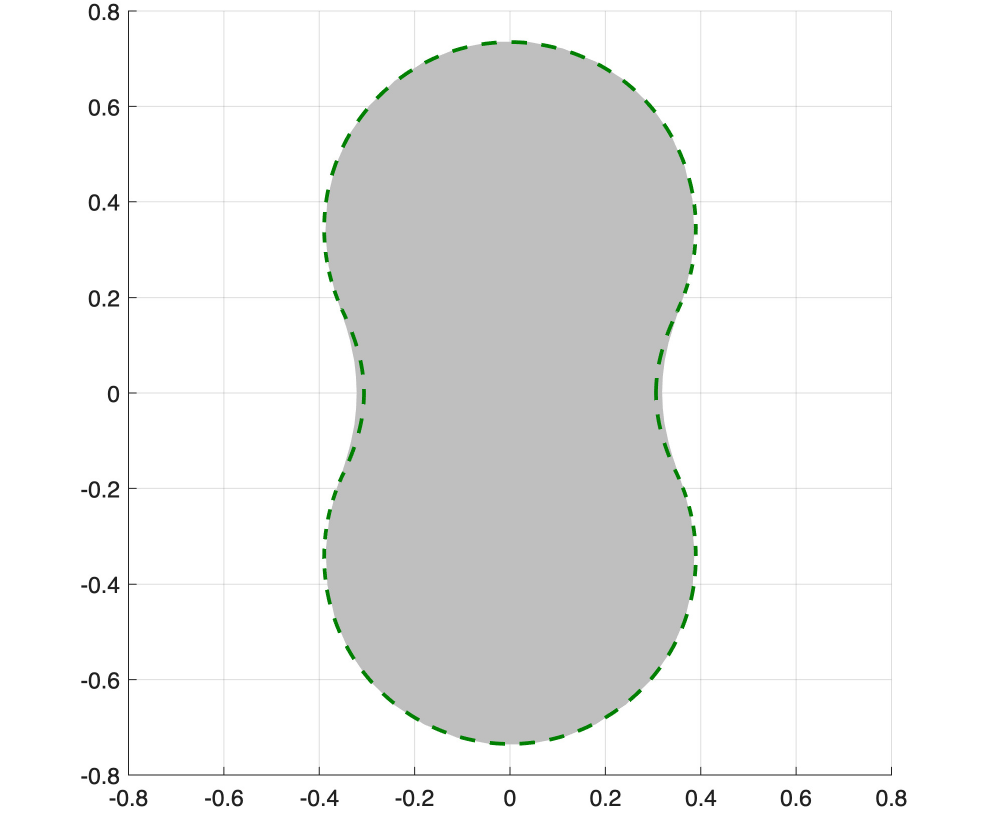}}
		\end{tabular} 
	}\\
	
	\subfigure[Reconstruction with 5$\%$ noise.]{
		\begin{tabular}{cccc}
			{\includegraphics[width=0.21\textwidth]{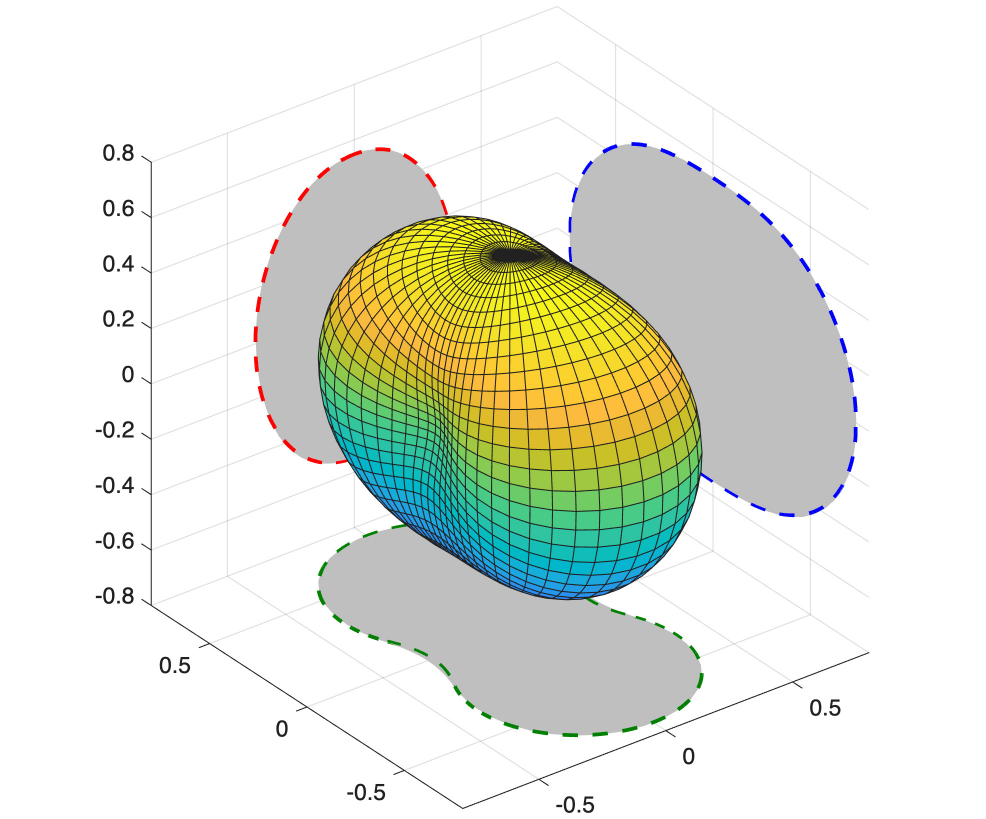}} &
			{\includegraphics[width=0.21\textwidth]{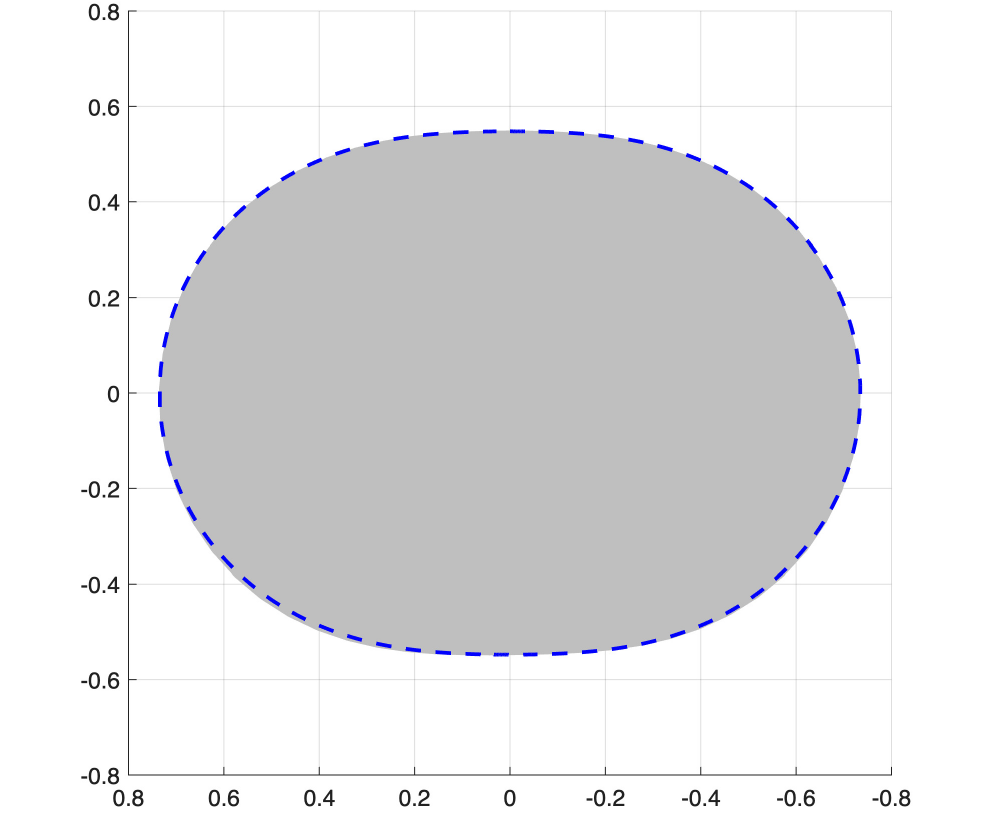}} &
			{\includegraphics[width=0.21\textwidth]{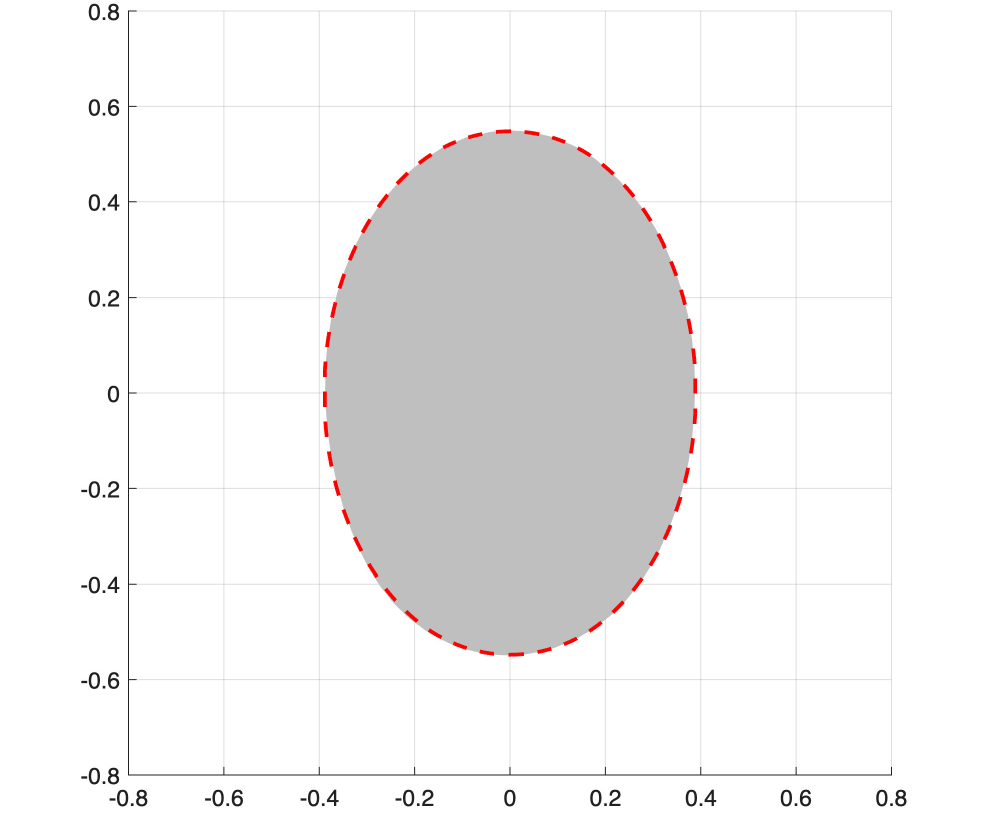}} &
			{\includegraphics[width=0.21\textwidth]{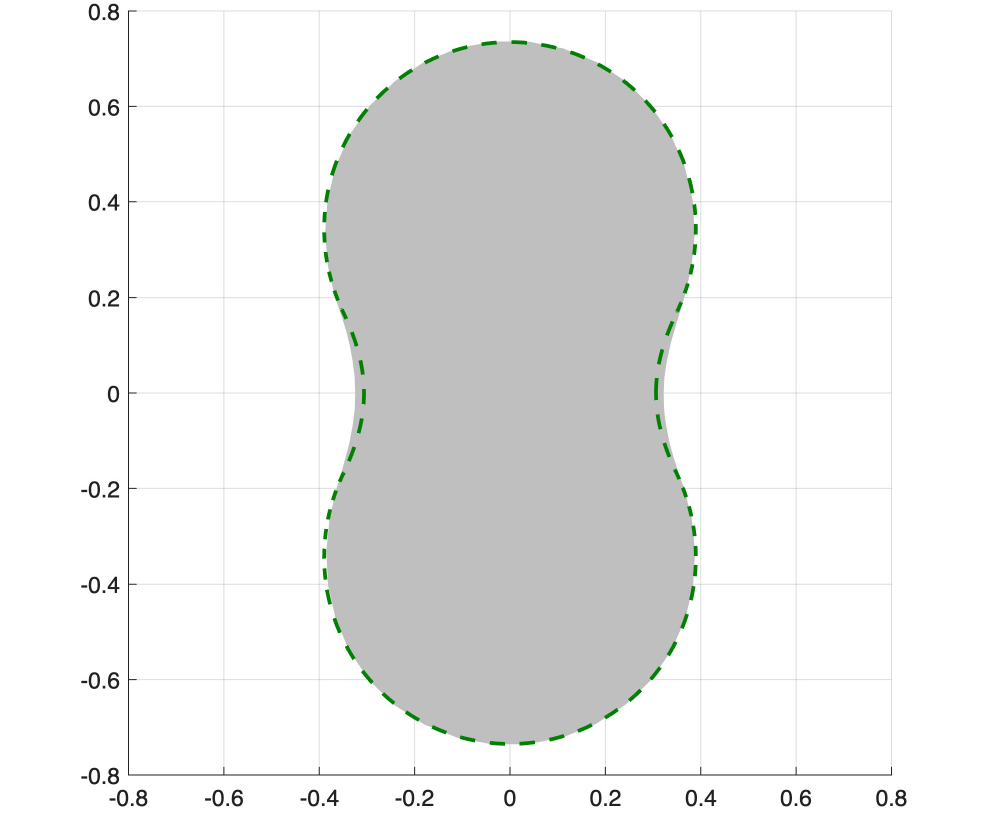}}
		\end{tabular} 
	}\\
	
	\subfigure[Reconstruction with 10$\%$ noise.]{
		\begin{tabular}{cccc}
			{\includegraphics[width=0.21\textwidth]{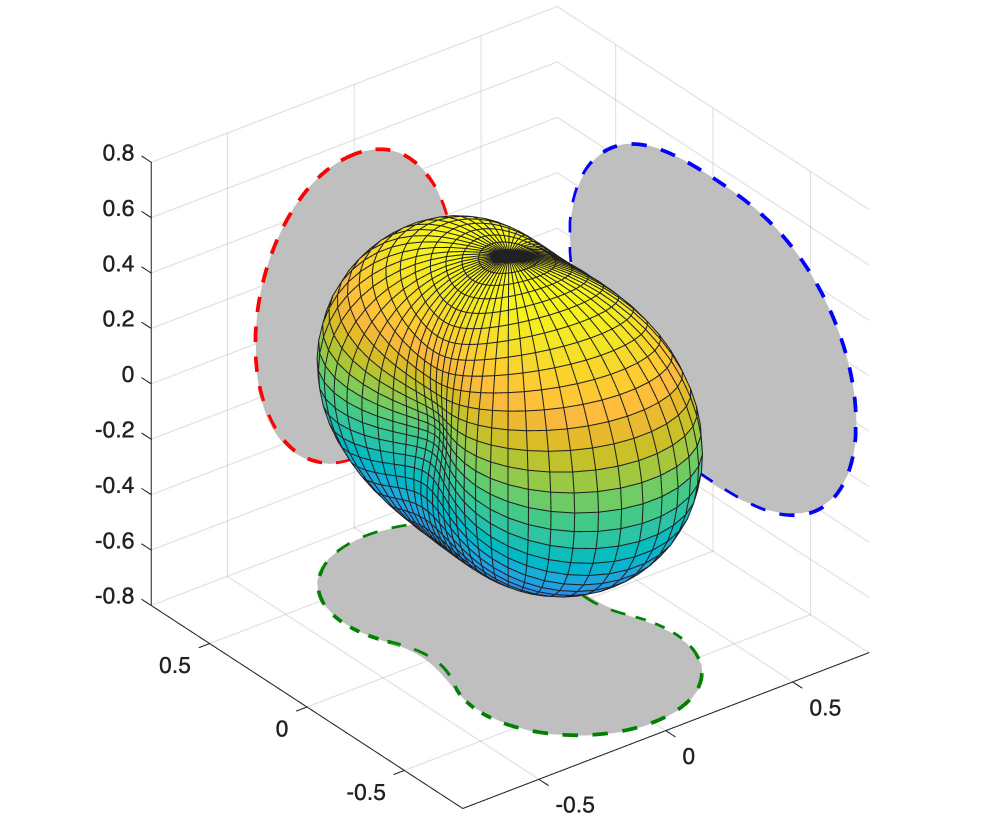}} &
			{\includegraphics[width=0.21\textwidth]{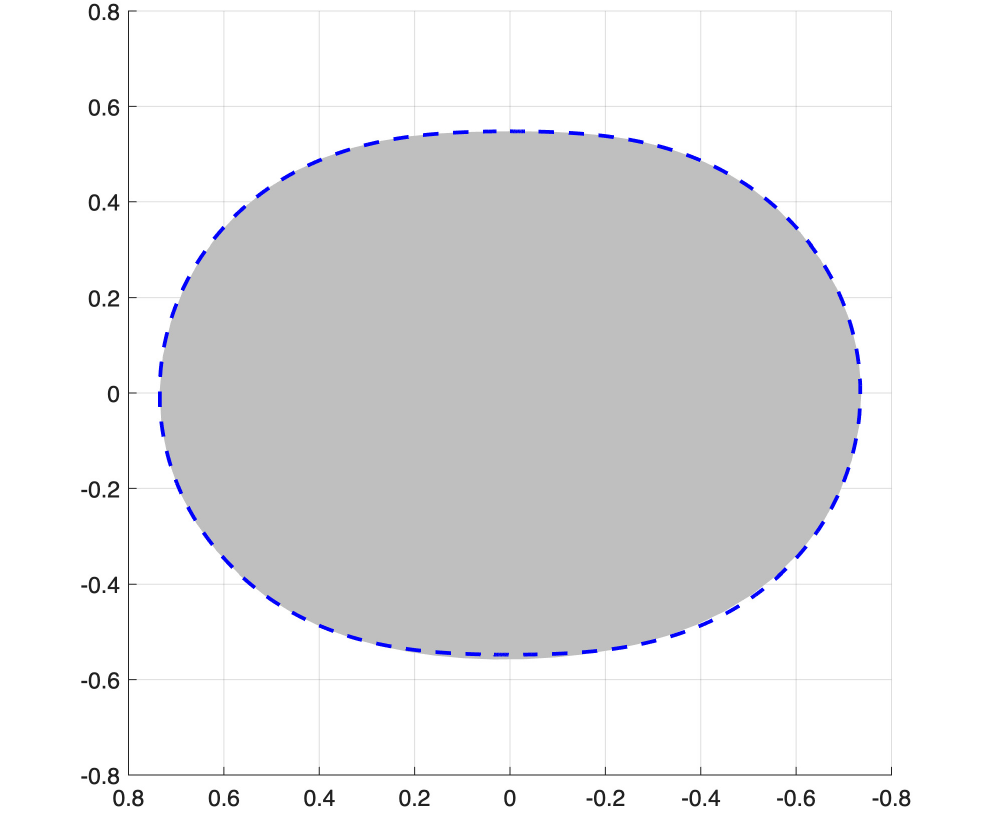}} &
			{\includegraphics[width=0.21\textwidth]{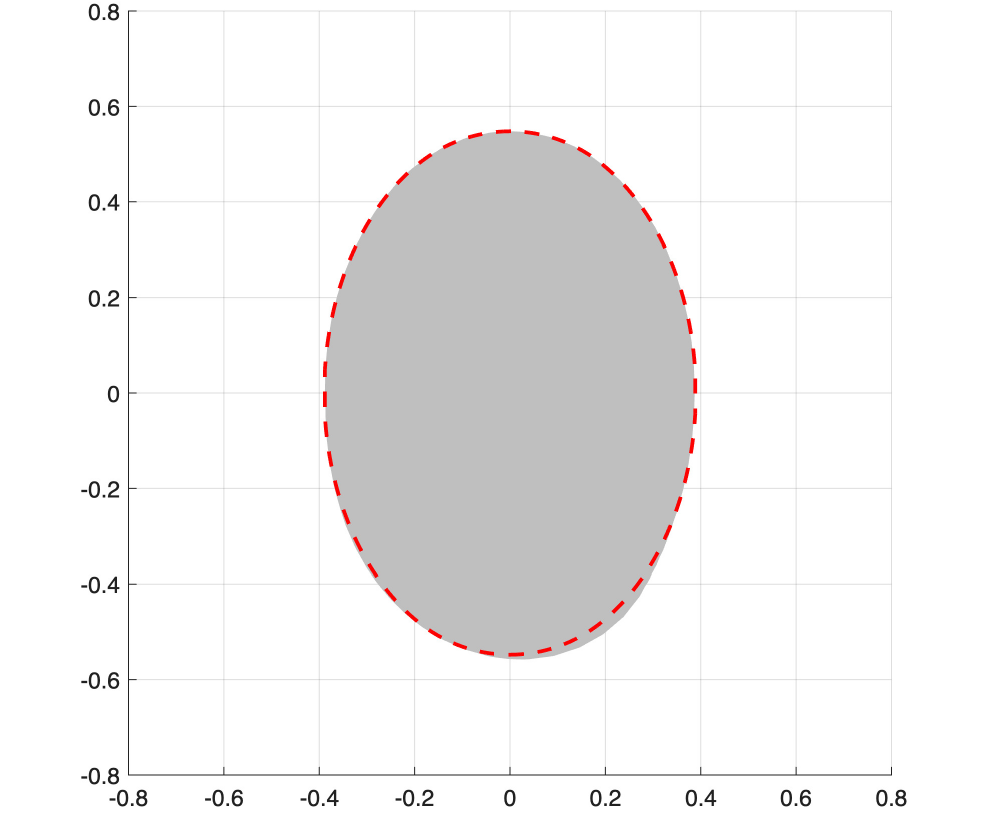}} &
			{\includegraphics[width=0.21\textwidth]{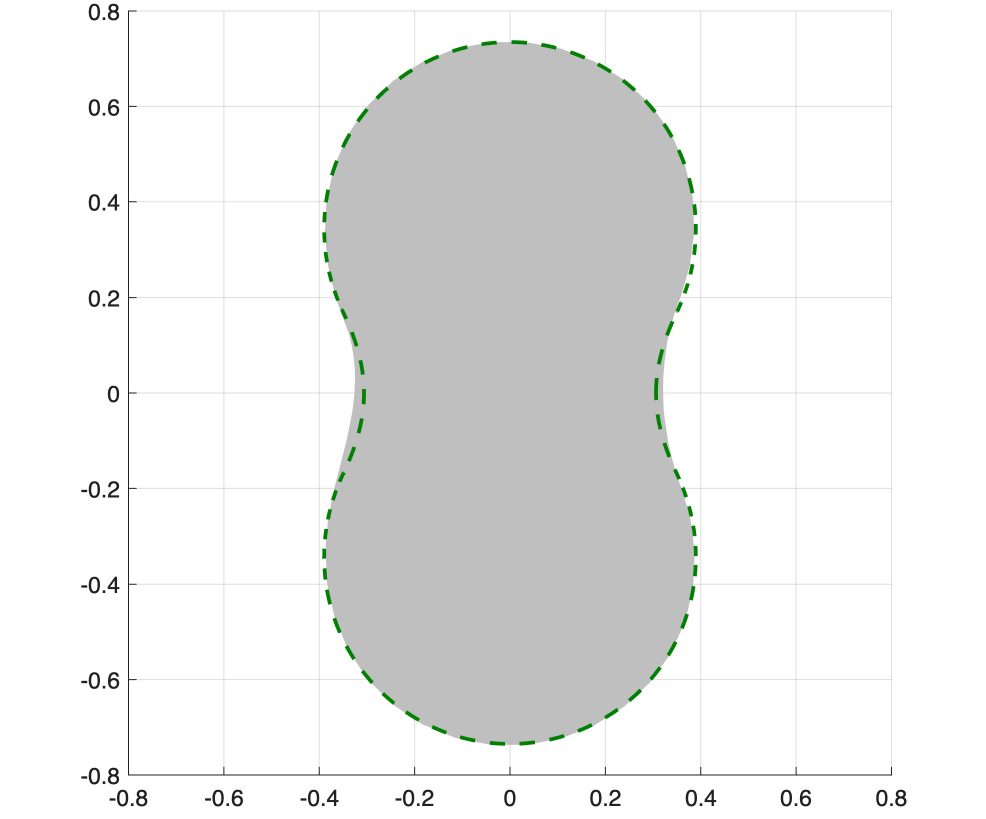}}
		\end{tabular} 
	}\\
	
	\caption{Reconstructions of a pinched ball-shaped obstacle with different levels of noise. 
		The initial guess is a sphere with  $\pmb c^{(0)} = (-0.5, 0.4, -0.3)^\top$ and  $r^{(0)} = 0.6$. The incident point source is located at $(0, 0, 5)^\top$.}\label{fig_pinchedball_ex1}
\end{figure}

\begin{figure}[h]
	\centering 
	
	\subfigure[True shape with multiple views.]{
		\begin{tabular}{cccc}
			\includegraphics[width=0.21\textwidth]{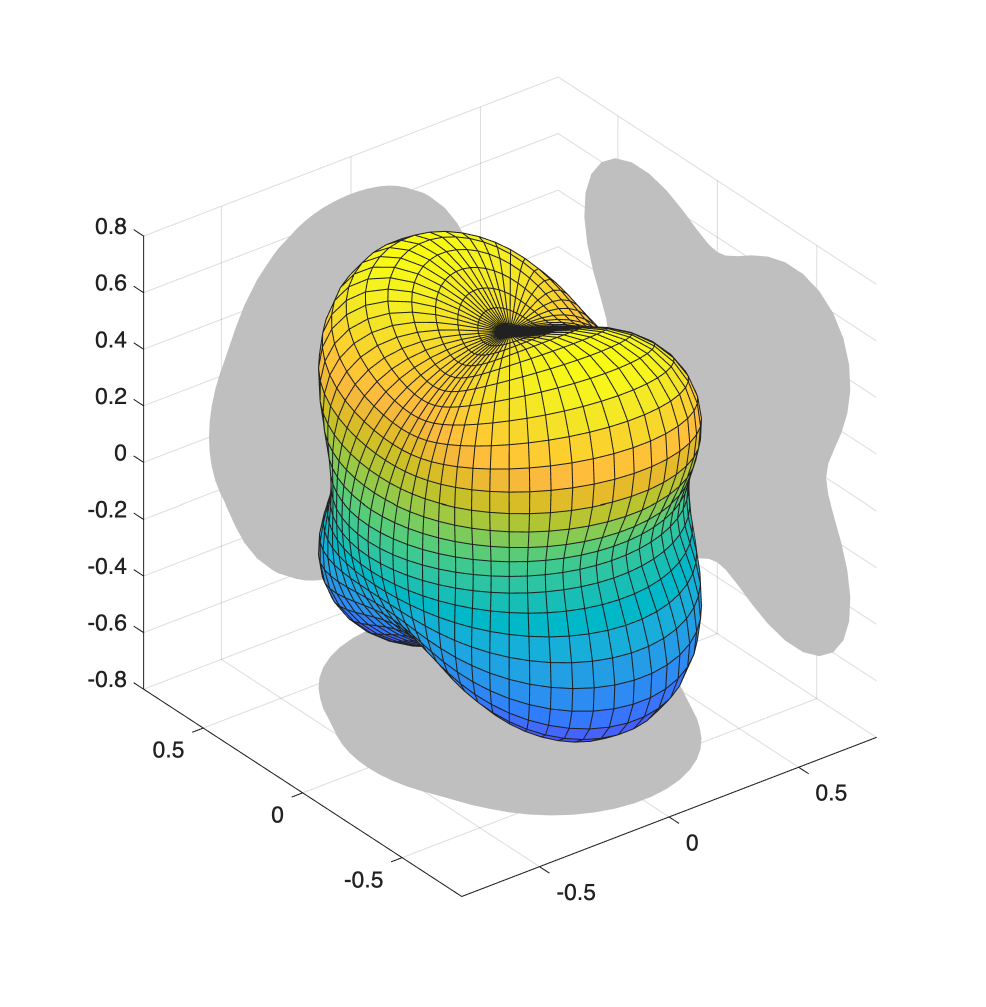} &
			\includegraphics[width=0.21\textwidth]{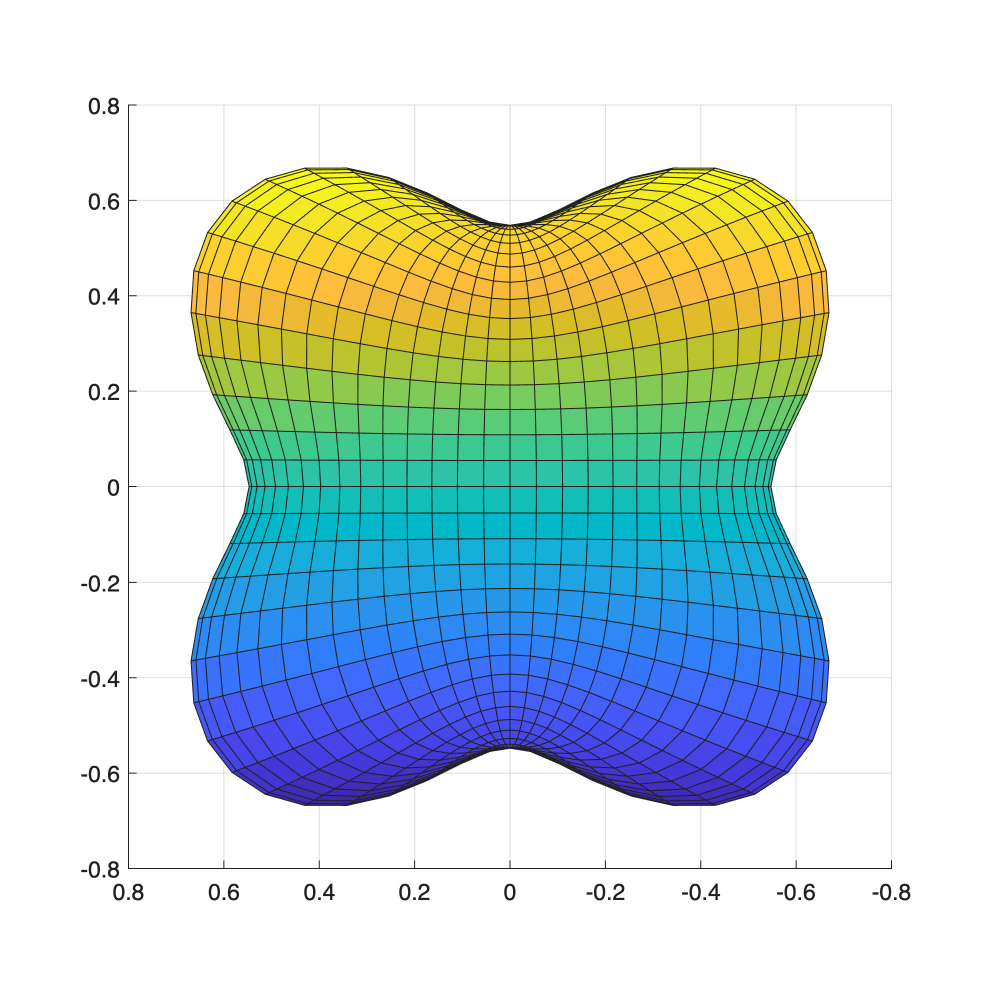} &
			\includegraphics[width=0.21\textwidth]{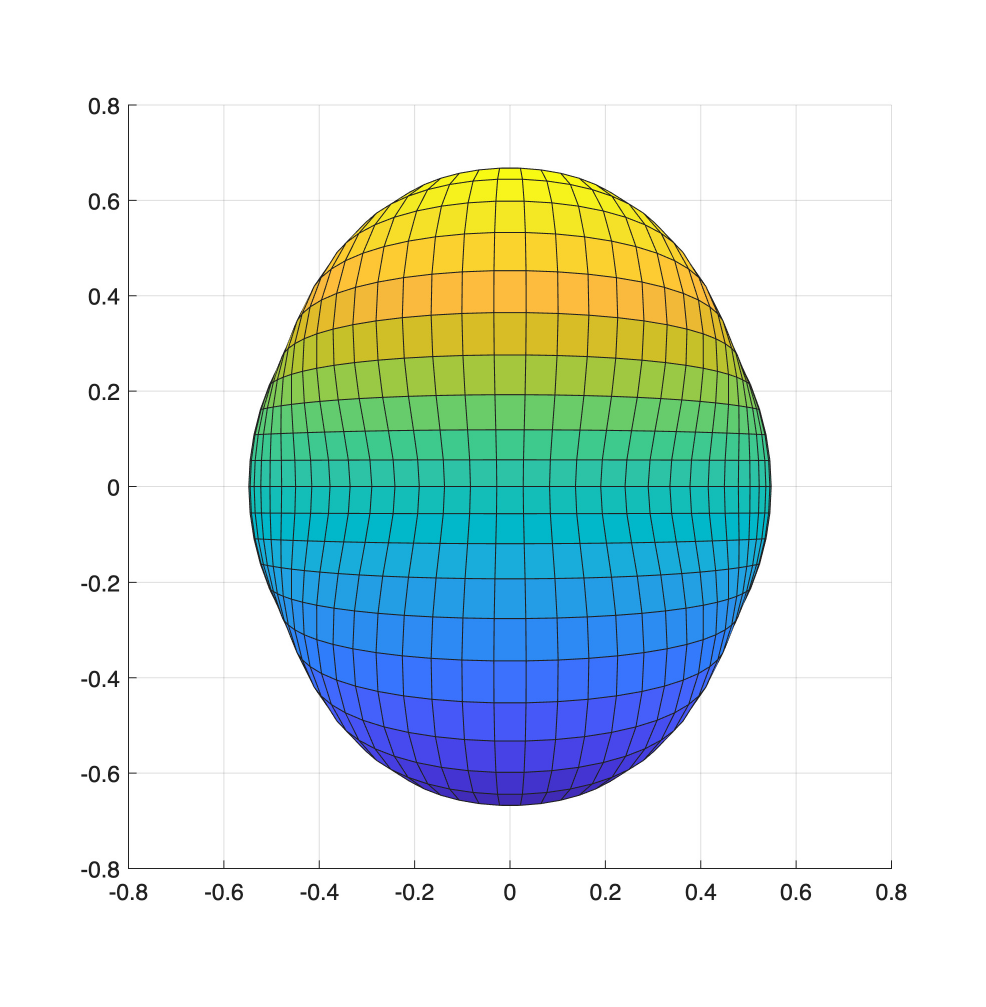} &
			\includegraphics[width=0.21\textwidth]{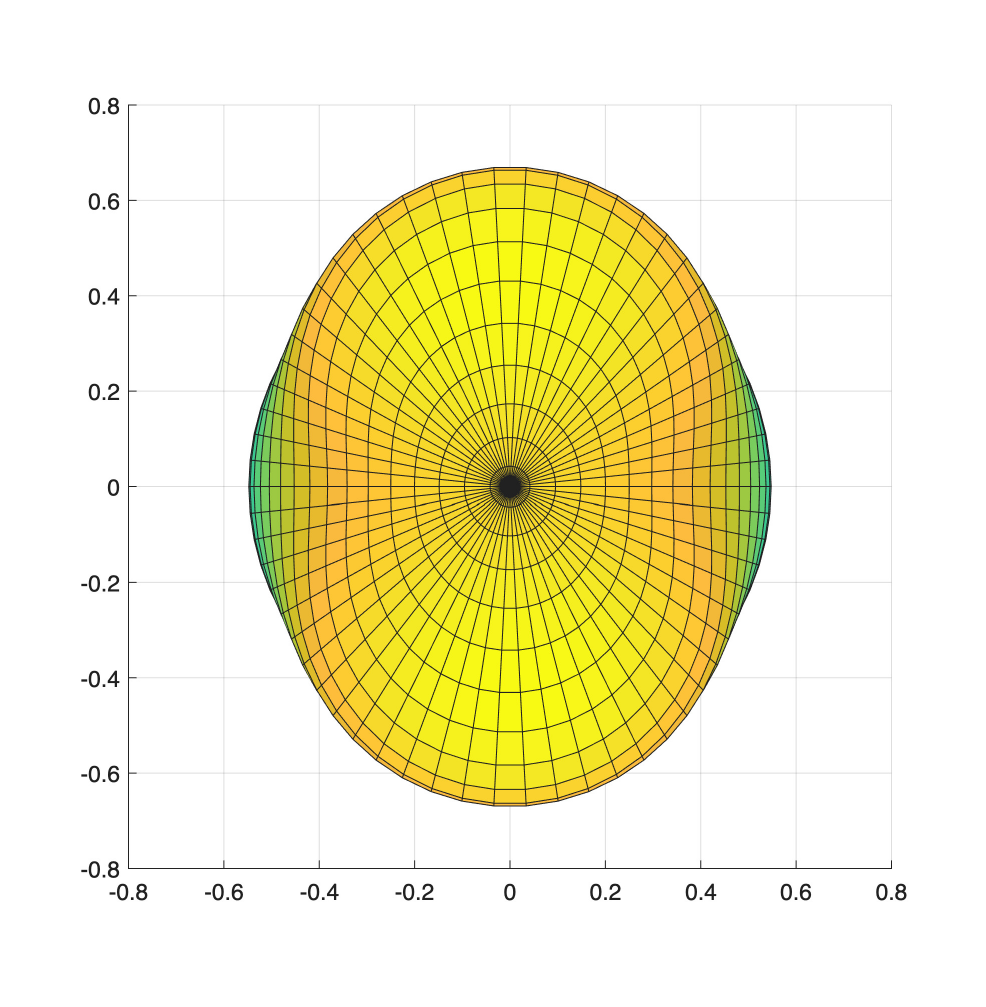}
		\end{tabular}
	} \\
	
	\subfigure[Reconstruction with 1$\%$ noise.]{
		\begin{tabular}{cccc}
			{\includegraphics[width=0.21\textwidth]{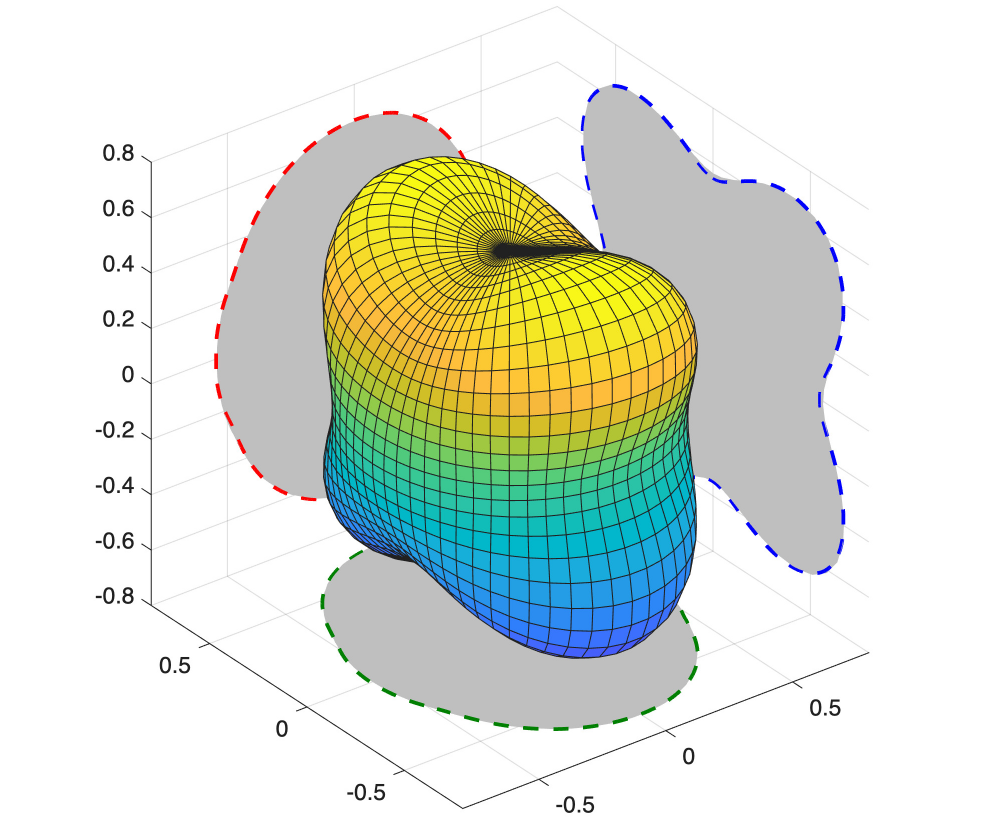}}  &
			{\includegraphics[width=0.21\textwidth]{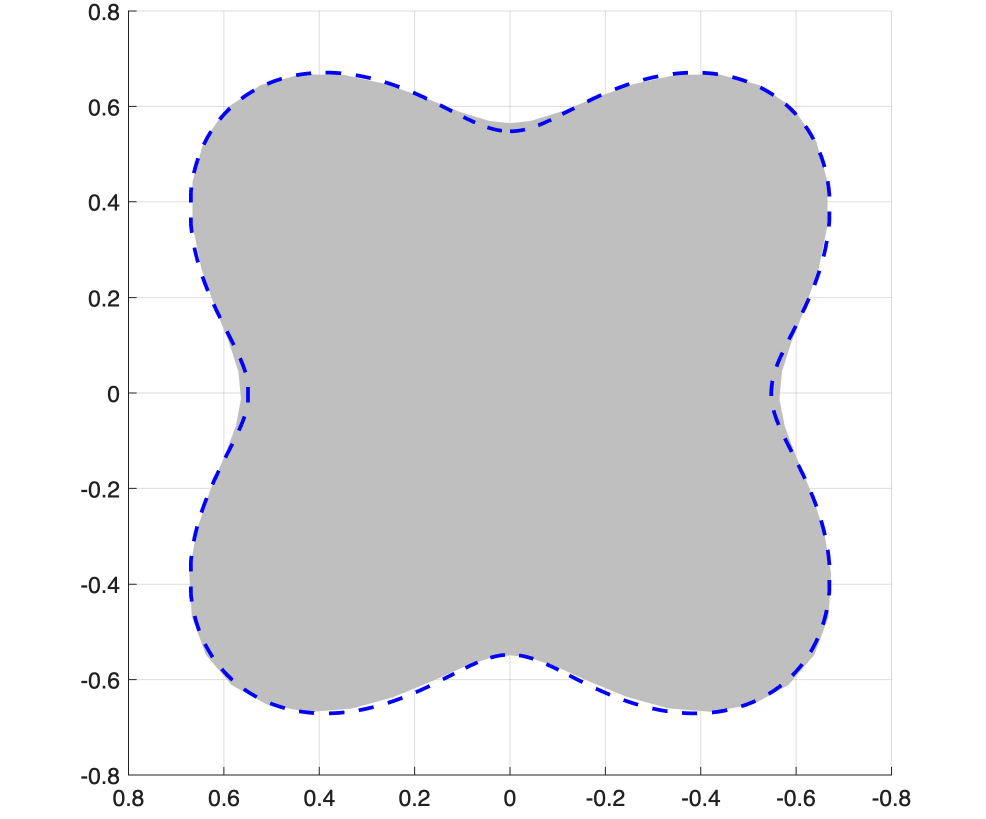}} &
			{\includegraphics[width=0.21\textwidth]{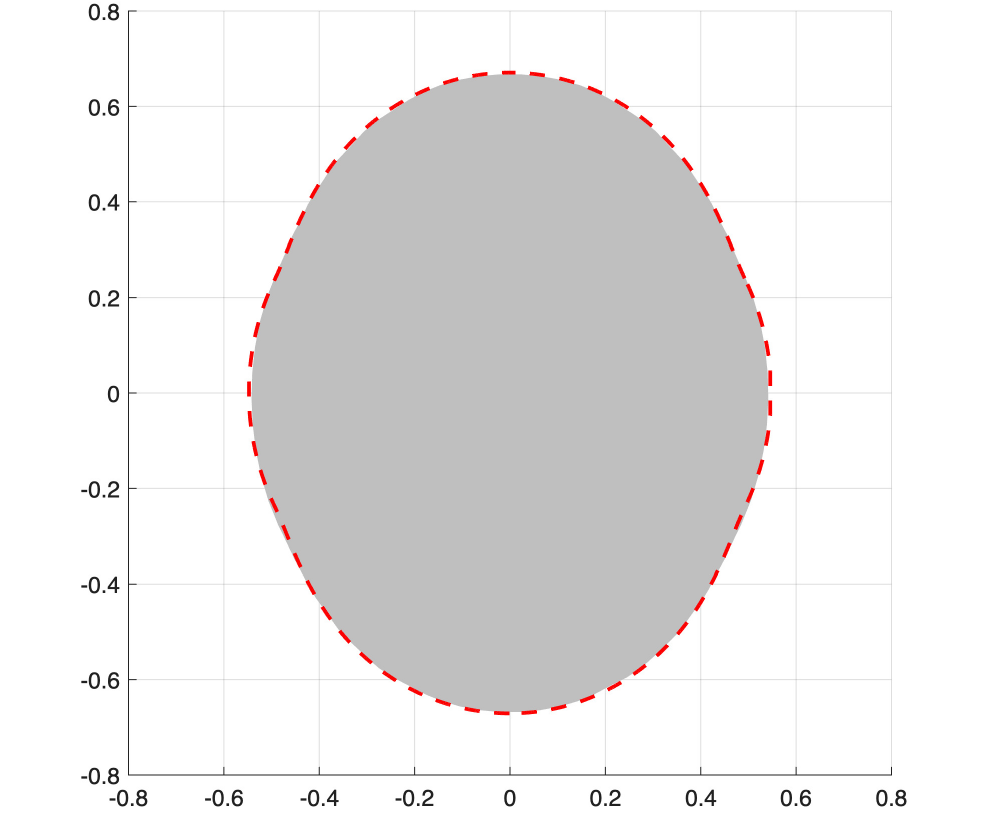}}  &
			{\includegraphics[width=0.21\textwidth]{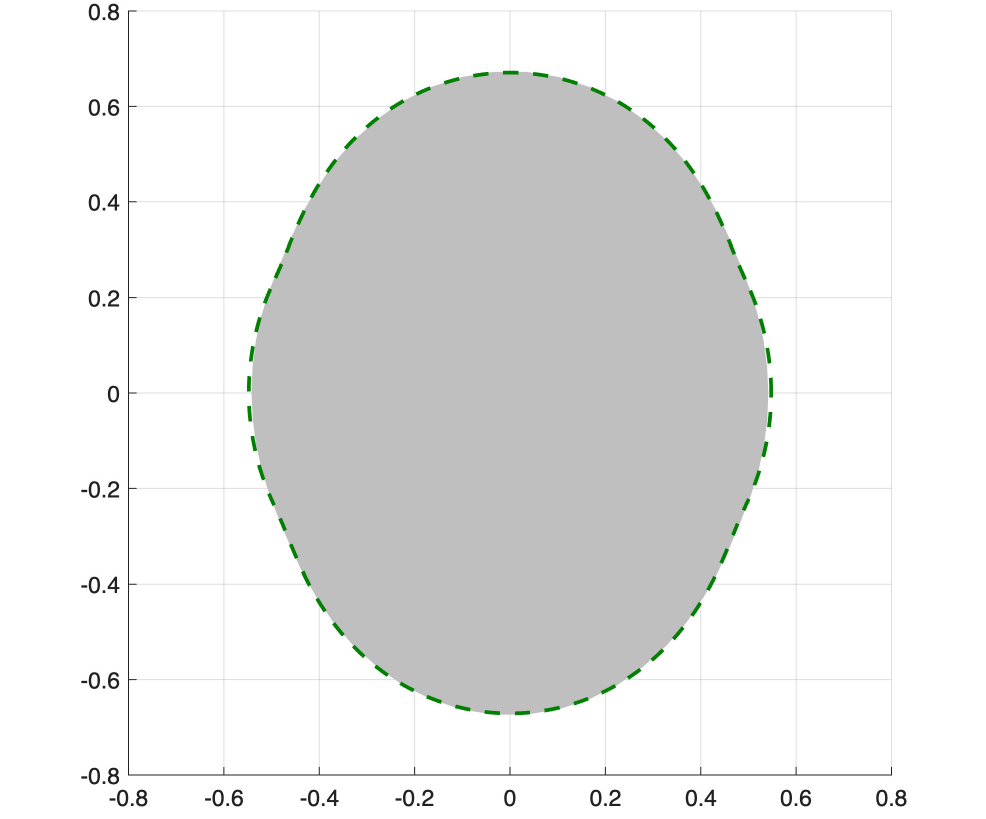}}
		\end{tabular} 
	}\\

	\subfigure[Reconstruction with 5$\%$ noise.]{
		\begin{tabular}{cccc}
			{\includegraphics[width=0.21\textwidth]{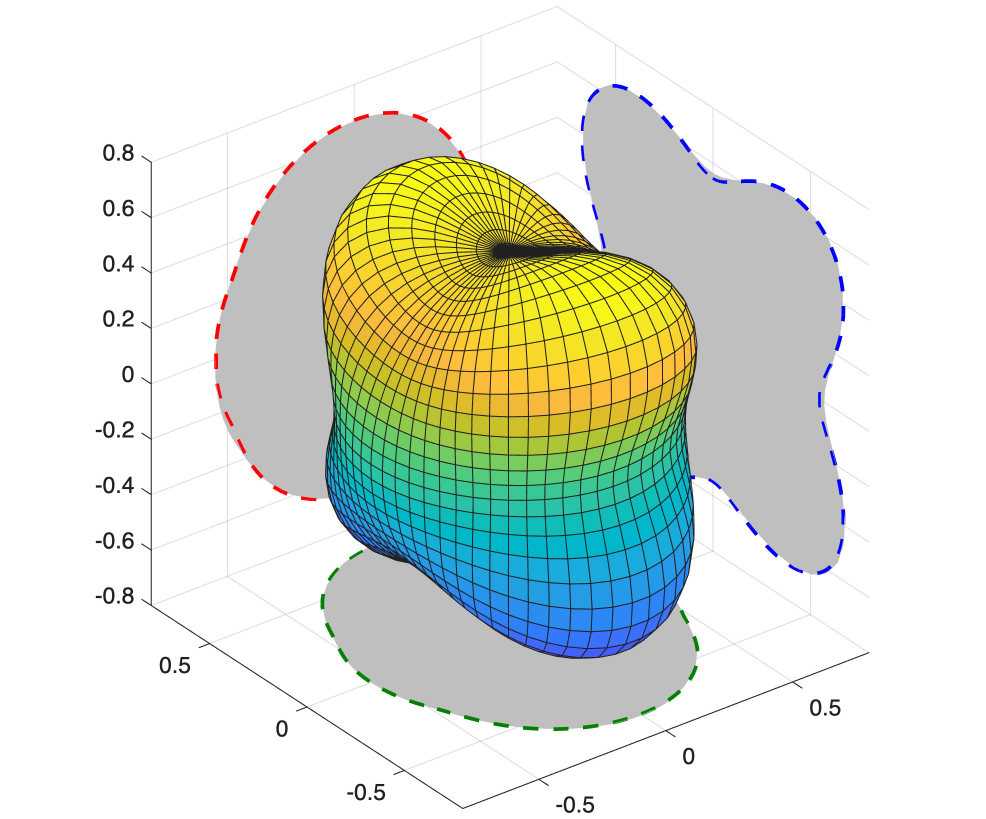}}  &
			{\includegraphics[width=0.21\textwidth]{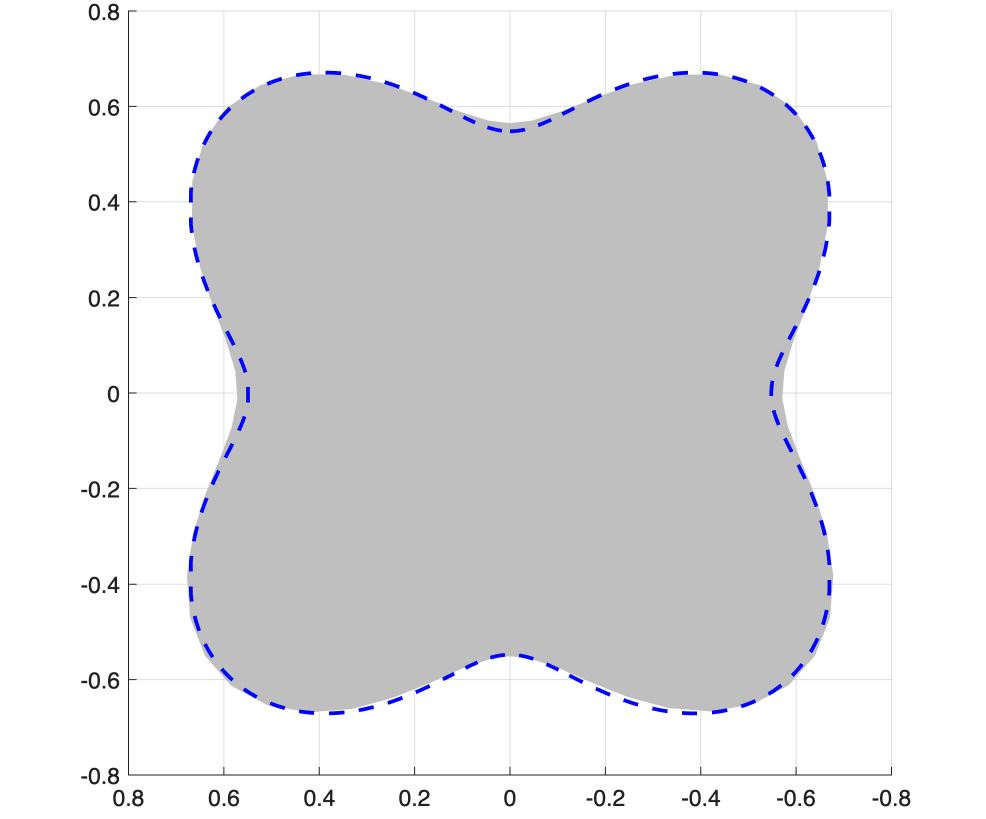}} &
			{\includegraphics[width=0.21\textwidth]{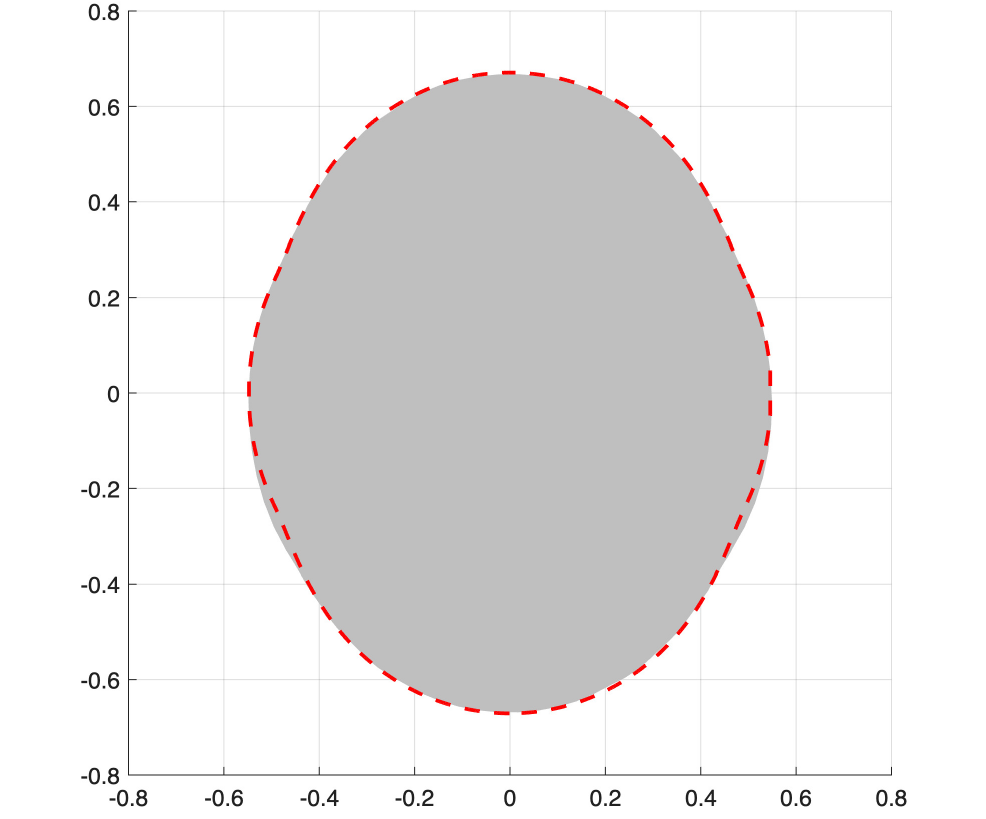}}  &
			{\includegraphics[width=0.21\textwidth]{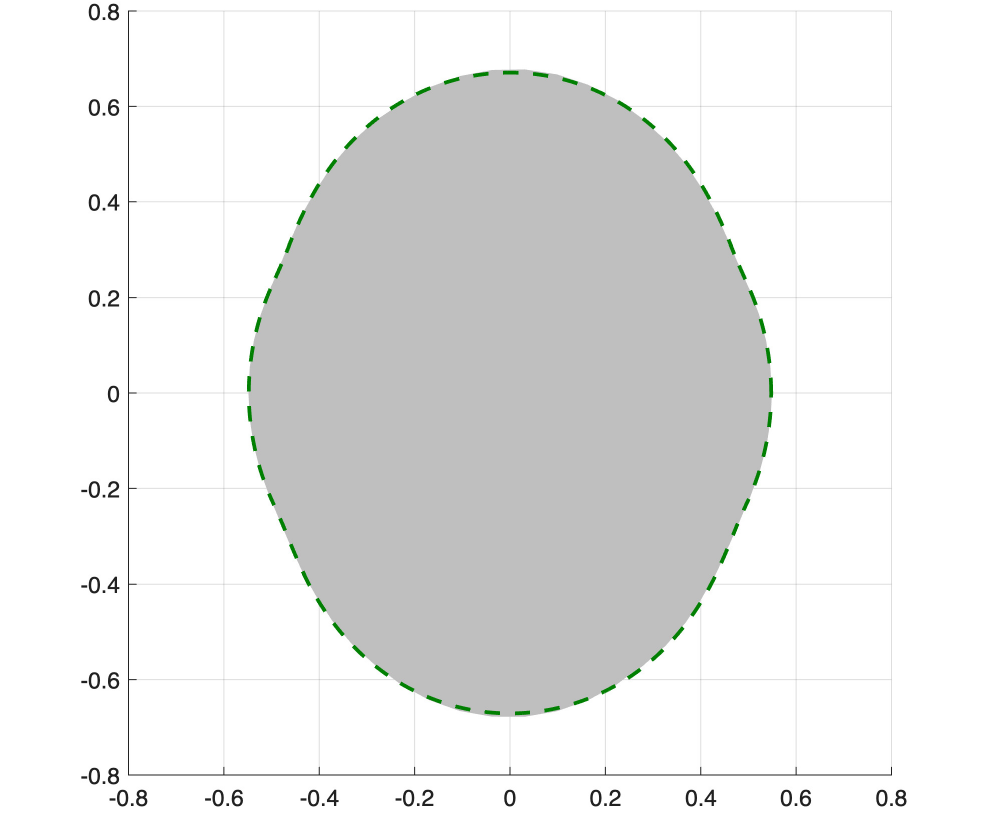}}
		\end{tabular} 
	}\\
	
		\subfigure[Reconstruction with 10$\%$ noise.]{
		\begin{tabular}{cccc}
			{\includegraphics[width=0.21\textwidth]{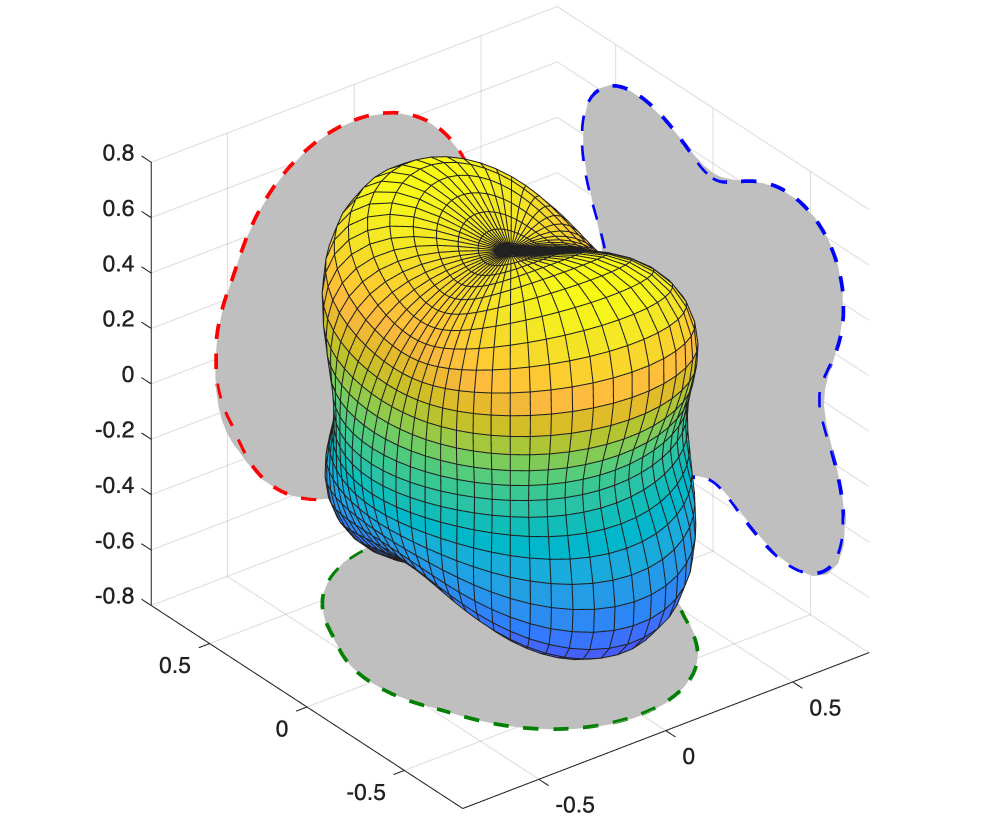}}  &
			{\includegraphics[width=0.21\textwidth]{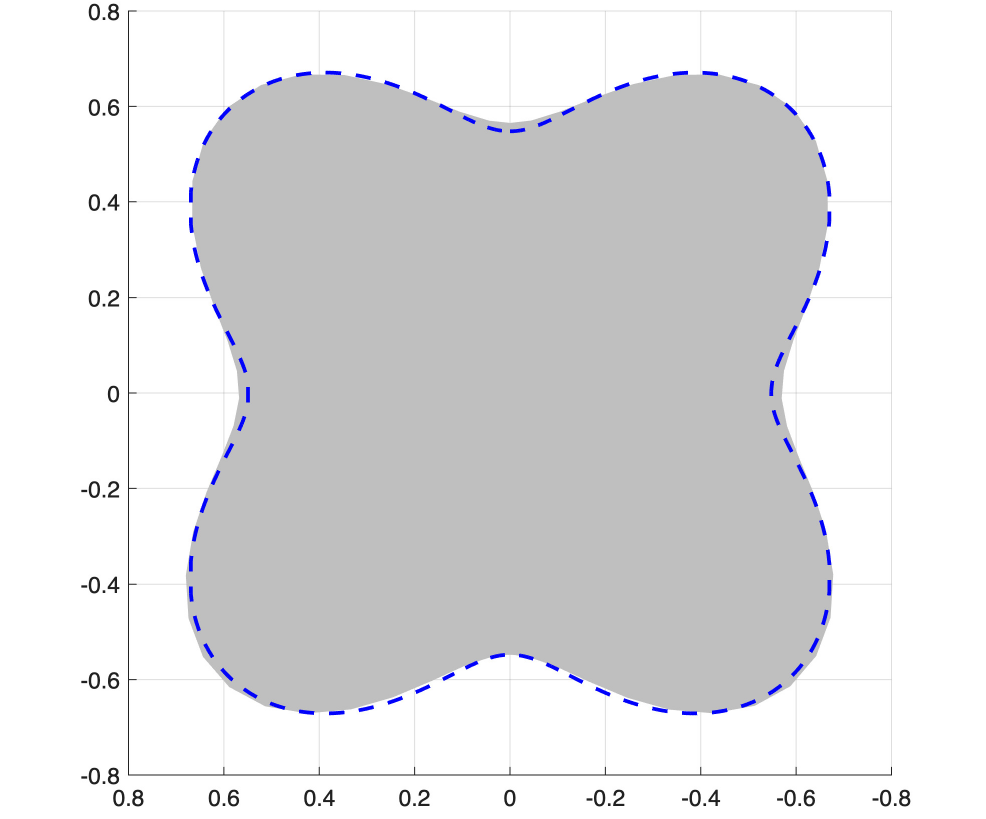}} &
			{\includegraphics[width=0.21\textwidth]{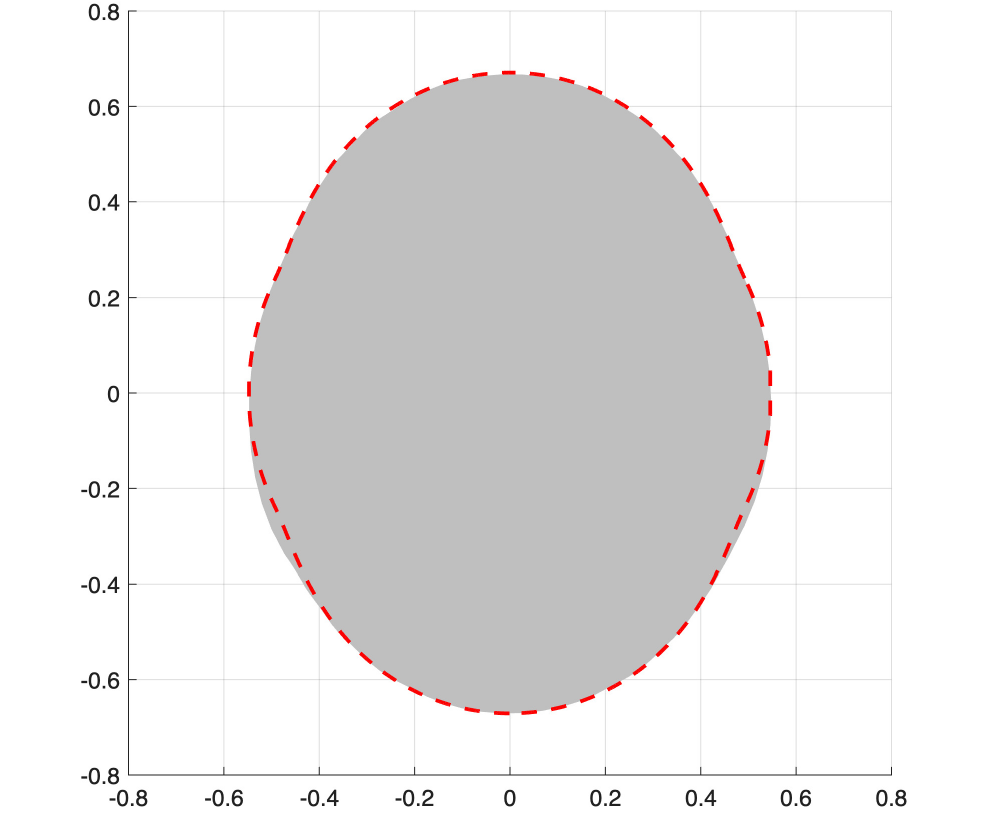}}  &
			{\includegraphics[width=0.21\textwidth]{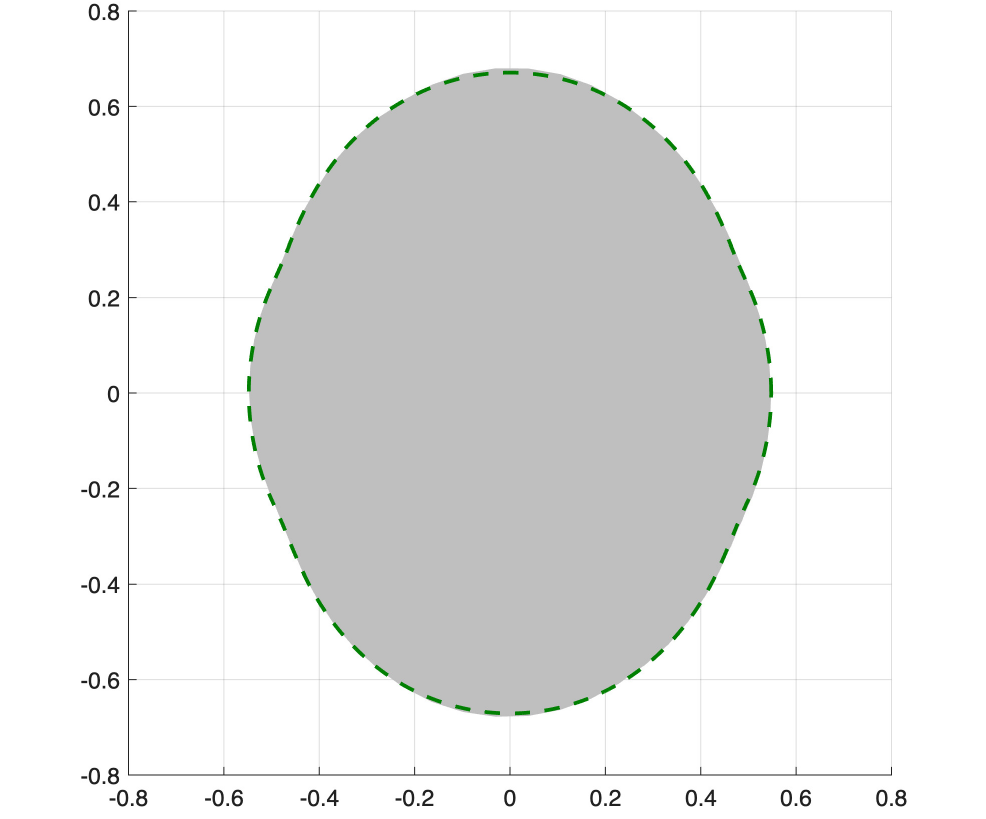}}
		\end{tabular} 
	}\\

	\caption{Reconstructions of a cushion-shaped obstacle with different levels of noise. The initial guess is a sphere with  $\pmb c^{(0)} = (-0.3, 0.2, -0.3)^\top$ and  $r^{(0)} = 0.5$. The incident point source is located at $(0, 0, 5)^\top$.
	}\label{fig_cushion_ex1}
\end{figure}

\begin{figure}[h]
	\centering 
	
	\subfigure[Reconstruction with 10$\%$ noise. The incident point sources are located at $(5, 0, 0)^\top$ and $(-5, 0, 0)^\top$.]{
		\begin{tabular}{cccc}
			{\includegraphics[width=0.21\textwidth]{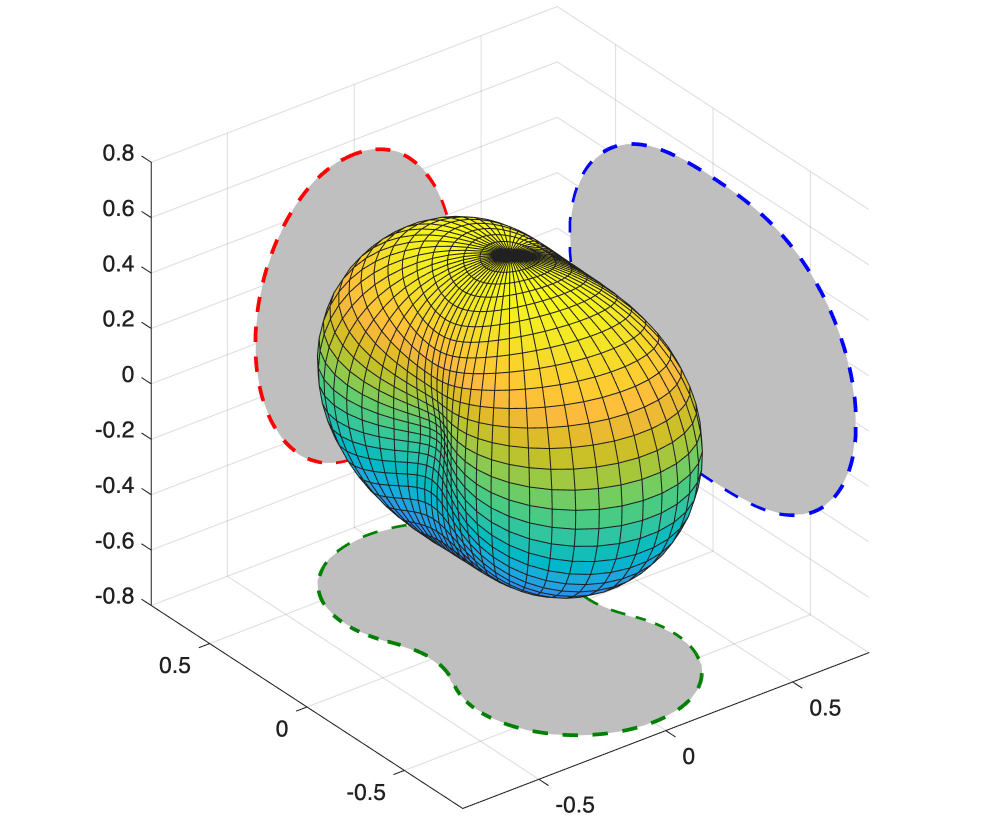}} &
			{\includegraphics[width=0.21\textwidth]{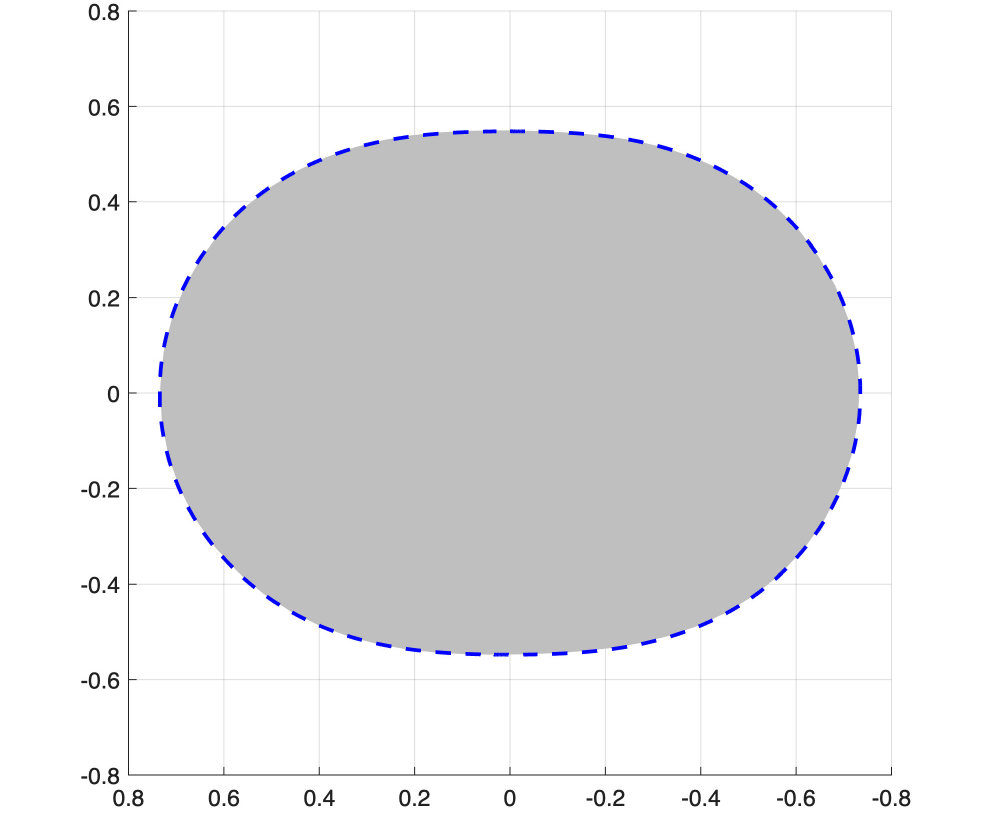}} &
			{\includegraphics[width=0.21\textwidth]{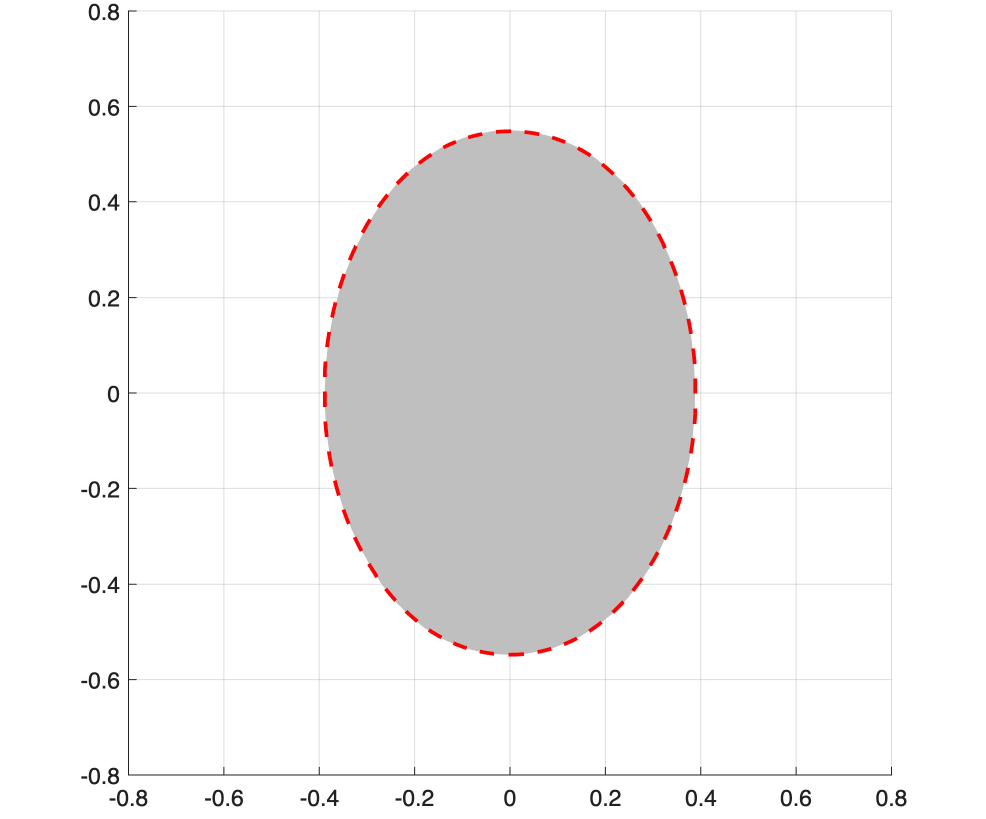}} &
			{\includegraphics[width=0.21\textwidth]{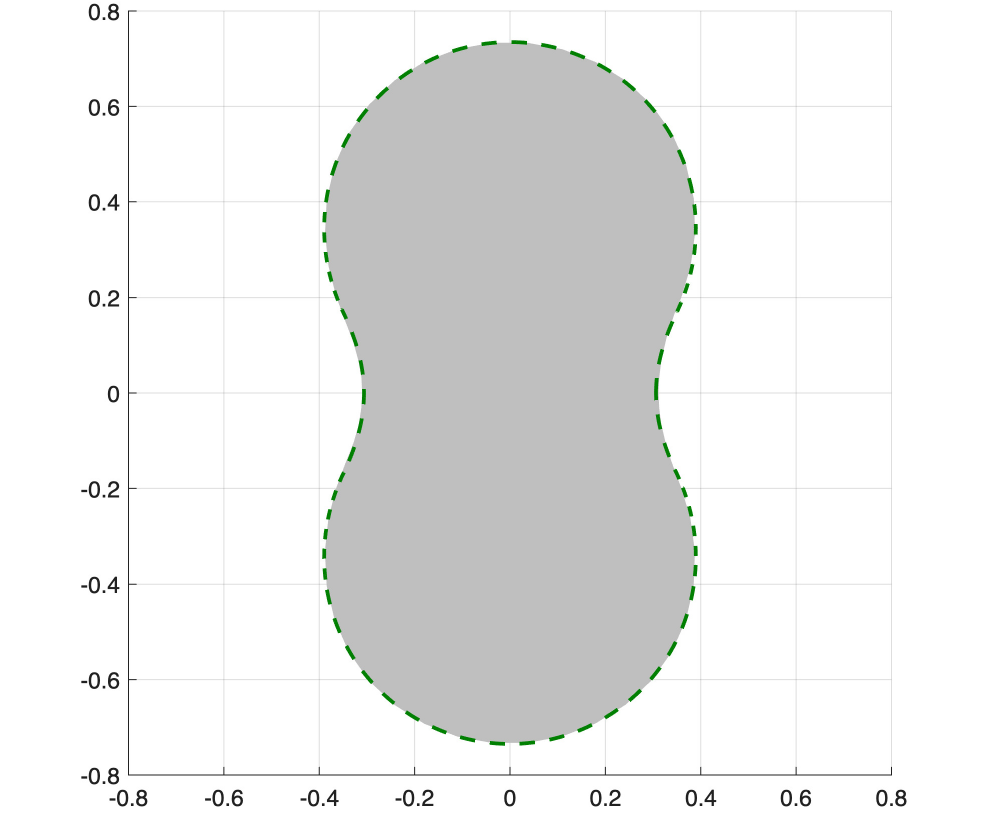}}
		\end{tabular} 
	}\\
	
	\subfigure[Reconstruction with 10$\%$ noise. The incident point sources are located at $(5, 0, 0)^\top$, $(-5, 0, 0)^\top$, $(0, 0, 5)^\top$ and $(0, 0, -5)^\top$.]{
		\begin{tabular}{cccc}
			{\includegraphics[width=0.21\textwidth]{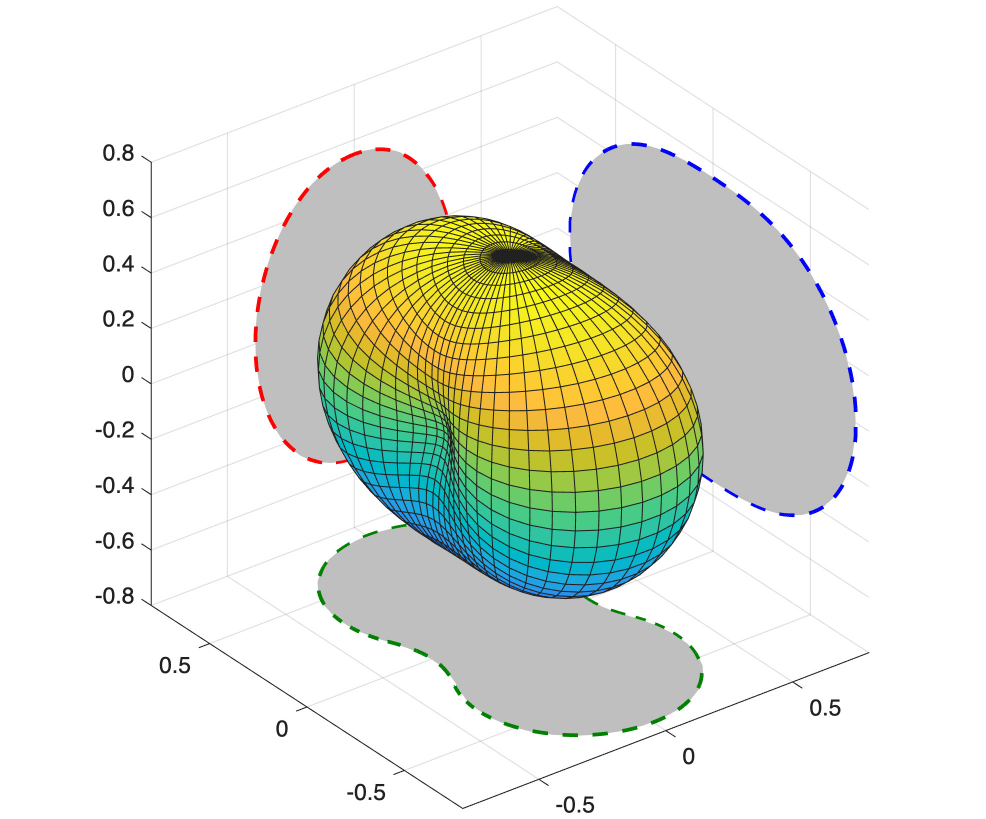}} &
			{\includegraphics[width=0.21\textwidth]{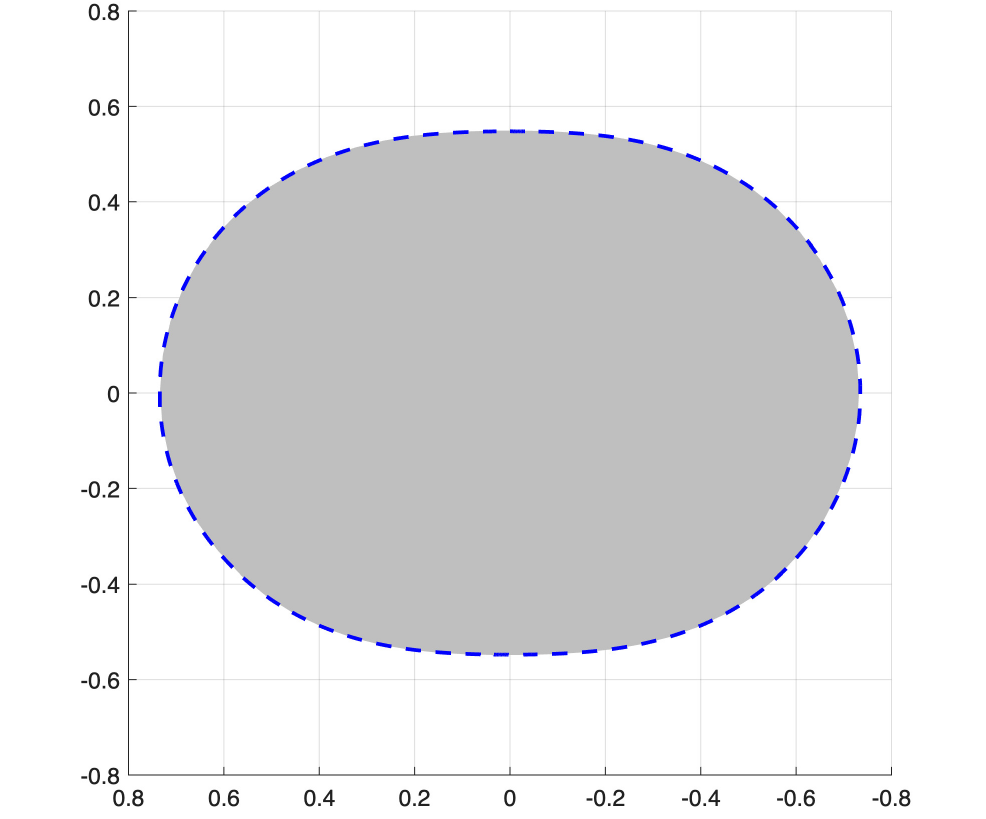}} &
			{\includegraphics[width=0.21\textwidth]{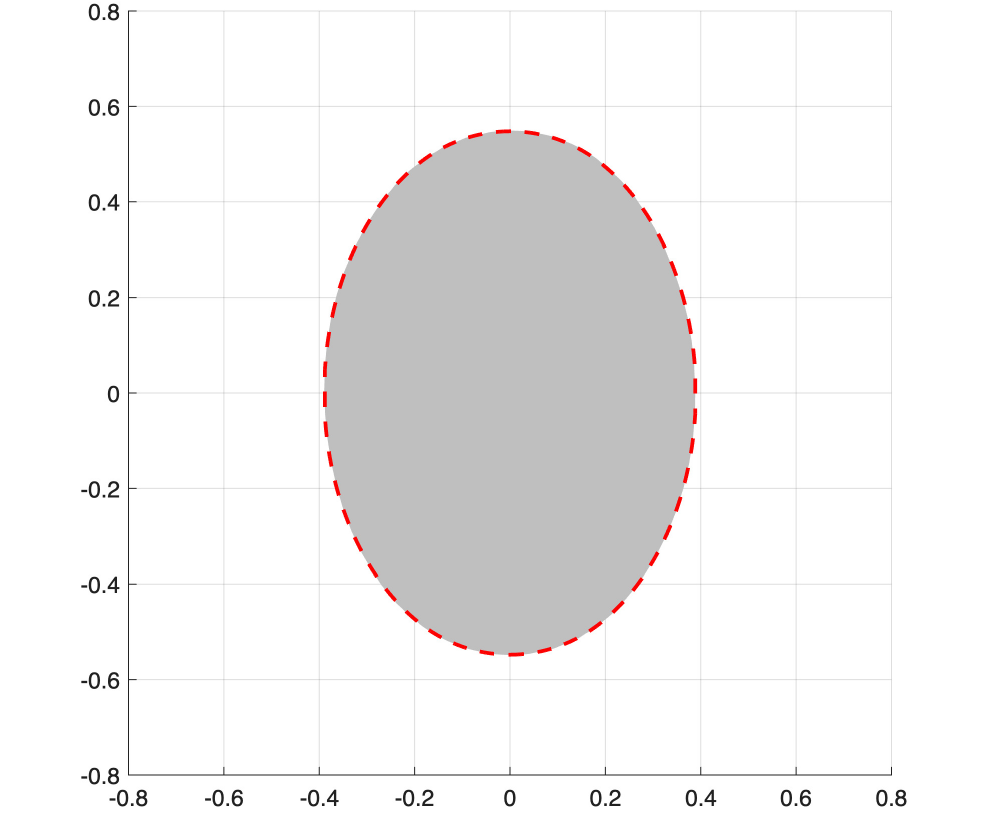}} &
			{\includegraphics[width=0.21\textwidth]{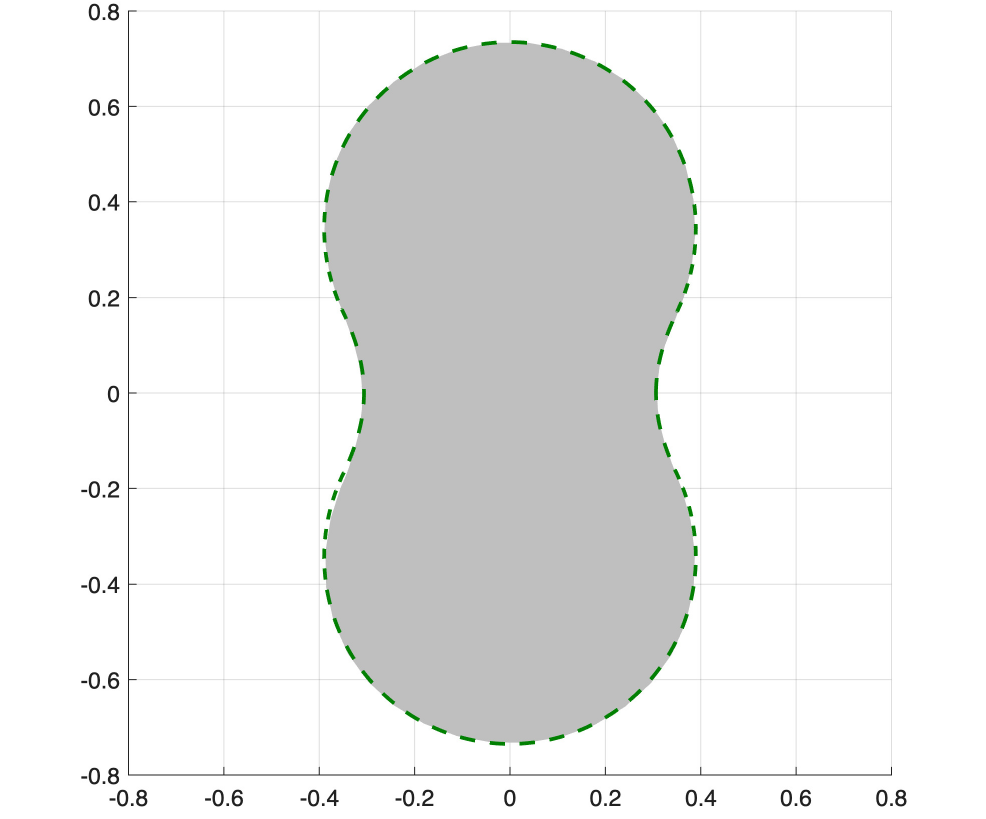}}
		\end{tabular} 
	}\\

	\caption{Reconstructions of a pinched ball-shaped obstacle with multiple launch positions. 
		The initial guess is a sphere with  $\pmb c^{(0)} = (-0.5, 0.4, -0.3)^\top$ and  $r^{(0)} = 0.6$.}\label{fig_pinchedball_ex2}
\end{figure}

\begin{figure}[h]
	\centering 
	\subfigure[Reconstruction with 10$\%$ noise. The incident point sources are located at $(0, 0, 5)^\top$ and $(0, 0, -5)^\top$.]{
		\begin{tabular}{cccc}
			{\includegraphics[width=0.21\textwidth]{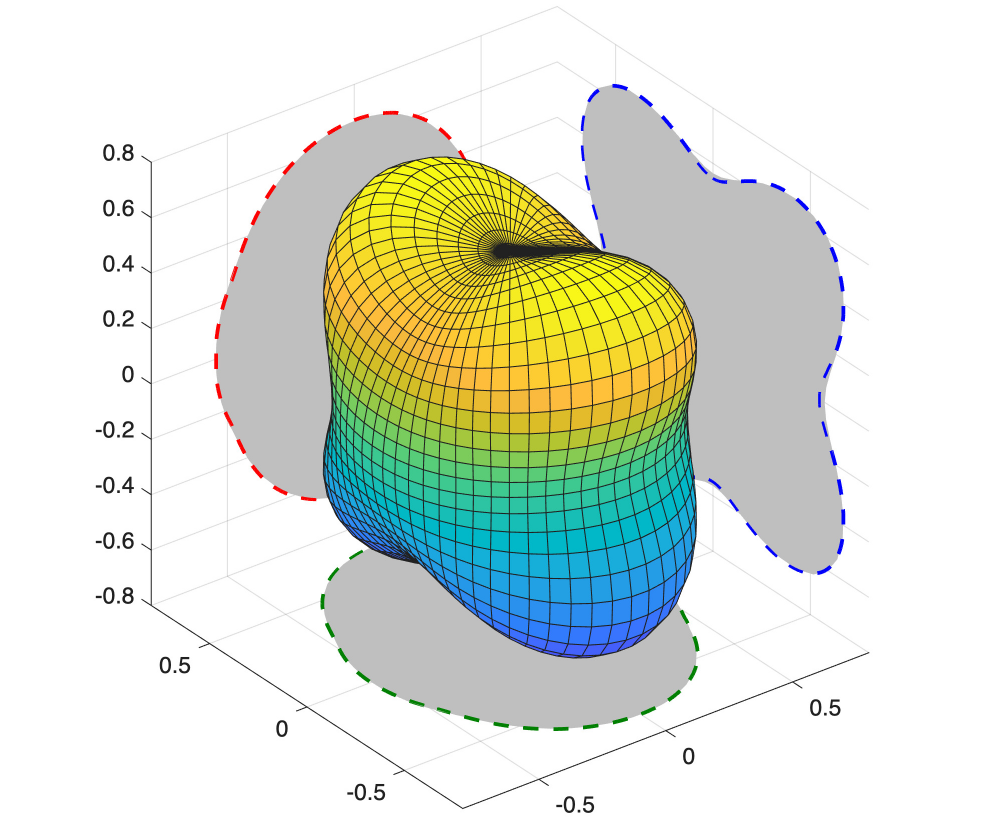}}  &
			{\includegraphics[width=0.21\textwidth]{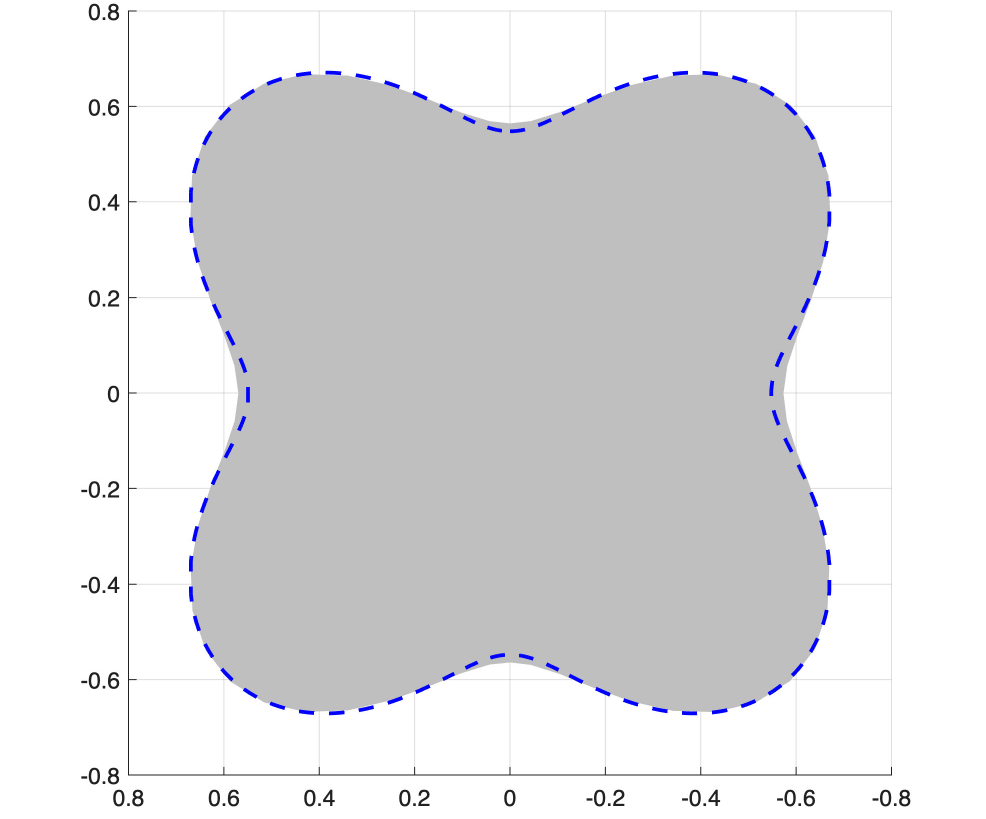}} &
			{\includegraphics[width=0.21\textwidth]{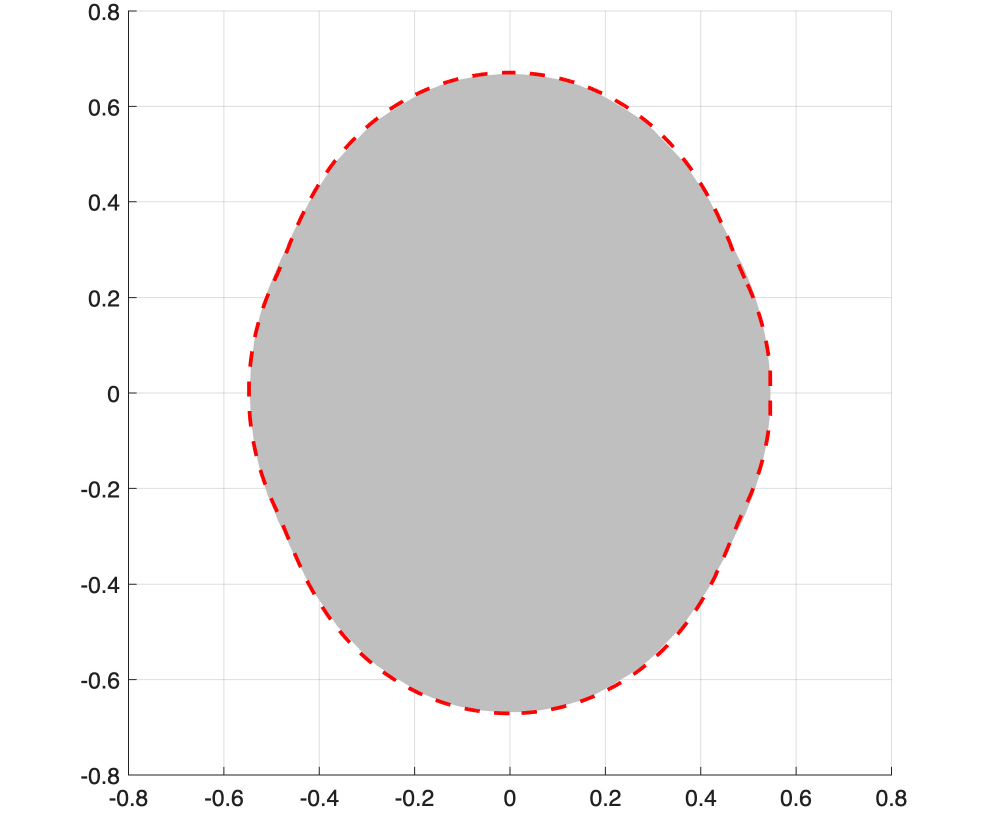}}  &
			{\includegraphics[width=0.21\textwidth]{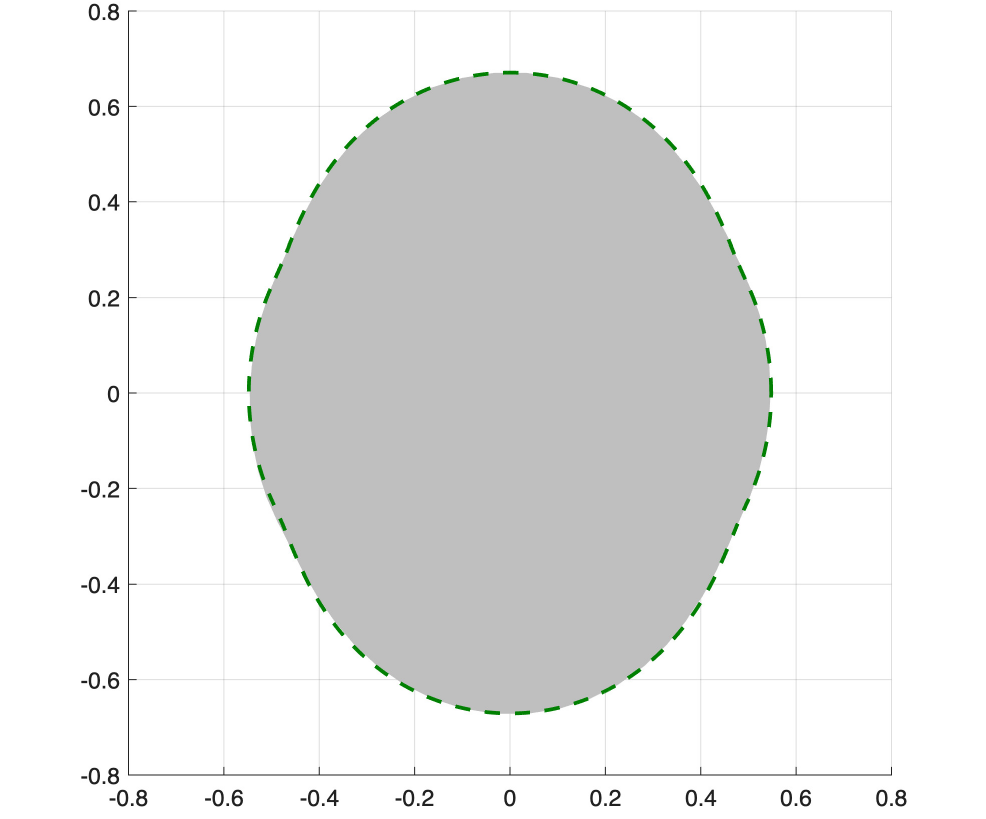}}
		\end{tabular} 
	}\\
	
	\subfigure[Reconstruction with 10$\%$ noise. The incident point sources are located at $(0, 5, 0)^\top$, $(0, -5, 0)^\top$, $(0, 0, 5)^\top$ and $(0, 0, -5)^\top$.]{
		\begin{tabular}{cccc}
			{\includegraphics[width=0.21\textwidth]{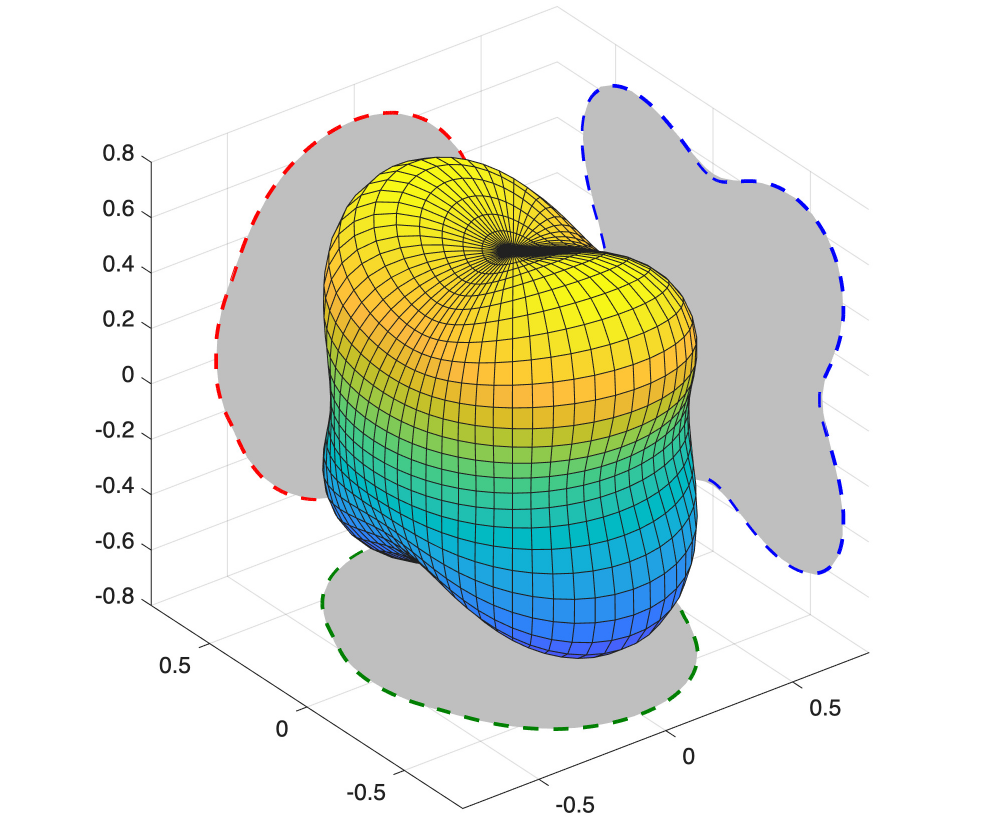}}  &
			{\includegraphics[width=0.21\textwidth]{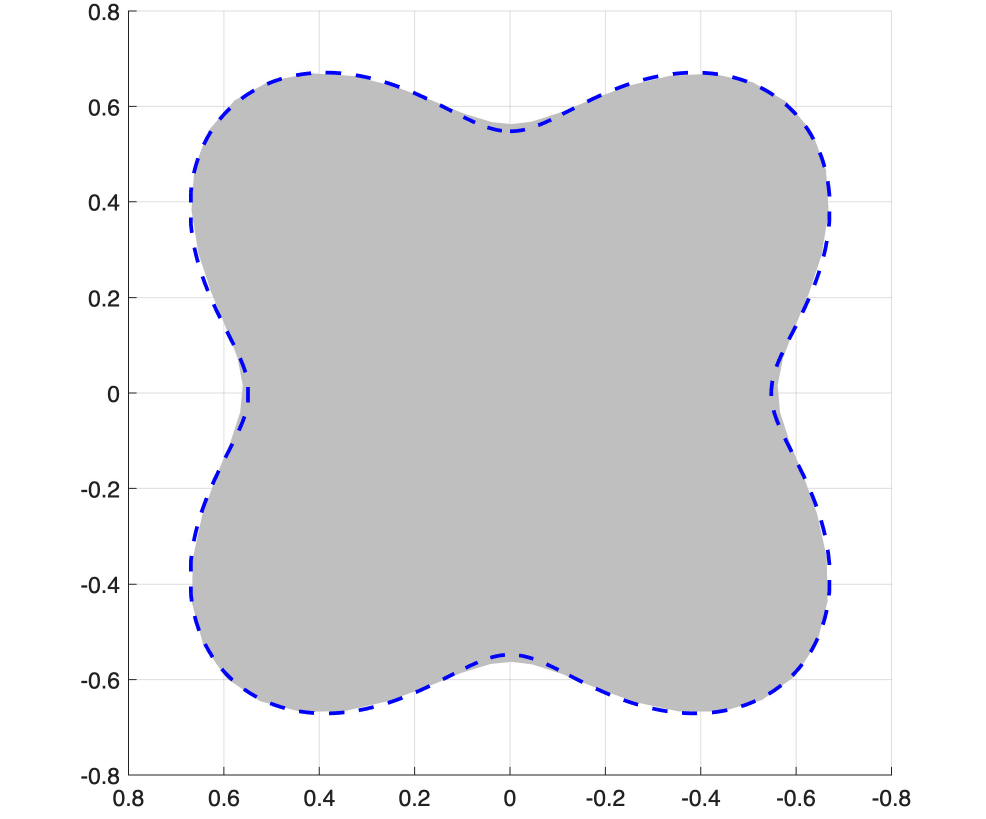}} &
			{\includegraphics[width=0.21\textwidth]{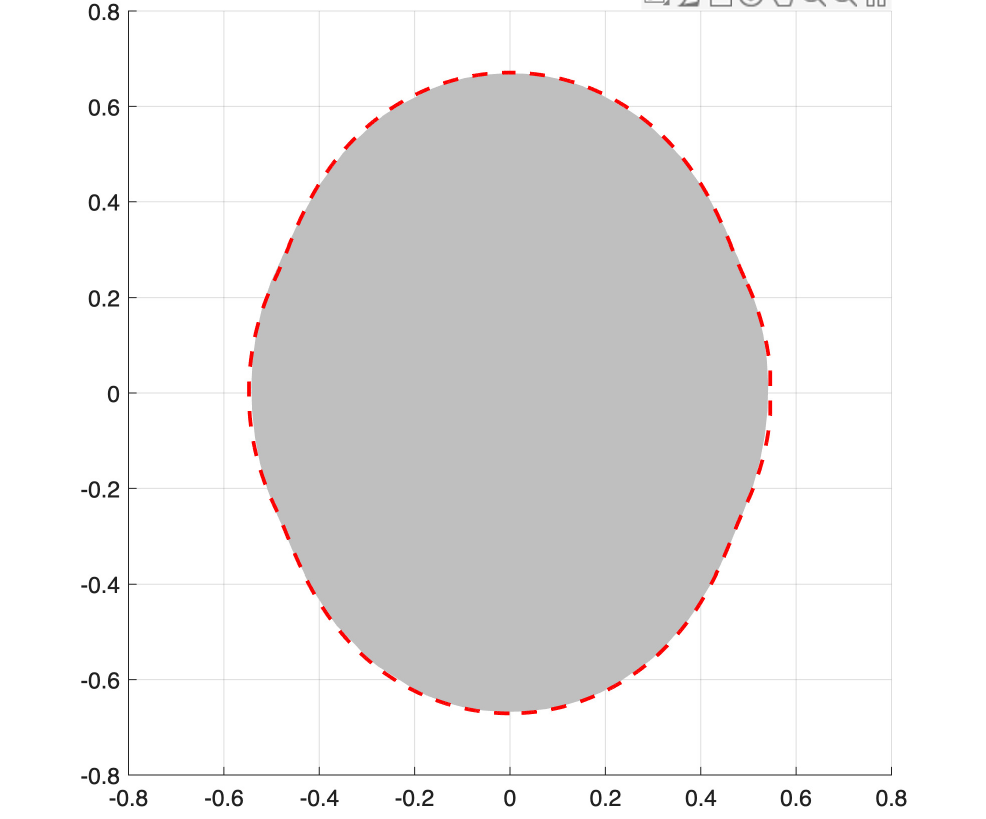}}  &
			{\includegraphics[width=0.21\textwidth]{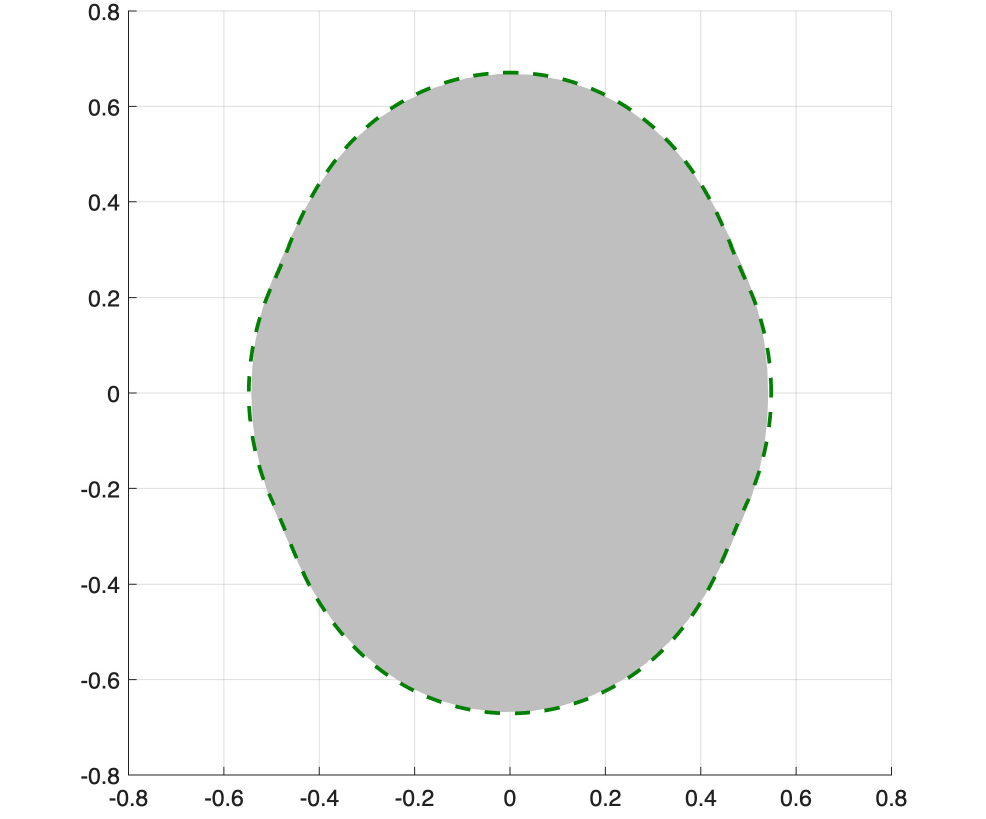}}
		\end{tabular} 
	}\\

	\caption{Reconstructions of a cushion-shaped obstacle with multiple launch positions. The initial guess is a sphere with  $\pmb c^{(0)} = (-0.3, 0.2, -0.3)^\top$ and  $r^{(0)} = 0.5$.}\label{fig_cushion_ex2}
\end{figure}

\begin{figure}[h]
	\centering 
	
	\subfigure[The geometric contraction factor is $\varsigma = 0.9$, the initial guess is a sphere with center $\pmb{c}^{(0)} = (-0.2, 0.3,-0.2)^\top$ and radius $r^{(0)} = 0.3$, and the observation radius $R = 1.5$.]{
		\begin{tabular}{cccc}
			{\includegraphics[width=0.21\textwidth]{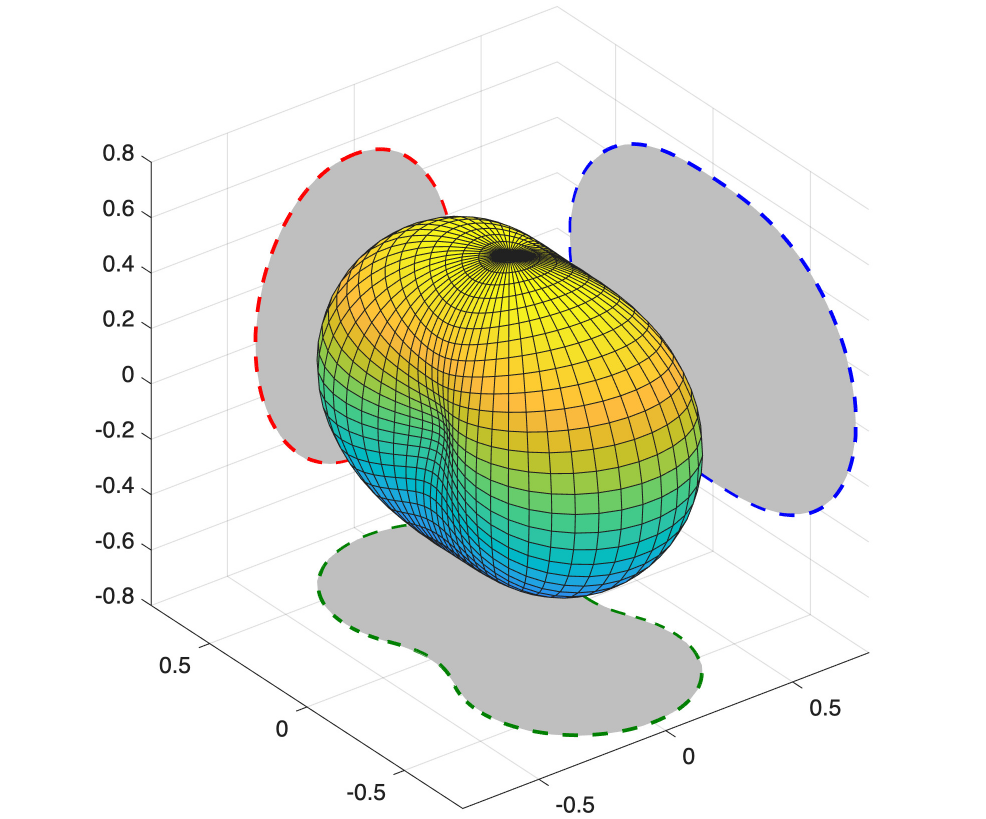}}  &
			{\includegraphics[width=0.21\textwidth]{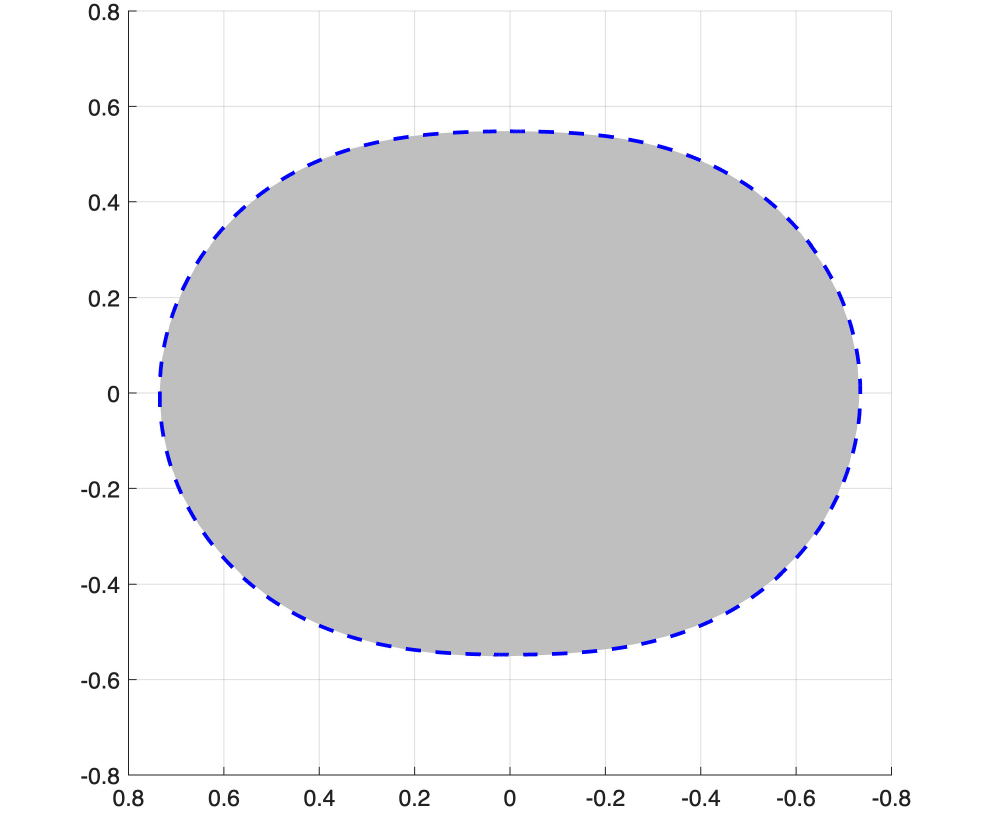}} &
			{\includegraphics[width=0.21\textwidth]{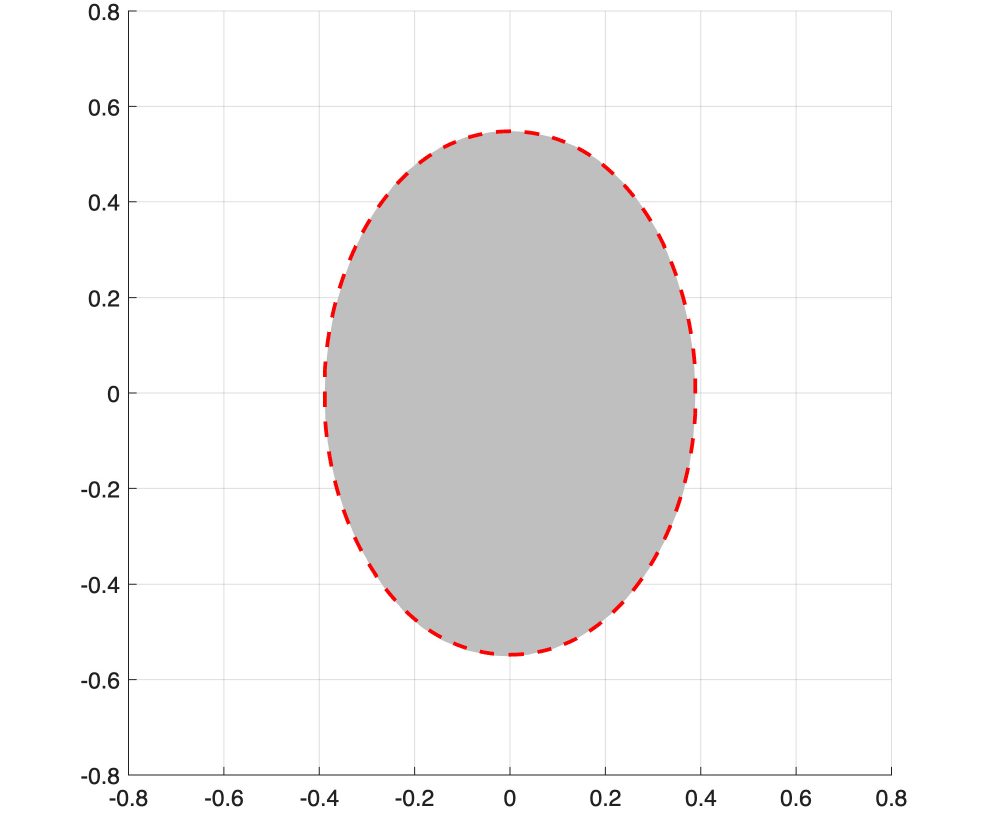}} &
			{\includegraphics[width=0.21\textwidth]{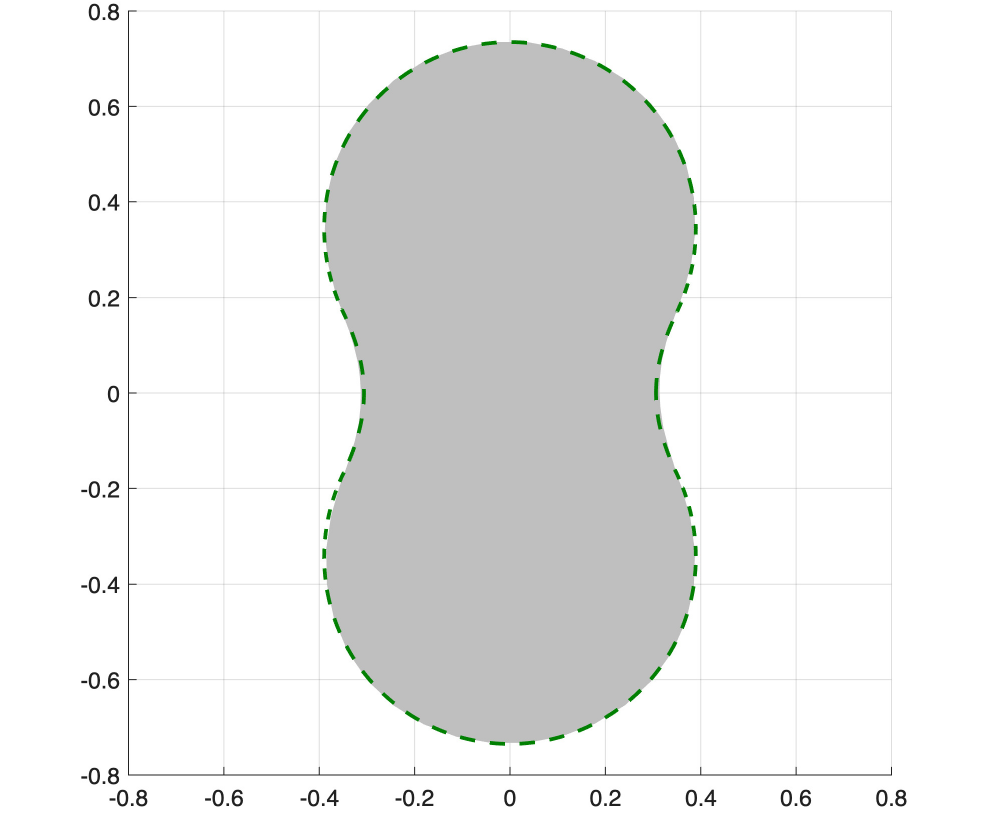}}
		\end{tabular} 
	}\\

	\subfigure[The geometric contraction factor is $\varsigma = 0.9$, the initial guess is a sphere with center $\pmb{c}^{(0)} = (0.3, 0.4, -0.3)^\top$ and radius $r^{(0)} = 0.7$, and the observation radius $R = 1.5$.]{
		\begin{tabular}{cccc}
			{\includegraphics[width=0.21\textwidth]{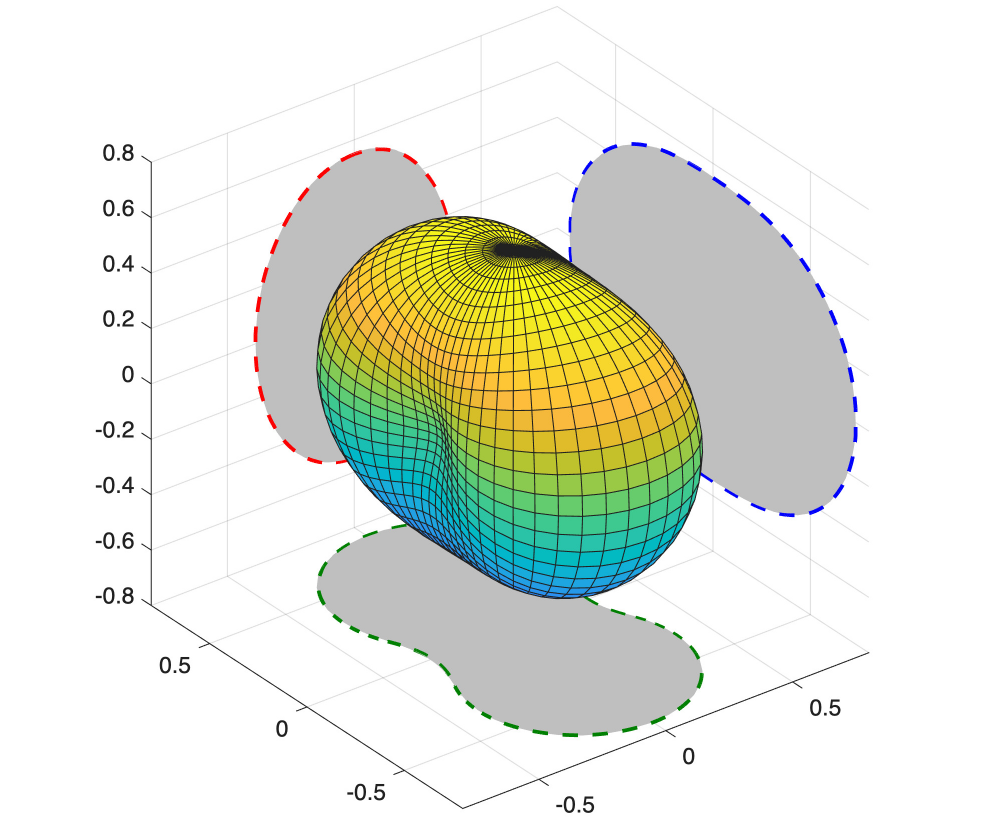}}  &
			{\includegraphics[width=0.21\textwidth]{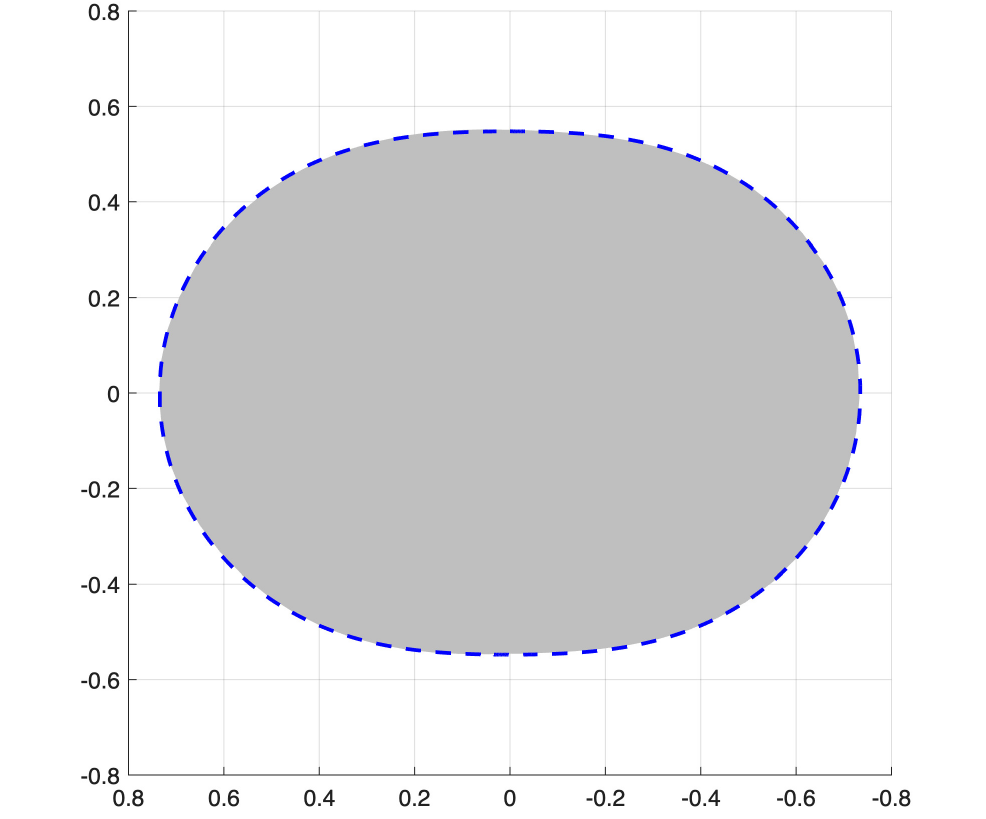}} &
			{\includegraphics[width=0.21\textwidth]{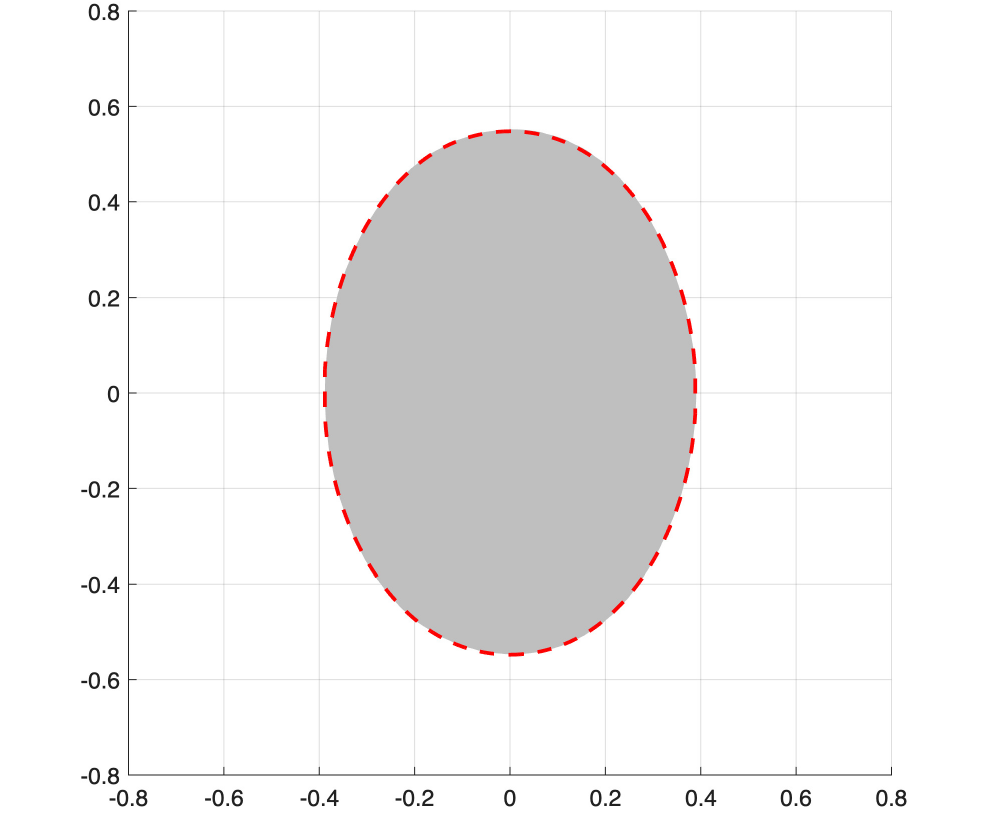}} &
			{\includegraphics[width=0.21\textwidth]{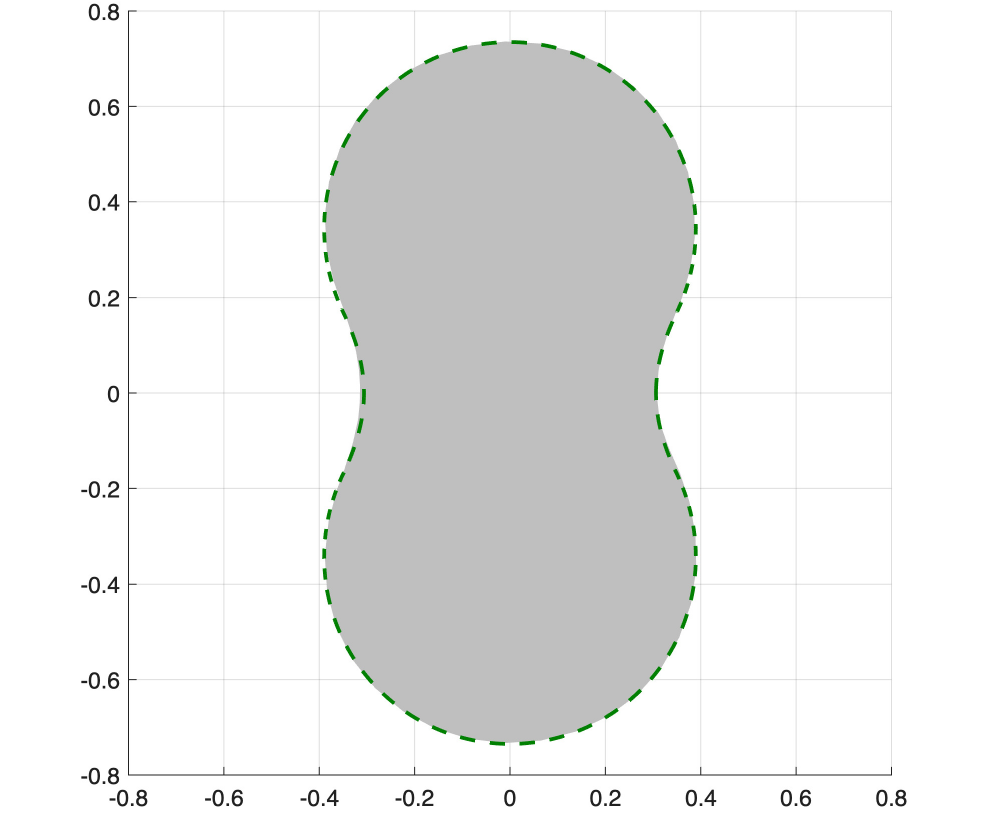}}
		\end{tabular} 
	}\\
	
	\subfigure[The geometric contraction factor is $\varsigma = 0.9$, the initial guess is a sphere with center $\pmb{c}^{(0)} = (0.3, 0.4, -0.3)^\top$ and radius $r^{(0)} = 0.7$, and the observation radius $R = 2.5$.]{
		\begin{tabular}{cccc}
			{\includegraphics[width=0.21\textwidth]{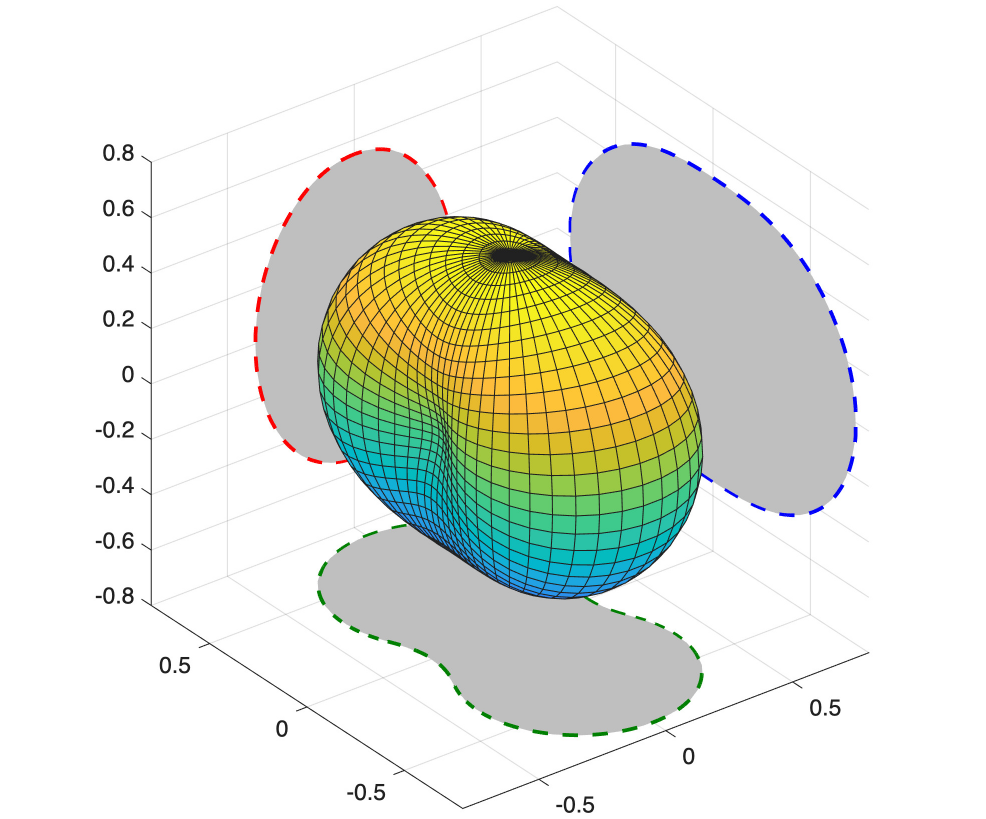}}  &
			{\includegraphics[width=0.21\textwidth]{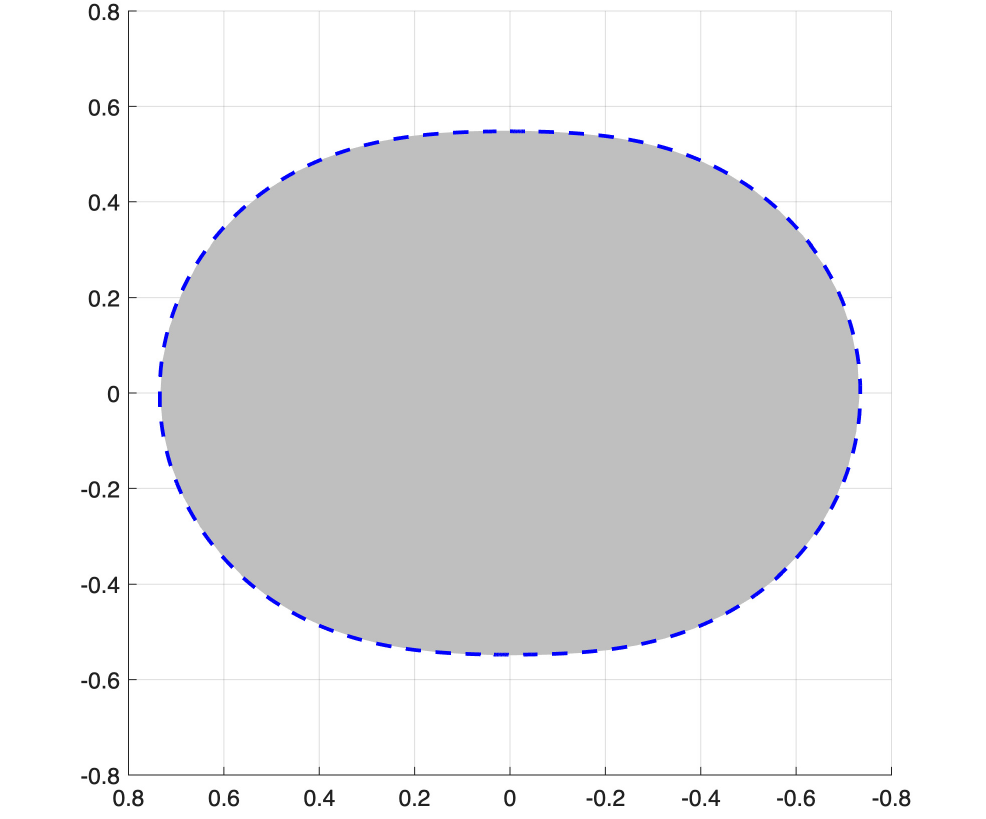}} &
			{\includegraphics[width=0.21\textwidth]{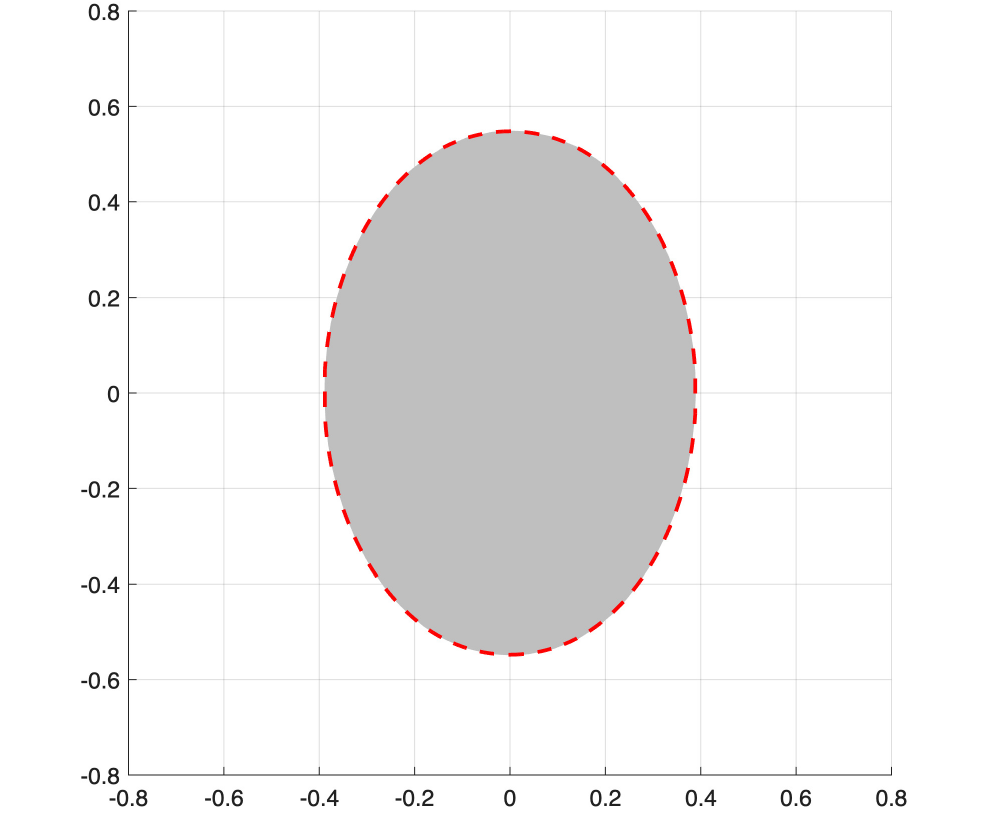}} &
			{\includegraphics[width=0.21\textwidth]{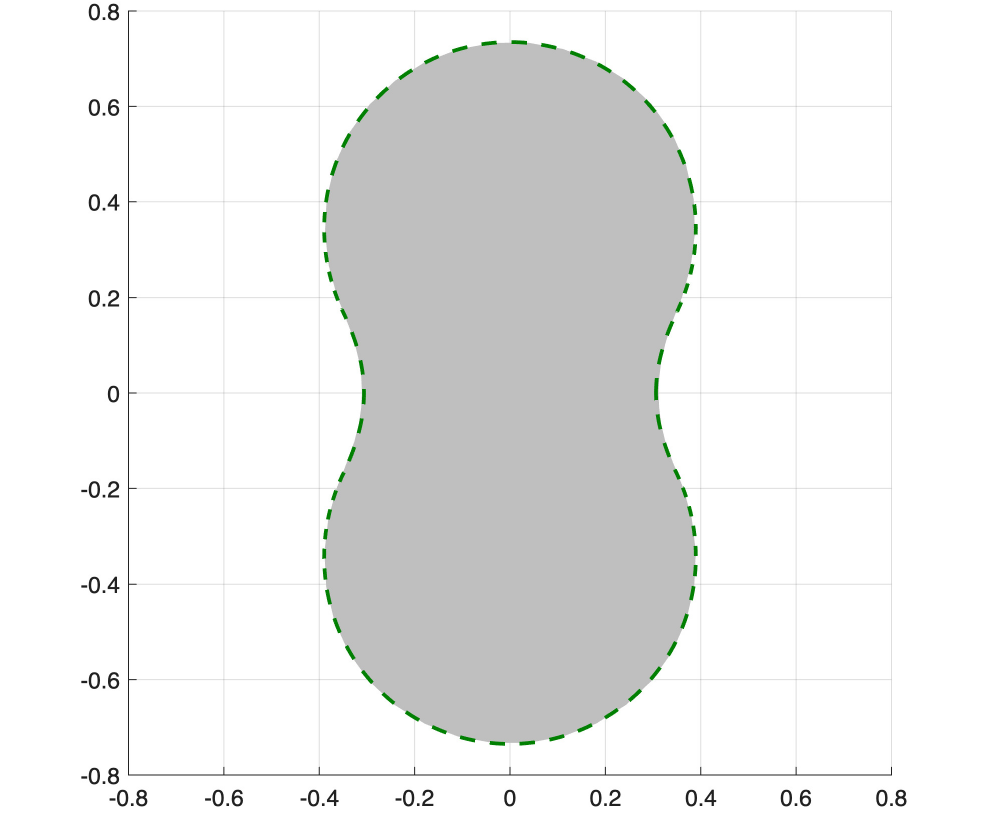}}
		\end{tabular} 
	}\\

	\caption{Reconstructions of a pinched ball-shaped obstacle with 10\% noise, where the incident point sources are located at $(0, 0, 5)^\top$ and $(0, 0, -5)^\top$. 
	}\label{fig_pinchedball_ex3}
\end{figure}

\begin{figure}[h]
	\centering 	
	
	\subfigure[The geometric contraction factor is $\varsigma = 0.9$, the initial guess is a sphere with center $\pmb{c}^{(0)} = (-0.1,0.3,-0.1)^\top$ and radius $r^{(0)} = 0.3$, and the observation radius $R = 1.5$.]{
		\begin{tabular}{cccc}
			{\includegraphics[width=0.21\textwidth]{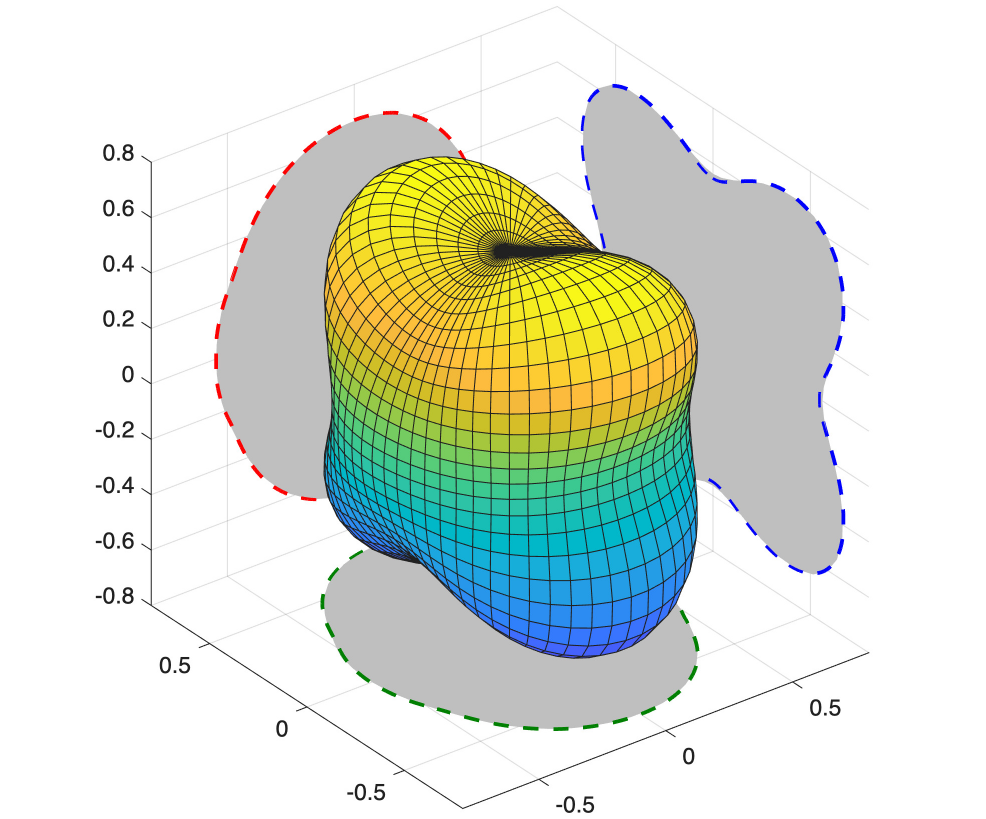}} &
			{\includegraphics[width=0.21\textwidth]{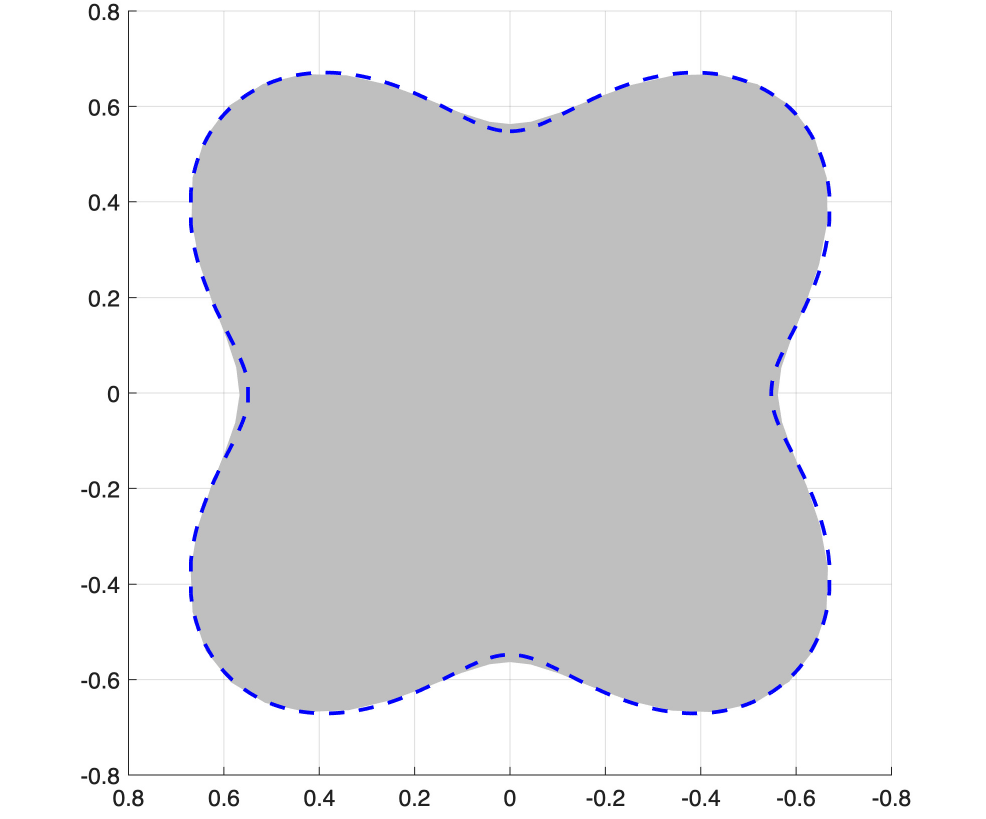}} &
			{\includegraphics[width=0.21\textwidth]{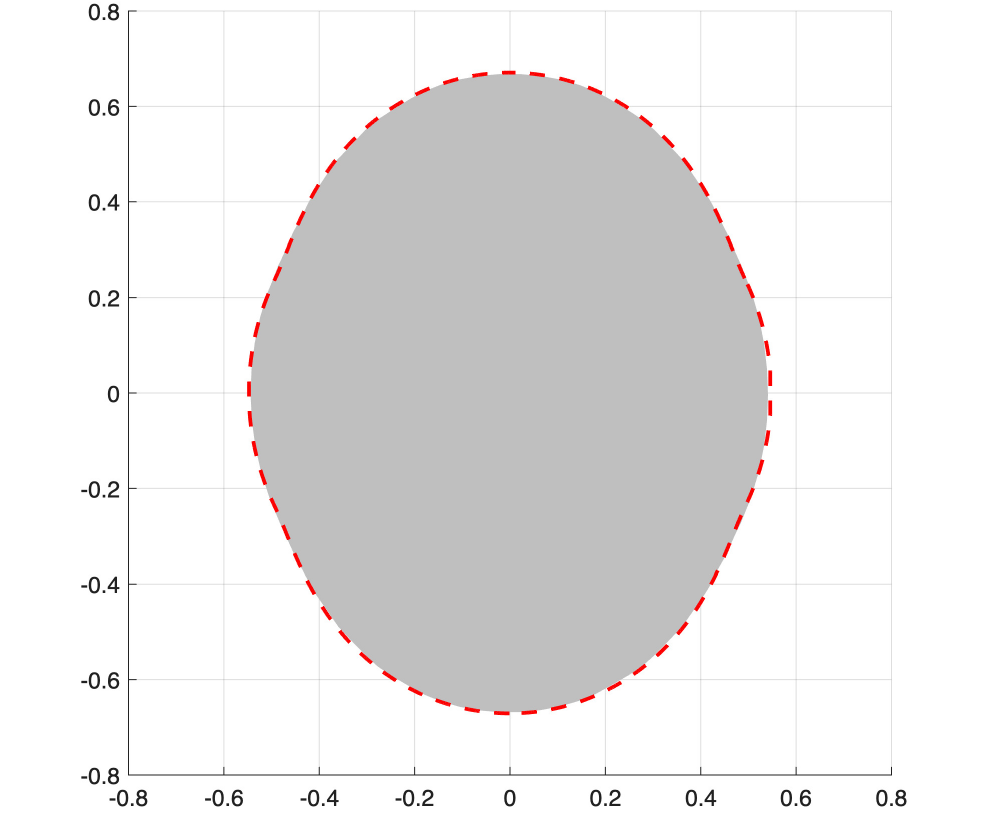}}  &
			{\includegraphics[width=0.21\textwidth]{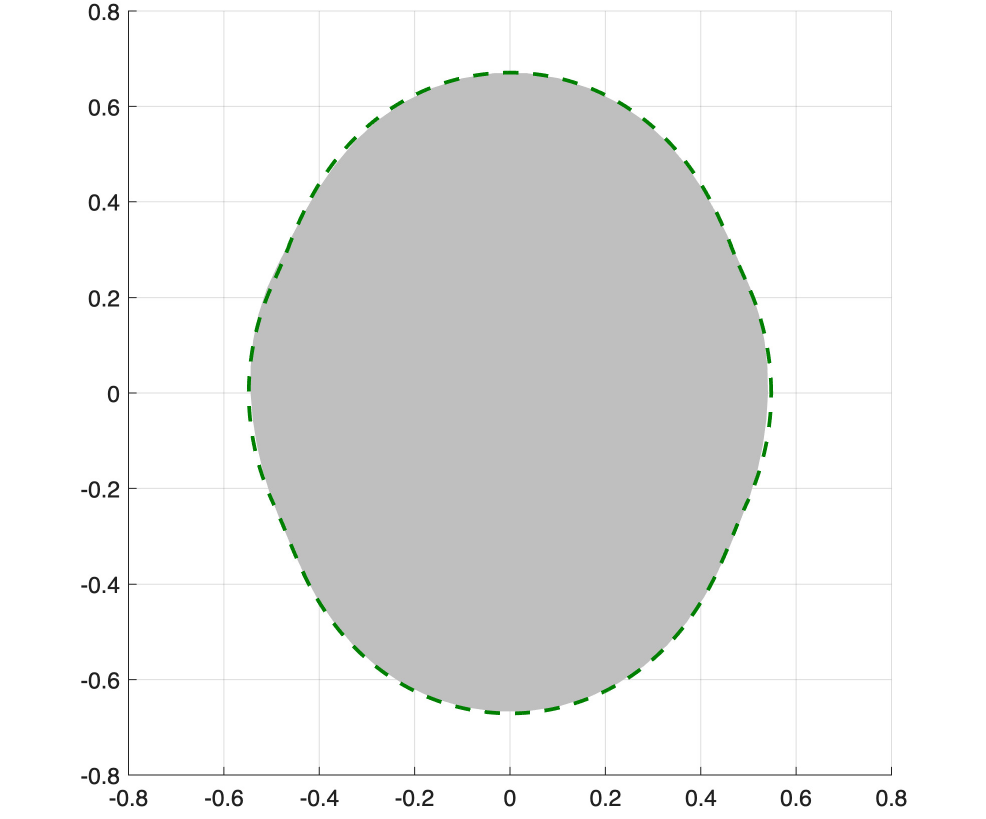}}
		\end{tabular} 
	}\\

	\subfigure[The geometric contraction factor is $\varsigma = 0.9$, the initial guess is a sphere with center $\pmb{c}^{(0)} = (0.4, 0.4,0.2)^\top$ and radius $r^{(0)} = 0.6$, and the observation radius $R = 1.5$.]{
		\begin{tabular}{cccc}
			{\includegraphics[width=0.21\textwidth]{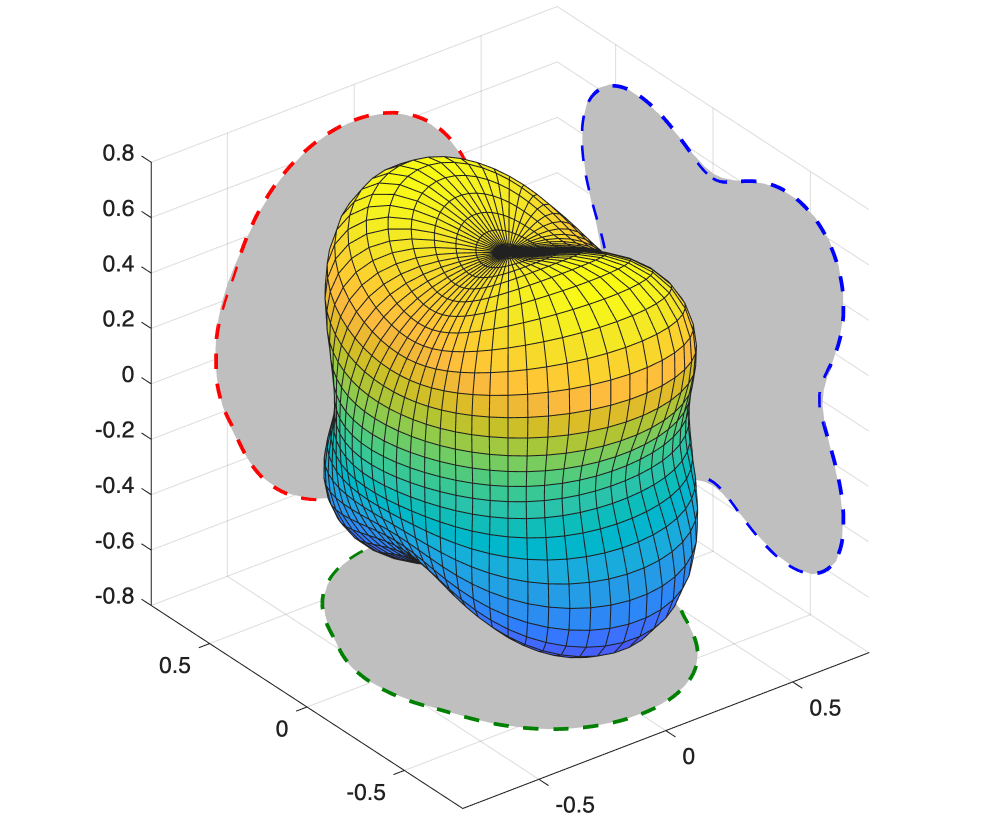}} &
			{\includegraphics[width=0.21\textwidth]{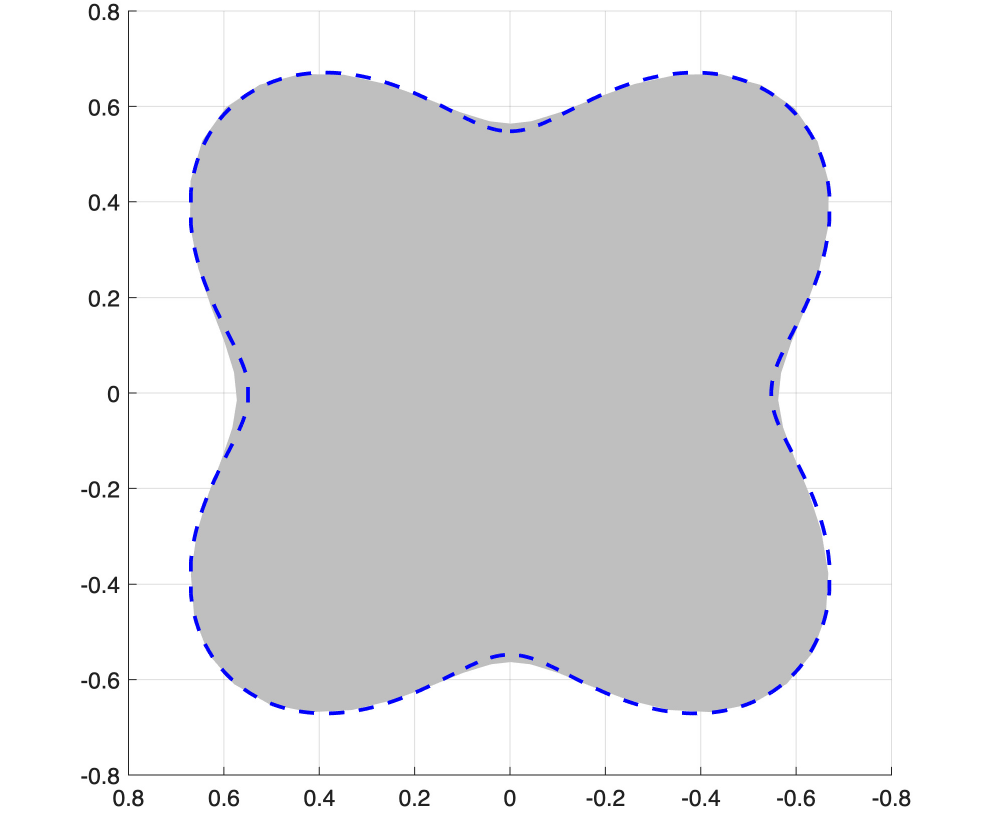}} &
			{\includegraphics[width=0.21\textwidth]{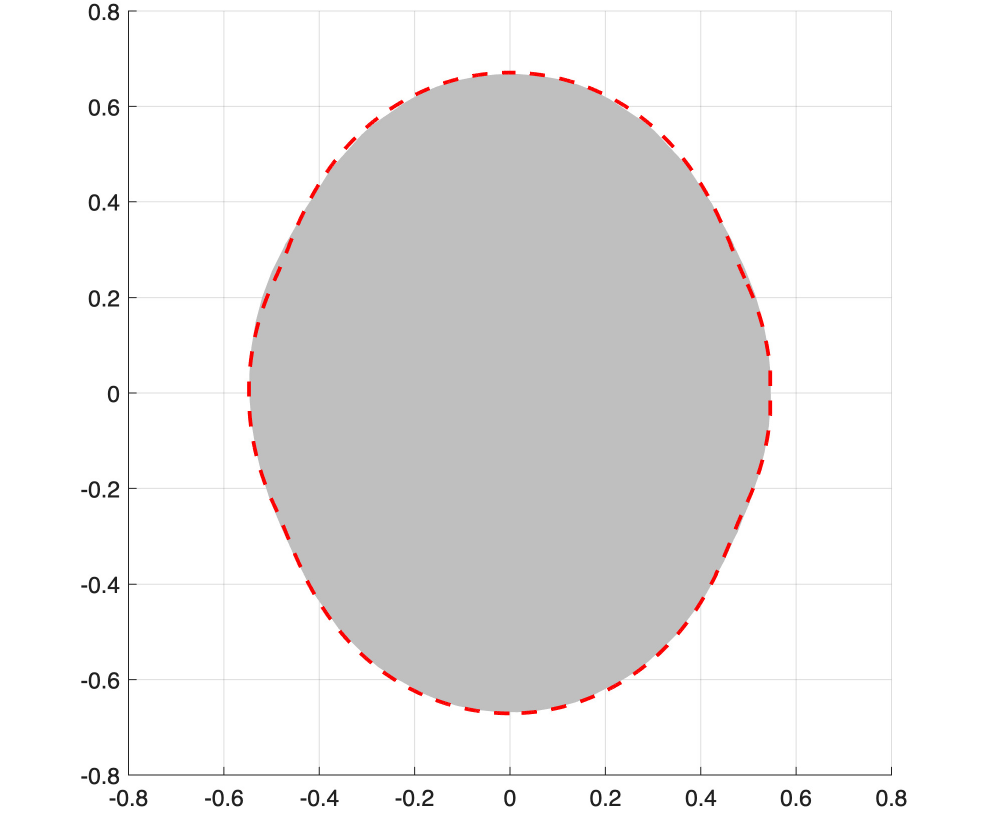}}  &
			{\includegraphics[width=0.21\textwidth]{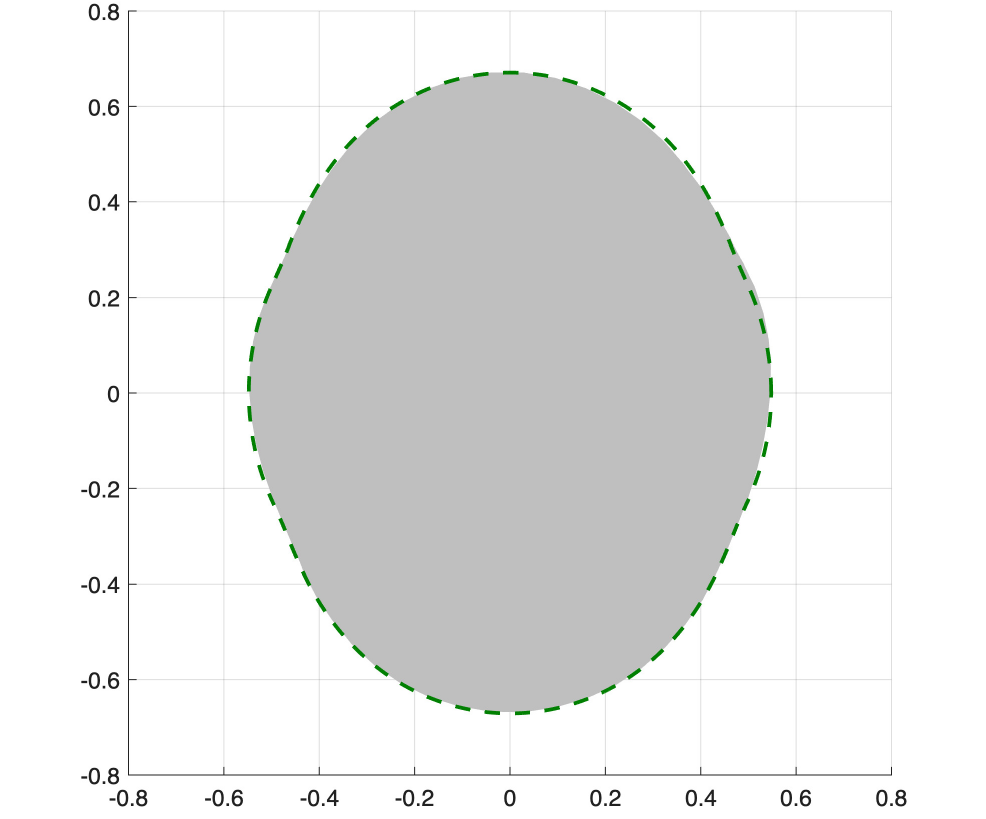}}
		\end{tabular} 
	}\\

	\subfigure[The geometric contraction factor is $\varsigma = 0.9$, the initial guess is a sphere with center $\pmb{c}^{(0)} = (0.4, 0.4,0.2)^\top$ and radius $r^{(0)} = 0.6$, and the observation radius $R = 3$.]{
		\begin{tabular}{cccc}
			{\includegraphics[width=0.21\textwidth]{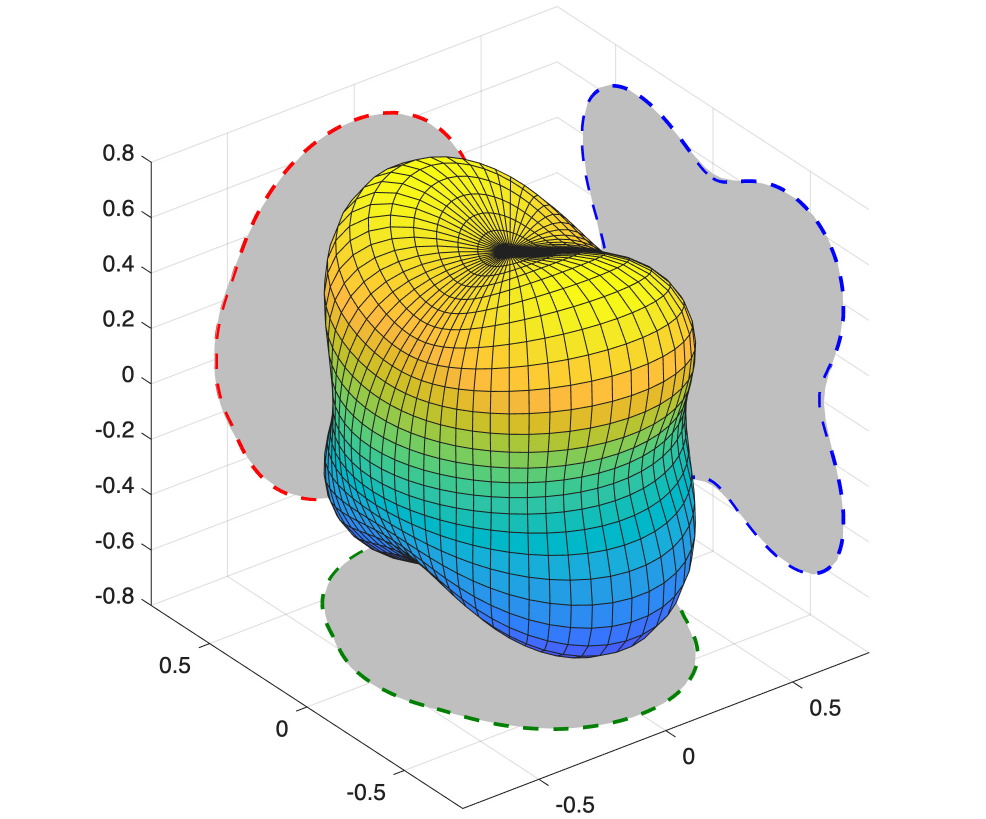}} &
			{\includegraphics[width=0.21\textwidth]{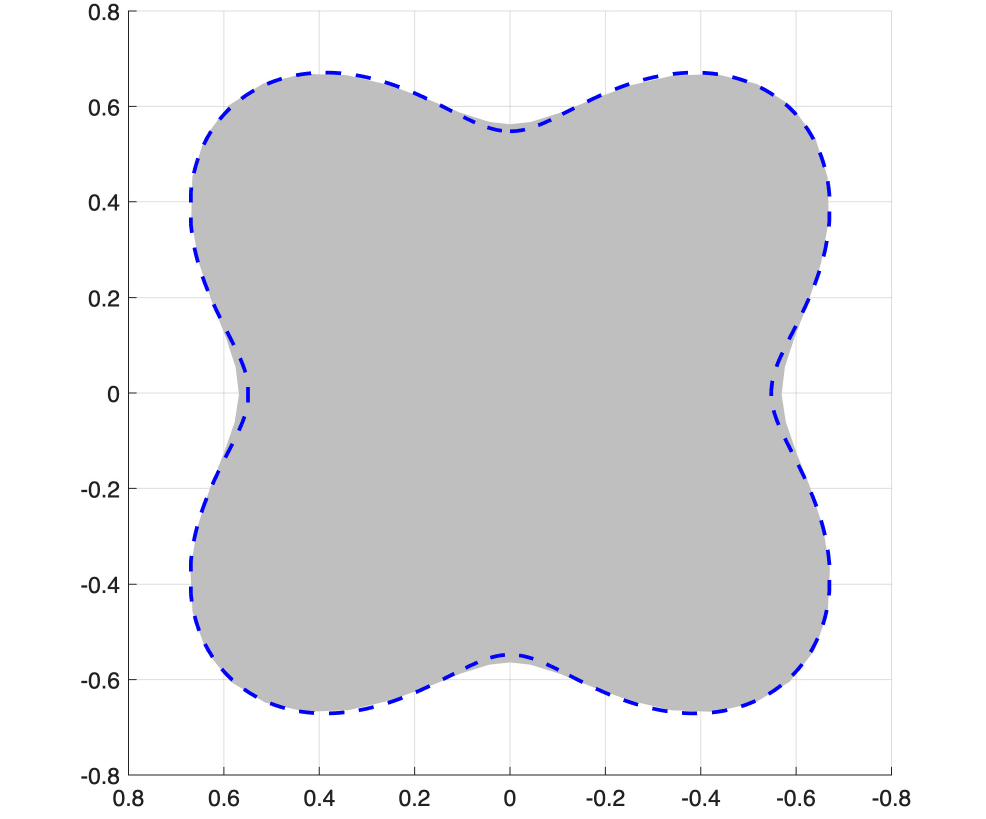}} &
			{\includegraphics[width=0.21\textwidth]{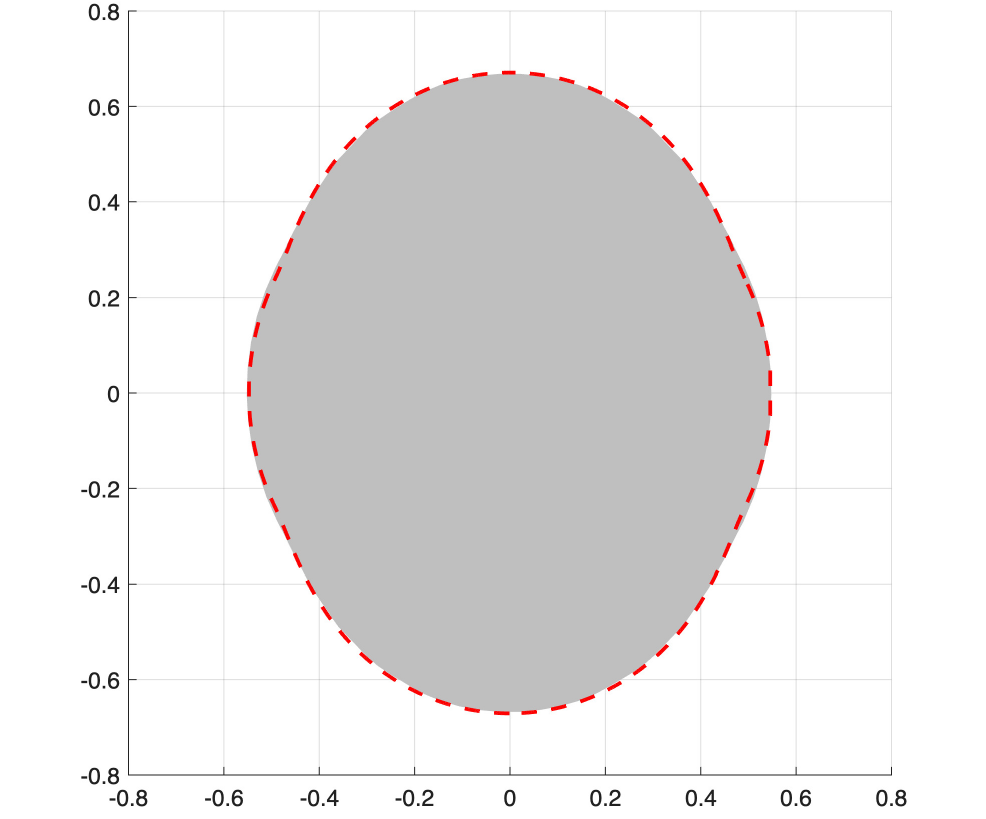}}  &
			{\includegraphics[width=0.21\textwidth]{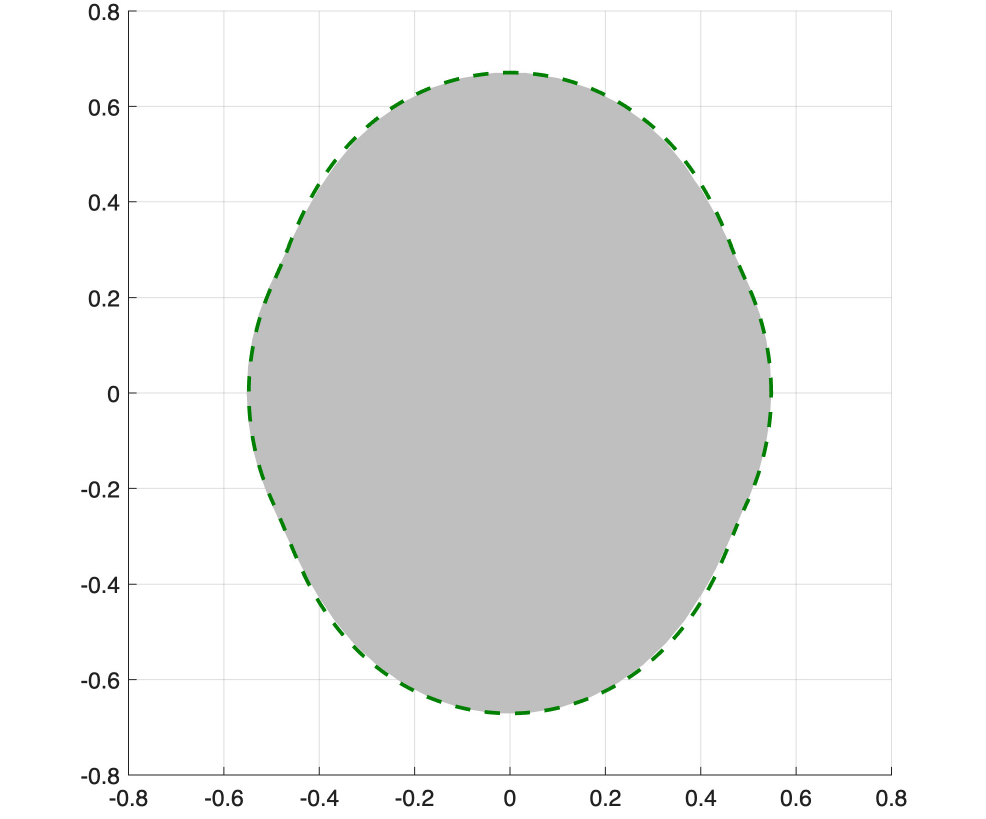}}
		\end{tabular} 
	}\\

	\caption{Reconstructions of a cushion-shaped obstacle with 10\% noise, where the incident point sources are located at $(0, 0, 5)^\top$ and $(0, 0, -5)^\top$.}\label{fig_cushion_ex3}
\end{figure}

\subsection{Inverse scattering problem}\label{inve ex}

We will present several numerical examples to verify the superiority and robustness of the proposed inverse approach.
In order to avoid the ``inverse crime", the scattered-field data $u_D^{\rm sc}(\pmb{x},t)$ on $\Gamma_B$ over $[0,T]$ are generated by solving the direct scattering problem for the incident wave $u^{\rm inc}(\pmb{x},t)$ emitted from the launch point $\pmb{x}_0$. Specifically, the forward problem is computed using the Galerkin-CQ method based on the single layer potential formulation described in Subsection~\ref{4.1111},
while the inverse problem is addressed by recovering the density via a regularization scheme based on homothetic surfaces.
To test stability, the noisy data $u^{\rm sc}_{\tilde{\delta}}(\pmb{x},t)$ is generated as follows
\[
u^{\rm sc}_{\tilde{\delta}}=(1+\tilde{\delta}\Theta)u_D^{\rm sc}
\]
with \(\Theta\) being normally distributed random variables supported in \([-1,1]\), and \({\tilde{\delta}}>0\) representing the relative noise level.

Similar to the iteration process in \cite{Zhao2022}, for fixed scaling factor $\rho\ge 0$, the update $\pmb{\varUpsilon}$ is computed by 
\[
\pmb \varUpsilon=\rho	\left( {\lambda}\tilde{\pmb{I}}+\Re\left( \pmb{B}_l^*\pmb{B}_l\right) \right) ^{-1}\Re\left( \pmb{B}_l^*\hat{\pmb{f}}_{D,l}^{\rm sc}\right) ,
\]
with $H^{\gamma}$ $(\gamma=1/2)$ penalty term, and the regularization parameters $\alpha$ in \eqref{field equation parama} and $\lambda$ are set to a constant value of $10^{-8}$ and $10^{-2}$.
In addition, we adopt the incident wave 
\begin{equation*}
	u^{\rm inc}(\pmb{x},t) =\left\{
	\begin{aligned}
		&\frac{1000\sin(4(t-|\pmb{x}-\pmb{x}_0|))e^{-1.2(t-|\pmb{x}-\pmb{x}_0|-2)^2}}{4\pi|\pmb{x}-\pmb{x}_0|}, &&t>|\pmb{x}-\pmb{x}_0|,\\
		&0,  &&t\le|\pmb{x}-\pmb{x}_0|,
	\end{aligned}
	\right.
\end{equation*}
where  $\pmb{x}_0$ denotes the point source location.

Throughout the numerical examples, unless otherwise specified, we set the final time $T=8$, the number of time steps $N=50$, the radius of the observation surface $R=1.5$, and the center $(0,0,0)^{\top}$. 
The scattered-field data are numerically generated at 800 observation points (i.e., $\tilde{n}=20$). 
The initial guess is taken as a sphere with center $\pmb{c}^{(0)}$ and radius $r^{(0)}$. 
We choose the scaling factor $\rho=0.5$, the geometric contraction factor $\varsigma=0.9$, and set the number of inner cycles per $l$ to be $loop=2$. 
The number of quadrature points on $\Gamma_{D'}$ is 512 (i.e., $n=15$), and the same number of discrete nodes is used on the boundary $\Gamma_D$. 
Finally, we set the maximum truncation number to be $M_{\rm max}=5$.

In addition, the complex wavenumbers ${\rm i}s_l$ may have a large imaginary part, which induces strong exponential damping of the corresponding fields. As a result, the measured data on $\Gamma_B$ become extremely small and contribute little to improving the reconstruction \cite{CQ}. Therefore, in our implementation, we skip those wavenumbers for which the data magnitude is below a prescribed threshold, and continue the iteration with the remaining complex wavenumbers; this provides a mild additional reduction in computational cost, whereas the substantial efficiency gain comes from the homothetic-surface technique.

In all figures, the shaded regions represent the projections of the reconstructions, whereas the colored dashed curves denote the true obstacle boundaries in the corresponding projection directions. This visualization allows a direct comparison between the exact geometry and the recovered shape.

\medskip
{\noindent\bf Example 1: Inverse obstacle scattering with a single launch position.}
\medskip

We investigate the inverse obstacle scattering problem of reconstructing both the shape and location of a rigid obstacle from time-domain scattered field data. 
In Figures~\ref{fig_pinchedball_ex1} and \ref{fig_cushion_ex1}, the true obstacles are shown in the subfigure~(a). 
The reconstruction results corresponding to noise levels of \(1\%\), \(5\%\), and \(10\%\) are presented in subfigures~(b)-(d) for the pinched ball-shaped and cushion-shaped obstacles. Even as the noise increases, the reconstruction remains stable and retains the main geometric features.

\medskip
{\noindent\bf Example 2: Inverse obstacle scattering problem with multiple launch positions.}
\medskip

To improve the reconstruction accuracy, we investigate multiple launch positions of the incident wave. In this example, the reconstructions are carried out using scattered field data contaminated with \(10\%\) noise for two and four incident point sources. 
For comparison, the reconstructions with 10\% noise for a single launch position are shown in Figures \ref{fig_pinchedball_ex1}(d) and \ref{fig_cushion_ex1}(d), with all other parameters remaining unchanged. Figures \ref{fig_pinchedball_ex2} and \ref{fig_cushion_ex2} present the reconstruction results for the pinched ball-shaped and cushion-shaped obstacles, respectively, using multiple launch positions of the incident wave. From these figures, we observe that the shape and location of the obstacle are more accurately reconstructed using two and four launch positions, which aligns with the expected physical behavior.

\medskip
{\noindent\bf Example 3: Inverse obstacle scattering problem with different parameters.}
\medskip

Under a fixed noise level of \(10\%\) and two incident point sources located at \((0,0,5)^\top\) and \((0,0,-5)^\top\), figures \ref{fig_pinchedball_ex3} and \ref{fig_cushion_ex3} show the reconstructions of the pinched ball-shaped obstacle and the cushion-shaped obstacle with different parameters. With all other parameters remaining the same, subfigures (a) and (b) demonstrate variations arising from distinct initial guesses, while subfigures (b) and (c) show the reconstruction results with different observation radii. The reconstruction results remain stable and effective across different initial guesses and observation radii, demonstrating the robustness of the algorithm under varied parameter settings.

\begin{figure}[!htpb]
		\centering 
		\subfigure[True shape with multiple views.]{
			\begin{tabular}{cccc}
				{\includegraphics[width=0.21\textwidth]{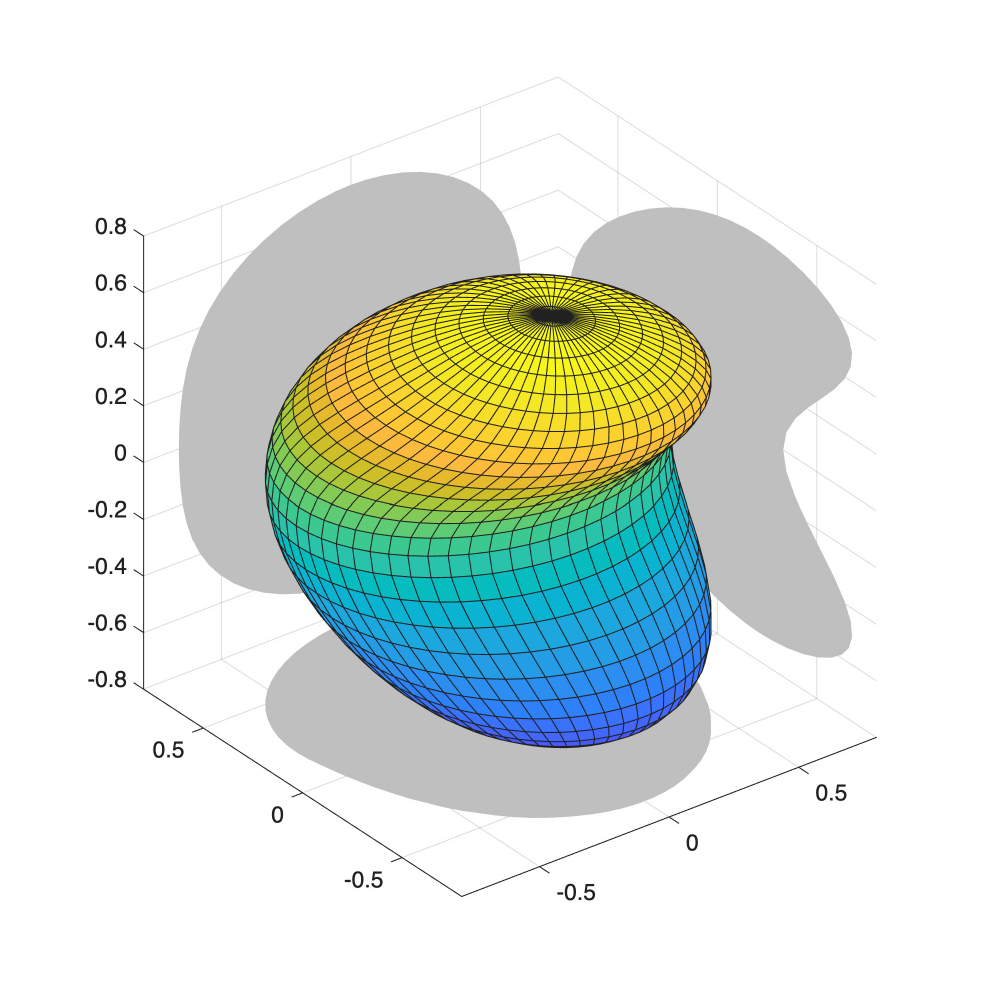}}  &
				{\includegraphics[width=0.21\textwidth]{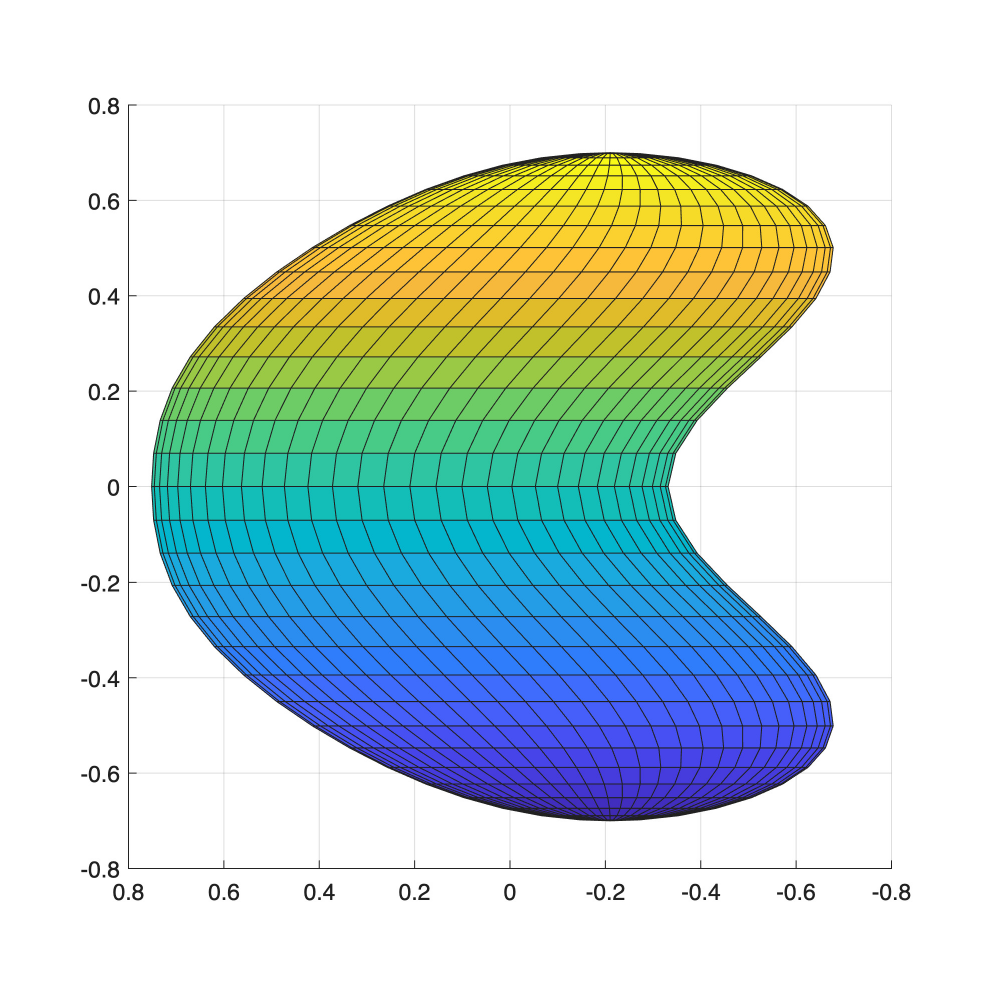}} &
				{\includegraphics[width=0.21\textwidth]{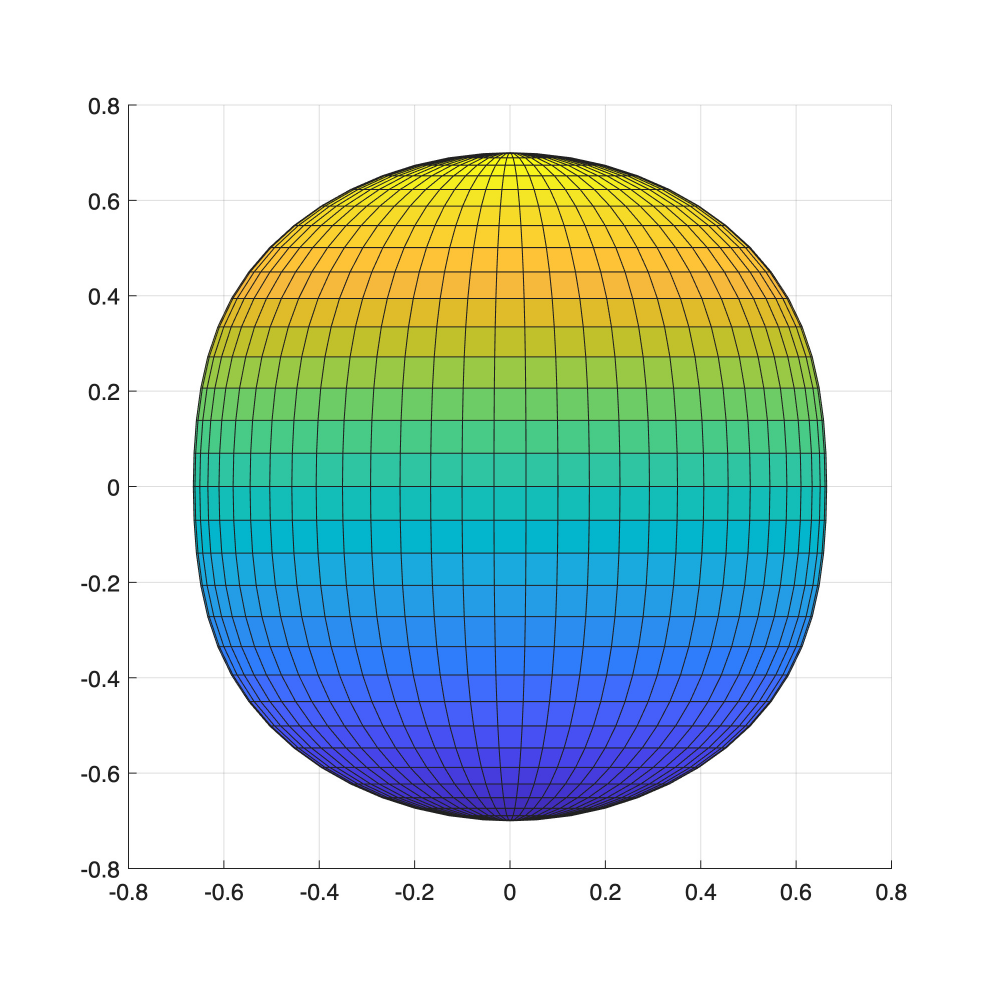}}  &
				{\includegraphics[width=0.21\textwidth]{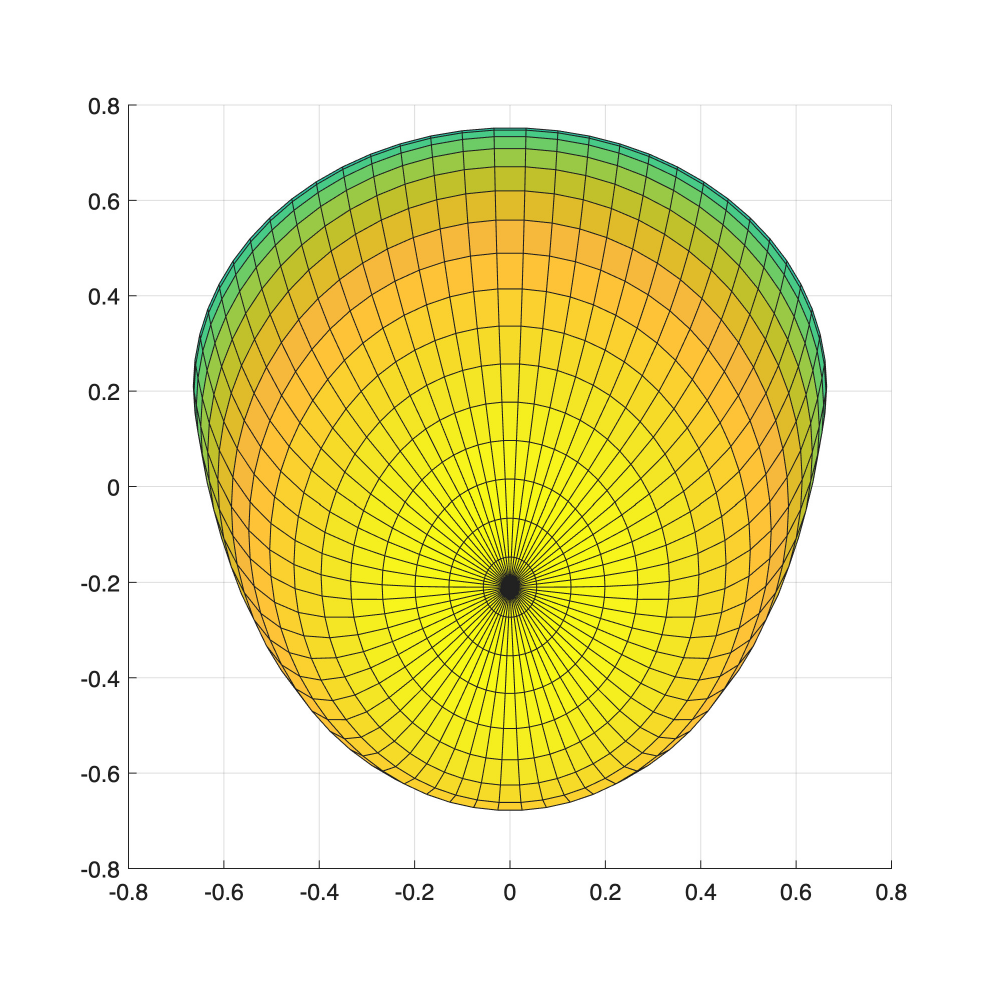}}
			\end{tabular} 
		}\\
		
		\subfigure[Reconstruction with 1$\%$ noise. The incident point sources are located at $(5, 0, 0)^\top$ and $(-5, 0, 0)^\top$.]{
			\begin{tabular}{cccc}
				{\includegraphics[width=0.21\textwidth]{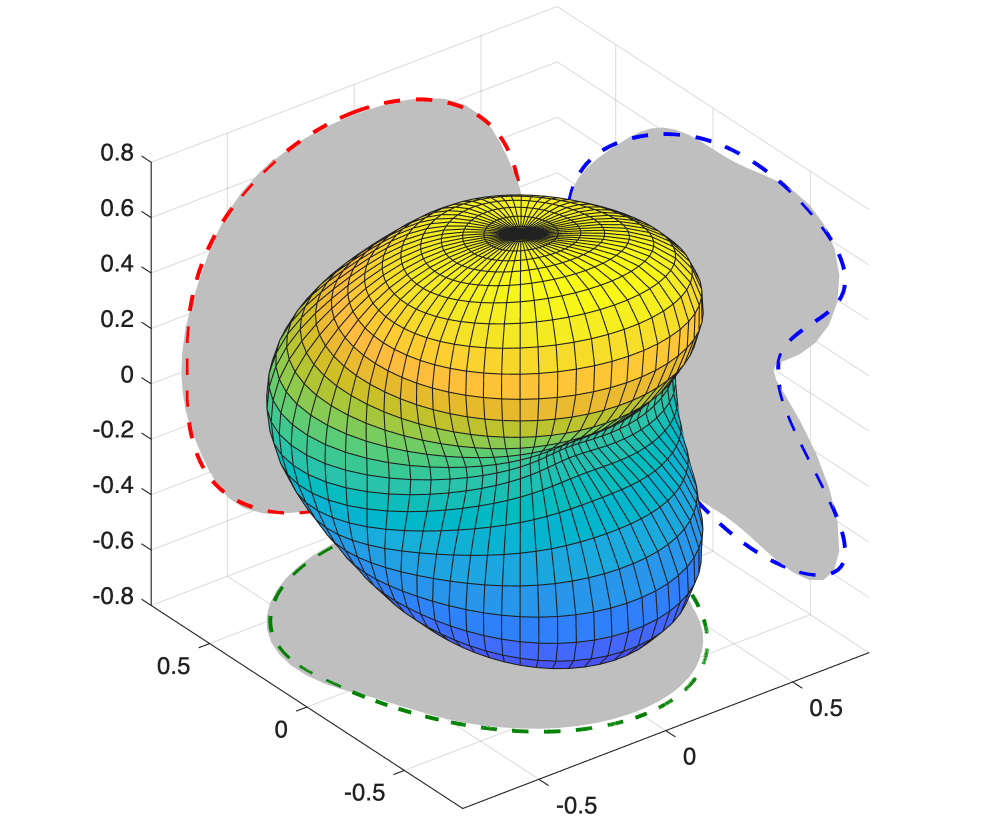}}  &
				{\includegraphics[width=0.21\textwidth]{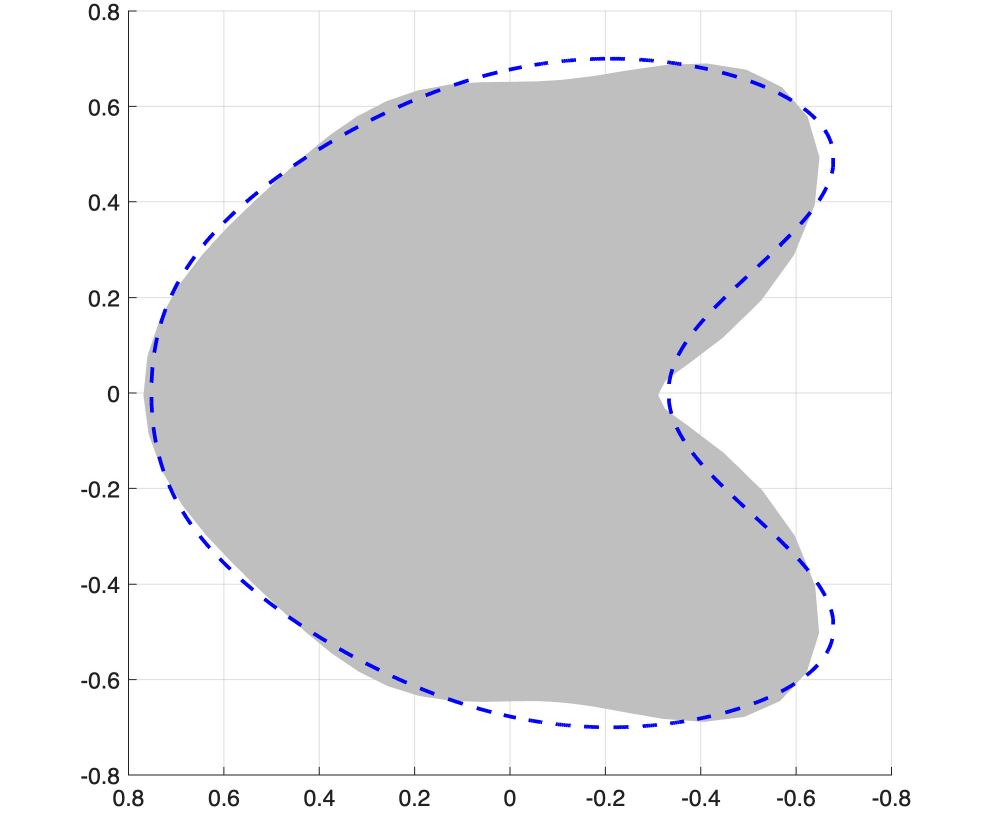}} &
				{\includegraphics[width=0.21\textwidth]{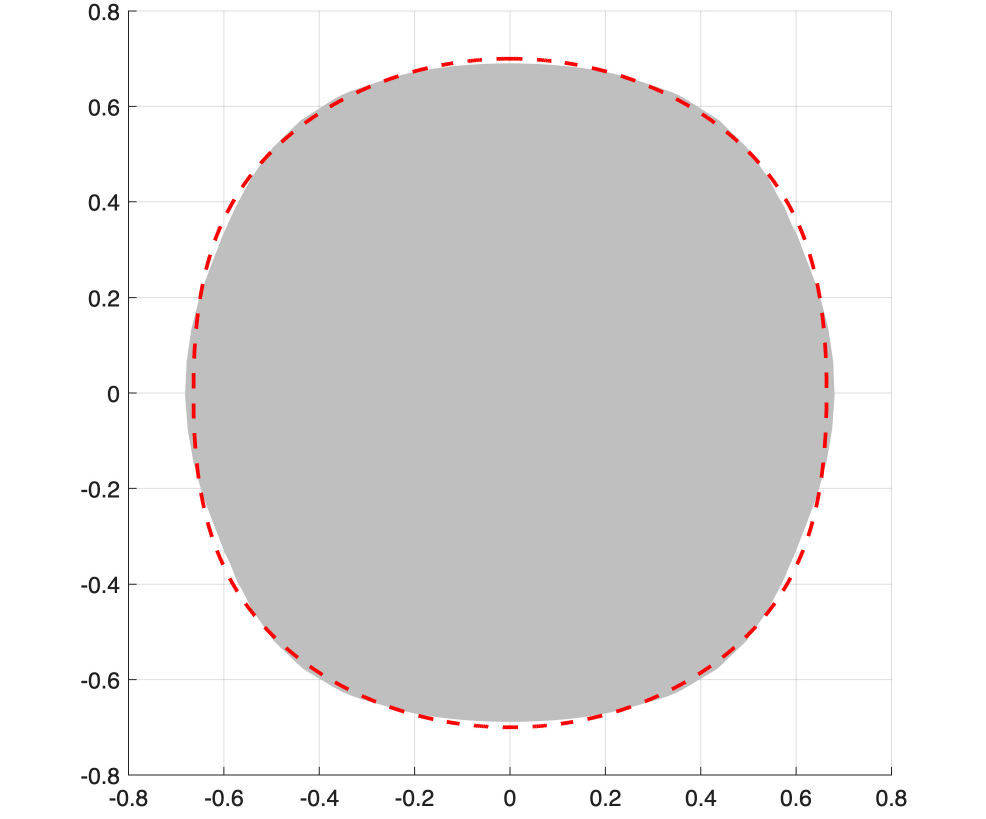}}  &
				{\includegraphics[width=0.21\textwidth]{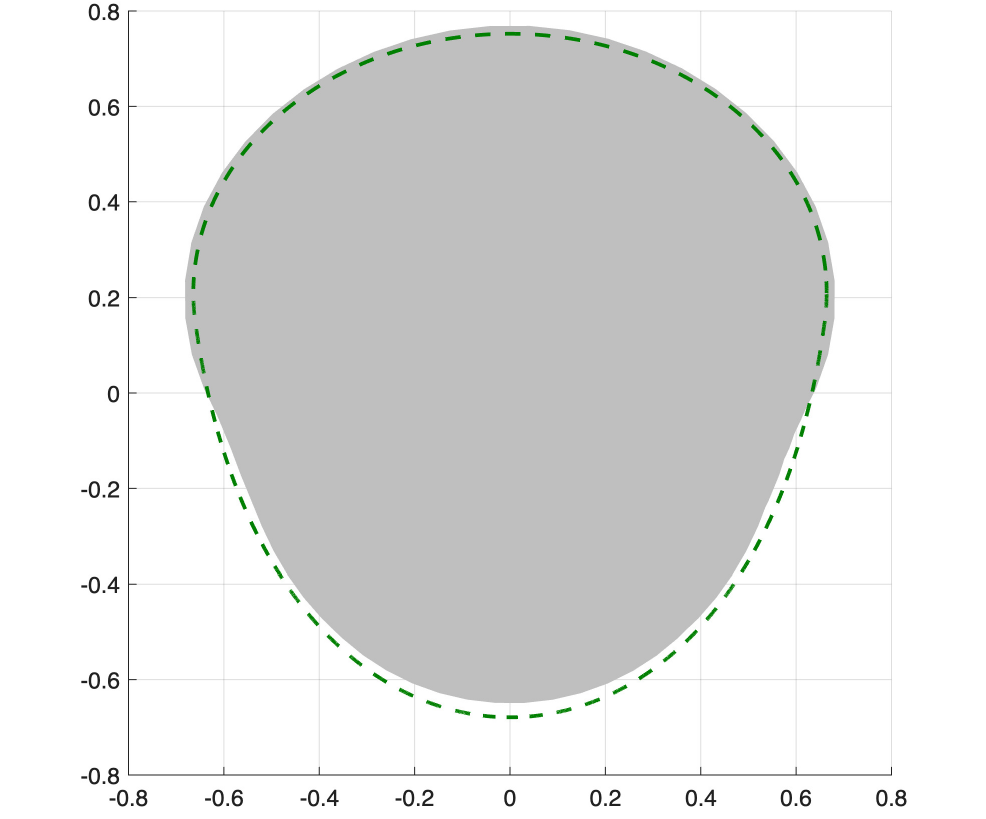}}
			\end{tabular} 
		}\\

		\caption{Reconstructions of a bean-shaped obstacle with 1\% noise. The initial guess is a sphere with  $\pmb c^{(0)} = (-0.2, 0.3, -0.4)^\top$ and  $r^{(0)} = 0.7$.}\label{bean}
	\end{figure}

\vspace{1.5ex}
{\noindent\bf Example 4: Reconstruction of non-star-shaped obstacle.}
\vspace{1.5ex}

Figure~\ref{bean} shows the reconstruction of a bean-shaped obstacle, which serves as a representative example of non-star-shaped geometries. 
The parametrization of the bean-shaped obstacle is given in Table \ref{obscatle1}.

Subfigure~(a) shows the true shape from multiple views, clearly highlighting the nonconvex features of the obstacle. 
Subfigure~(b) presents the reconstruction from time-domain scattered field data generated by two point sources located at $(5,0,0)$ and $(-5,0,0)$ with $1\%$ noise. 
The proposed approach successfully recovers both the shape and location of the bean-like surface, illustrating its effectiveness in handling more complex obstacle configurations.

\begin{figure}[!htbp]
\centering 

\subfigure[ True shape of complex obstacle.]{
	\begin{tabular}{cccc}
		{\includegraphics[width=0.21\textwidth]{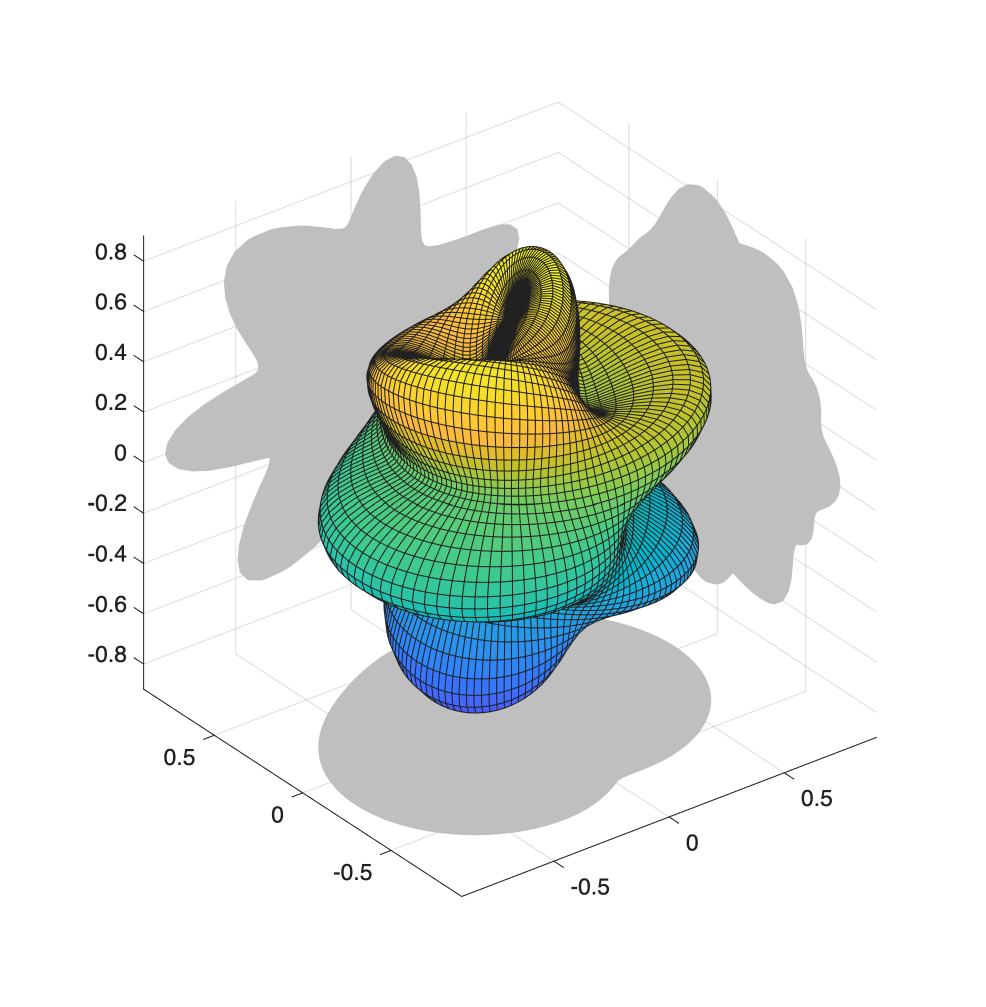}} &
		{\includegraphics[width=0.21\textwidth]{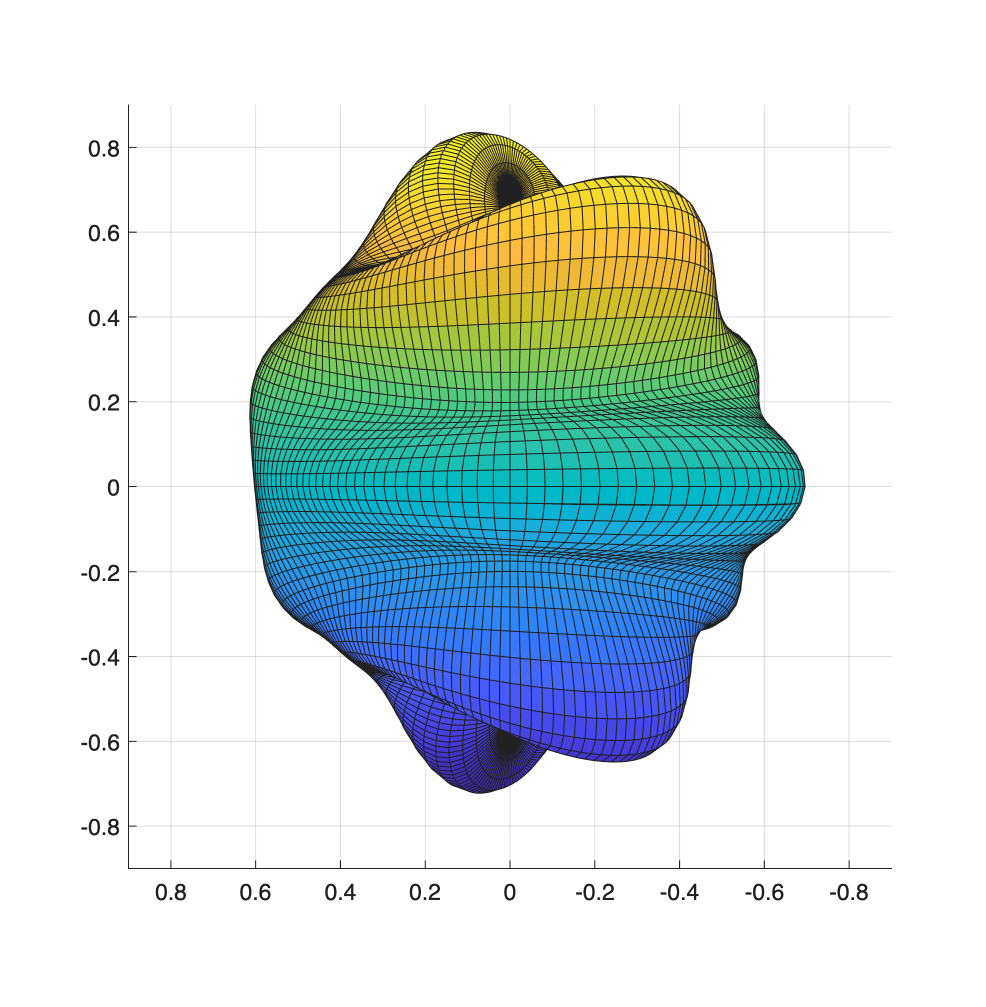}} &
		{\includegraphics[width=0.21\textwidth]{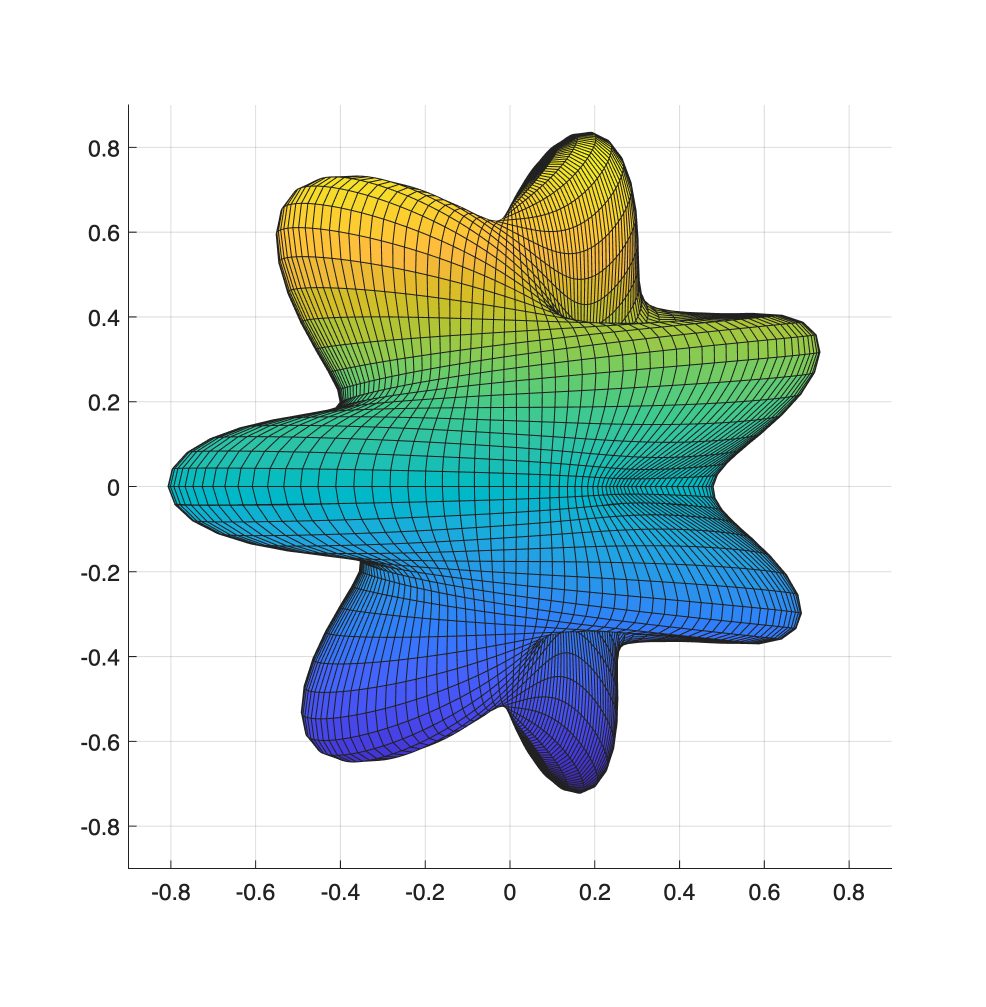}} &
		{\includegraphics[width=0.21\textwidth]{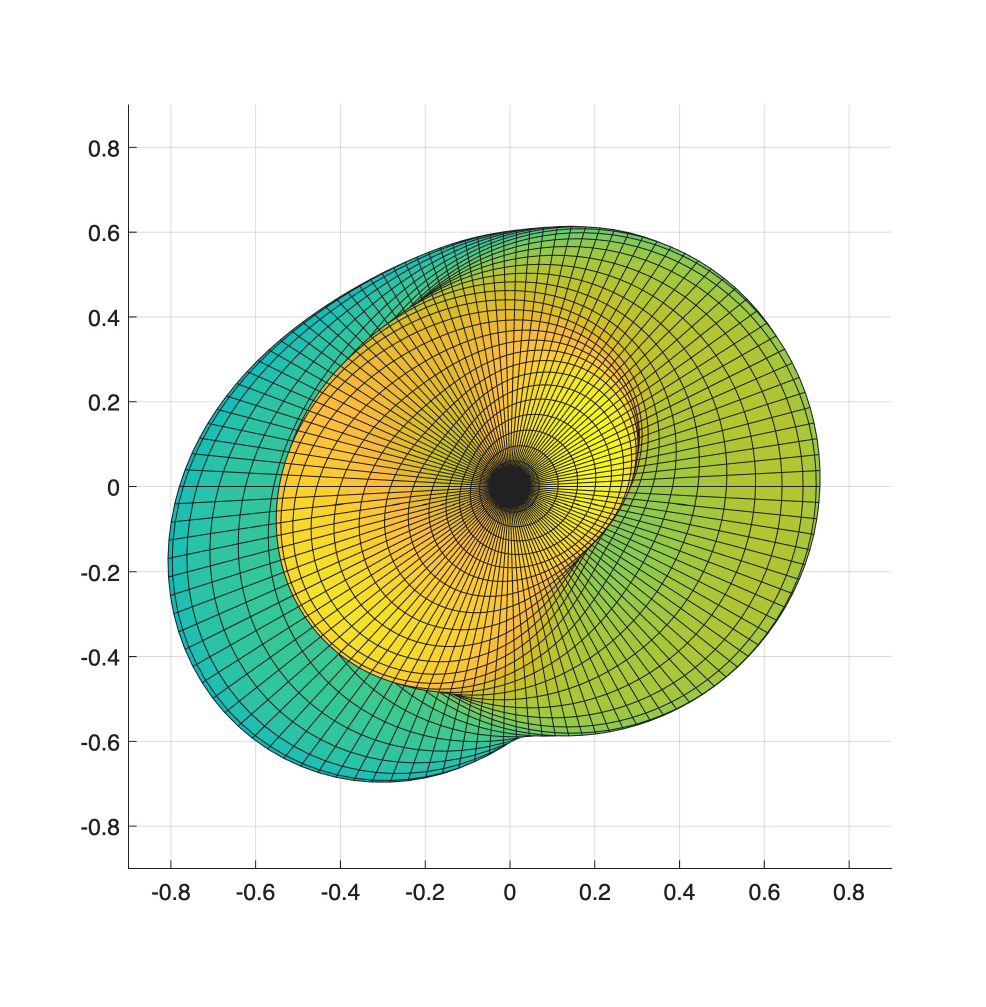}}
	\end{tabular}
}\\	[-1.2ex]

\subfigure[Reconstruction with 1\% noise.]{
	\begin{tabular}{cccc}
		{\includegraphics[width=0.21\textwidth]{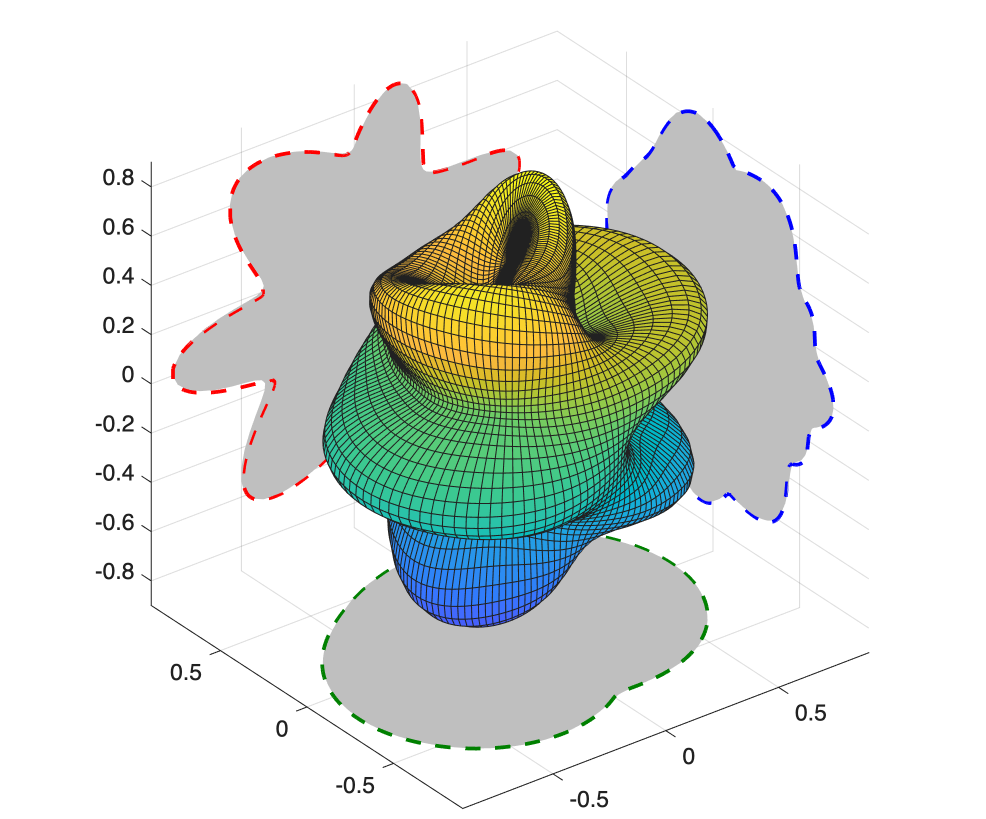}} &
		{\includegraphics[width=0.21\textwidth]{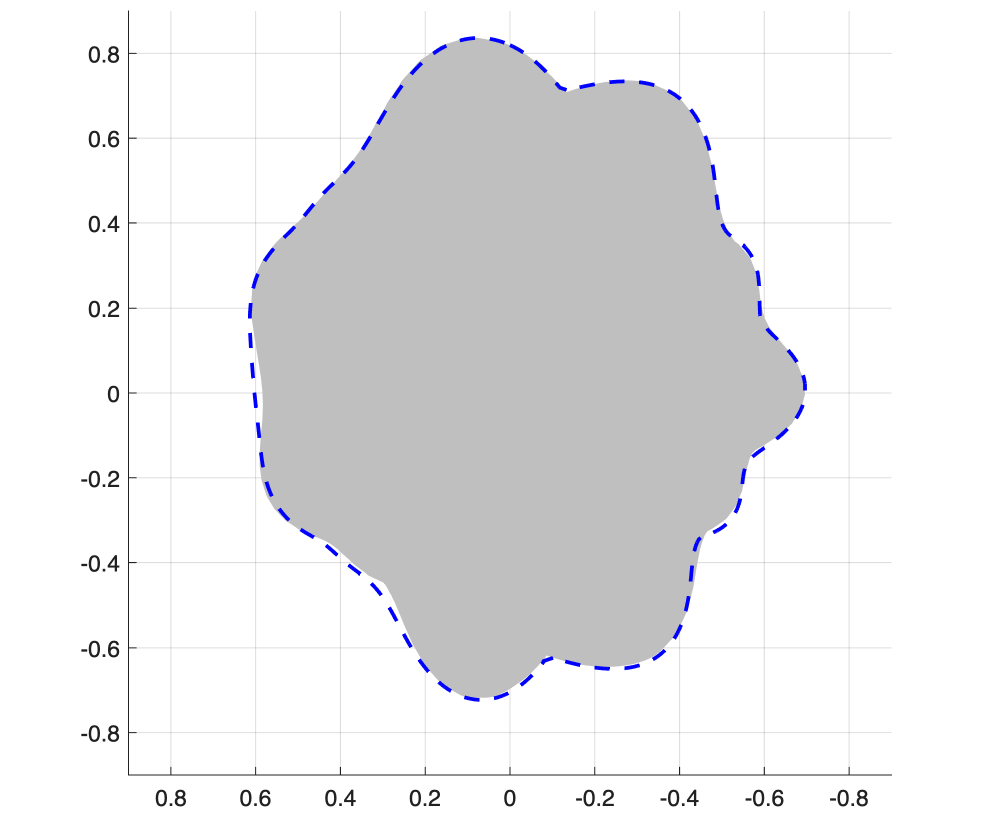}} &
		{\includegraphics[width=0.21\textwidth]{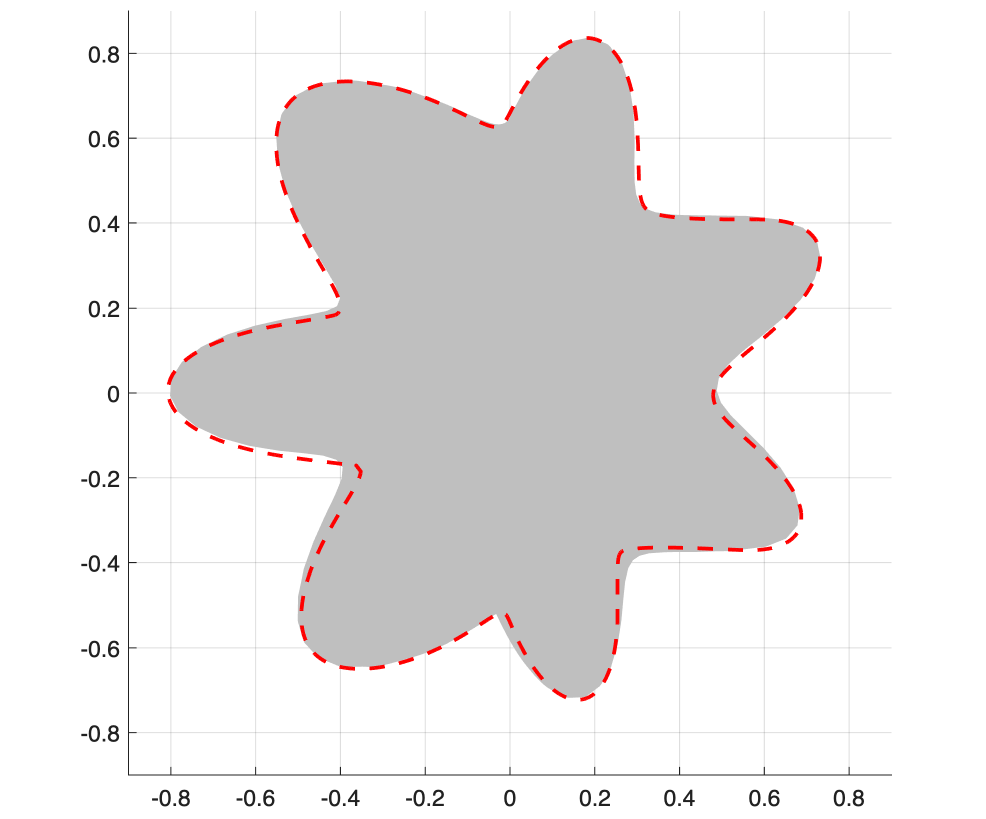}} &
		{\includegraphics[width=0.21\textwidth]{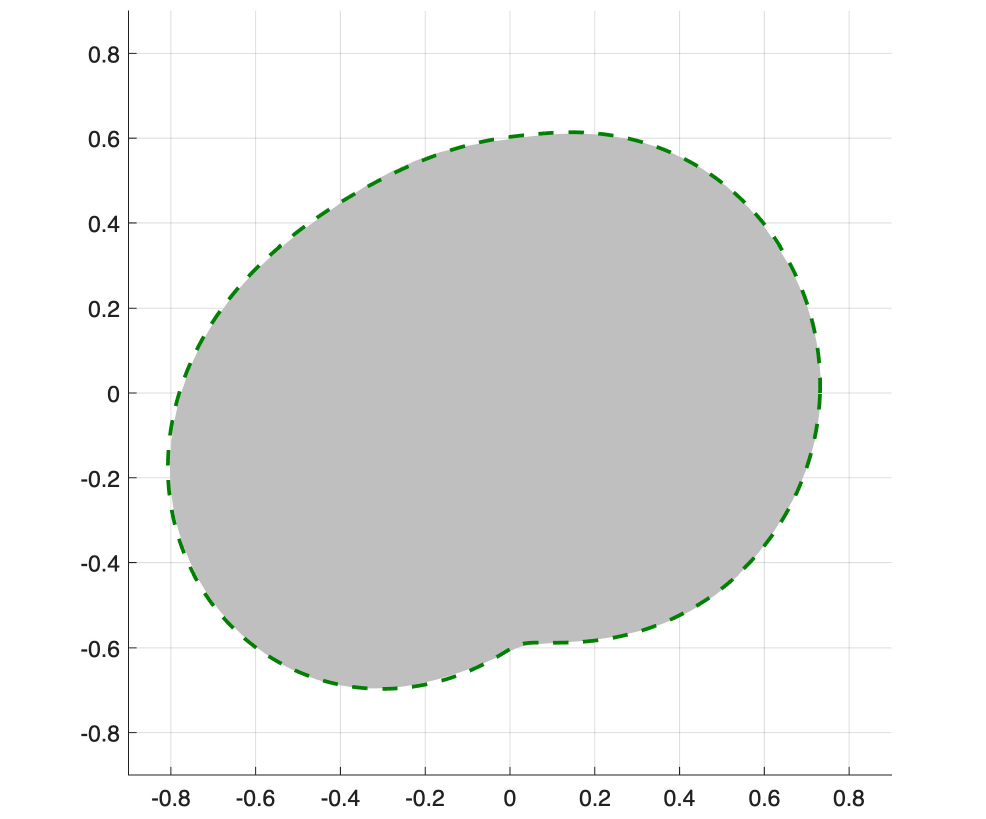}}
	\end{tabular} 
}\\	[-1.2ex]

\subfigure[Reconstruction with 5\% noise.]{
	\begin{tabular}{cccc}
		{\includegraphics[width=0.21\textwidth]{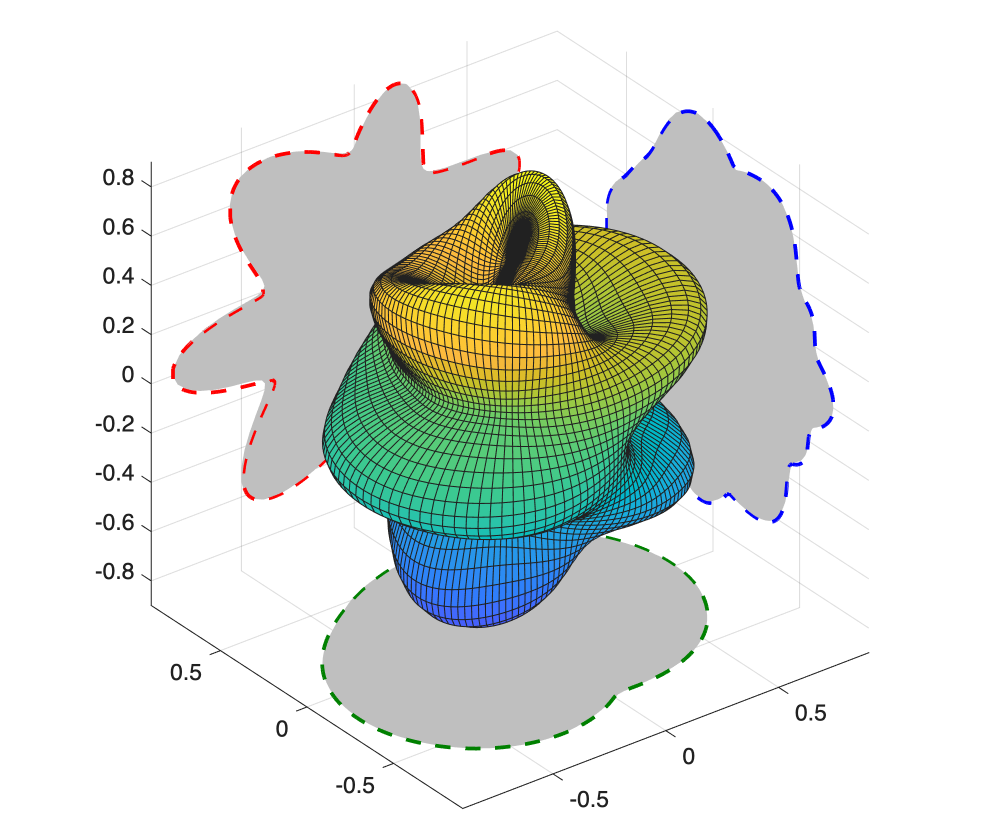}} &
		{\includegraphics[width=0.21\textwidth]{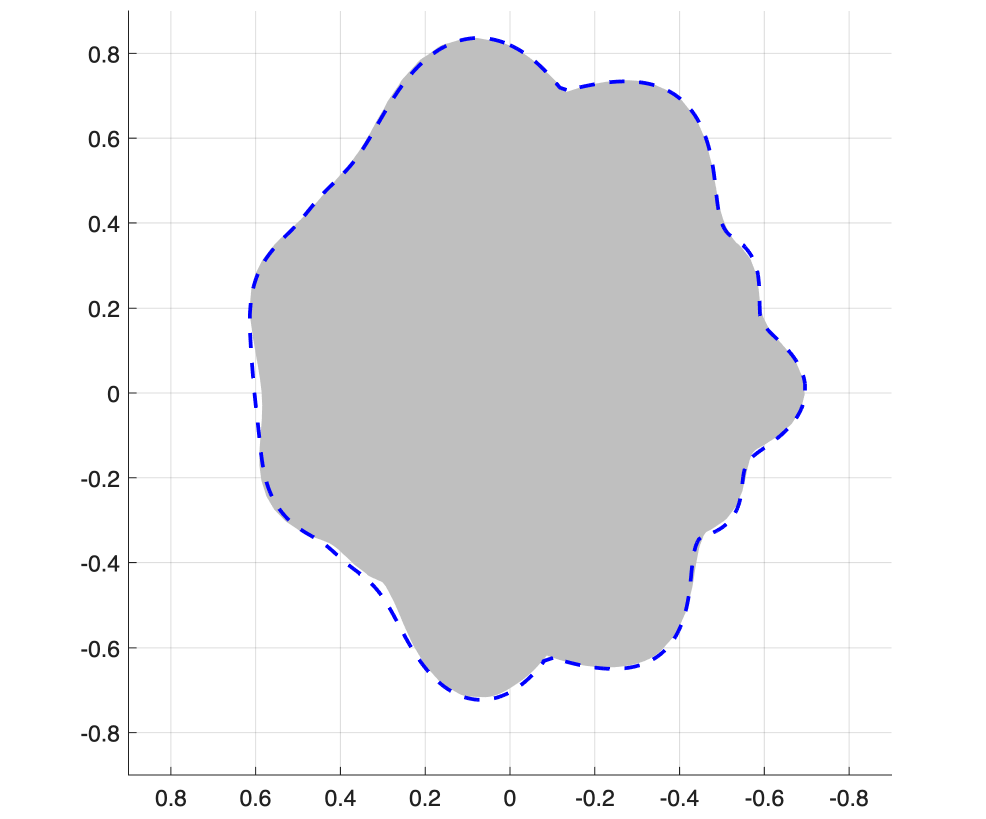}} &
		{\includegraphics[width=0.21\textwidth]{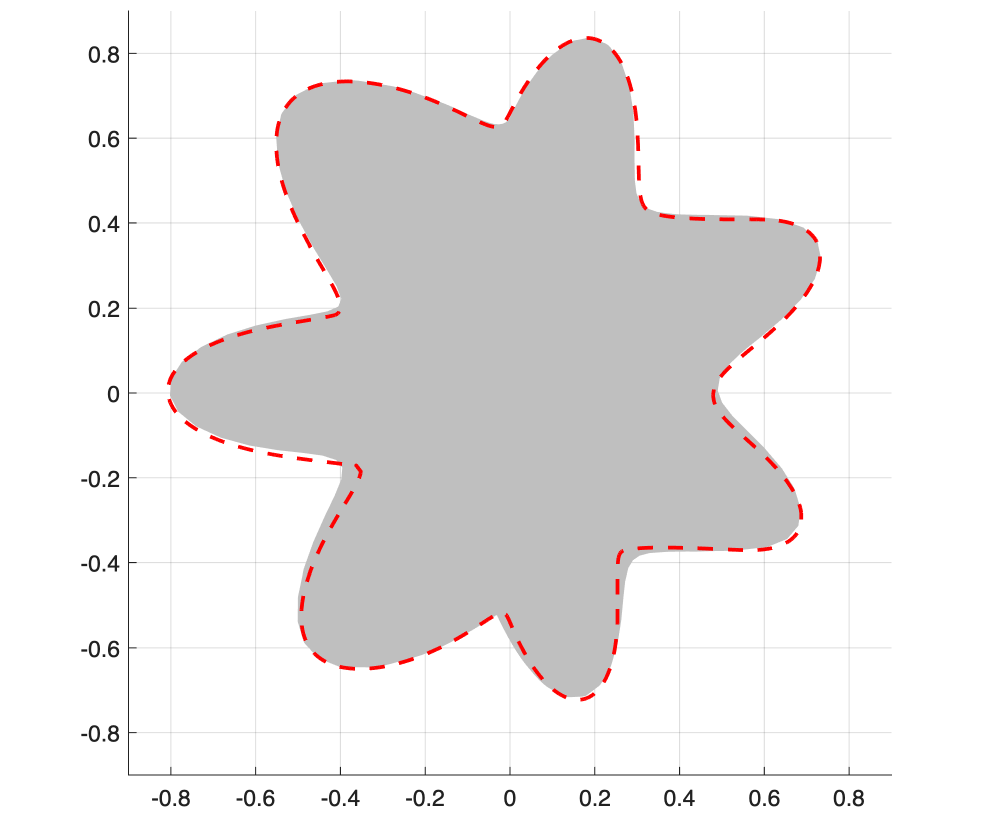}} &
		{\includegraphics[width=0.21\textwidth]{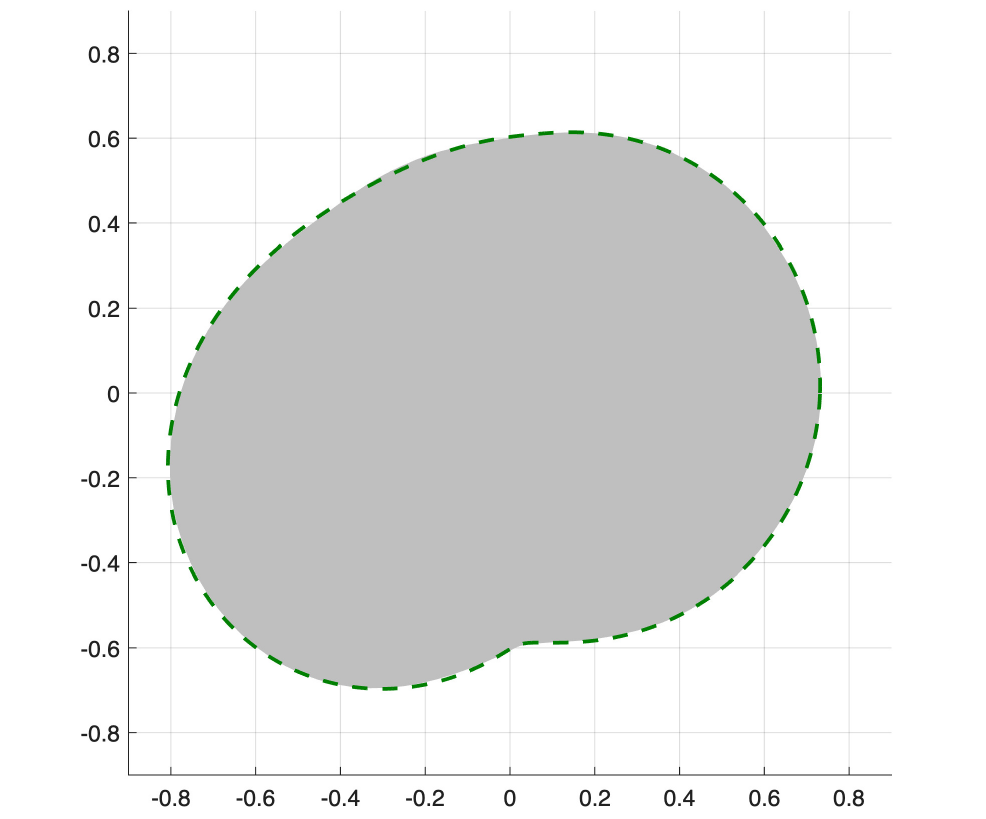}}
	\end{tabular} 
}\\	[-1.2ex]

\subfigure[Reconstruction with 10\% noise.]{
	\begin{tabular}{cccc}
		{\includegraphics[width=0.21\textwidth]{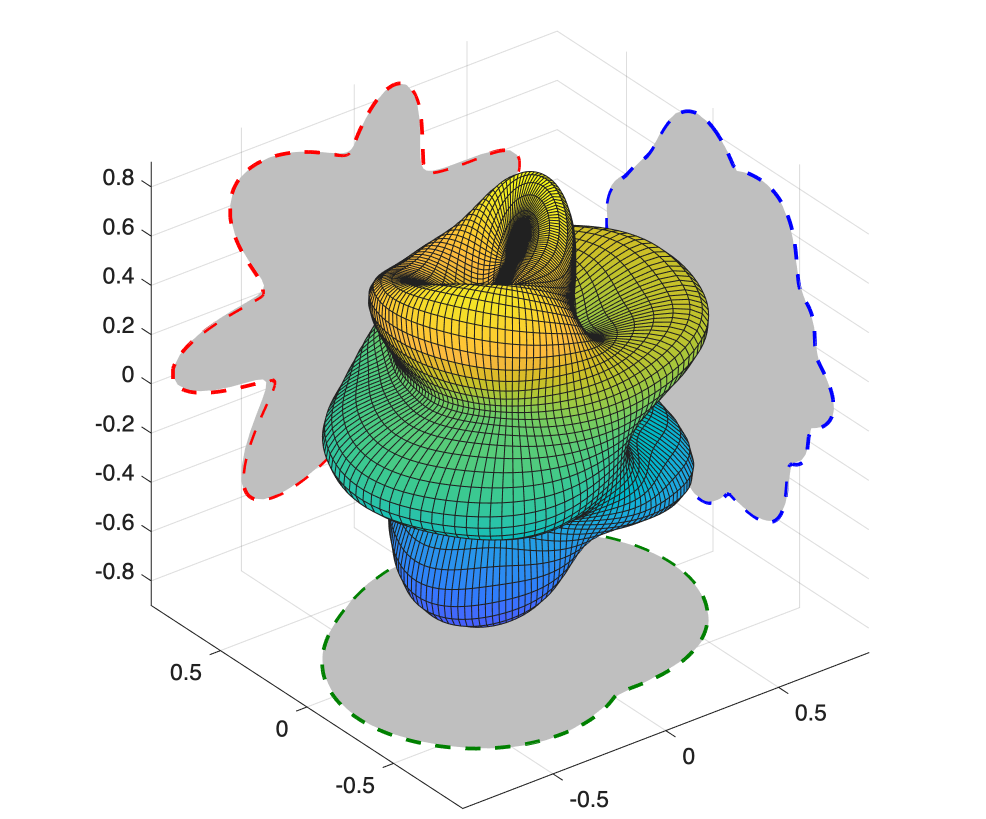}} &
		{\includegraphics[width=0.21\textwidth]{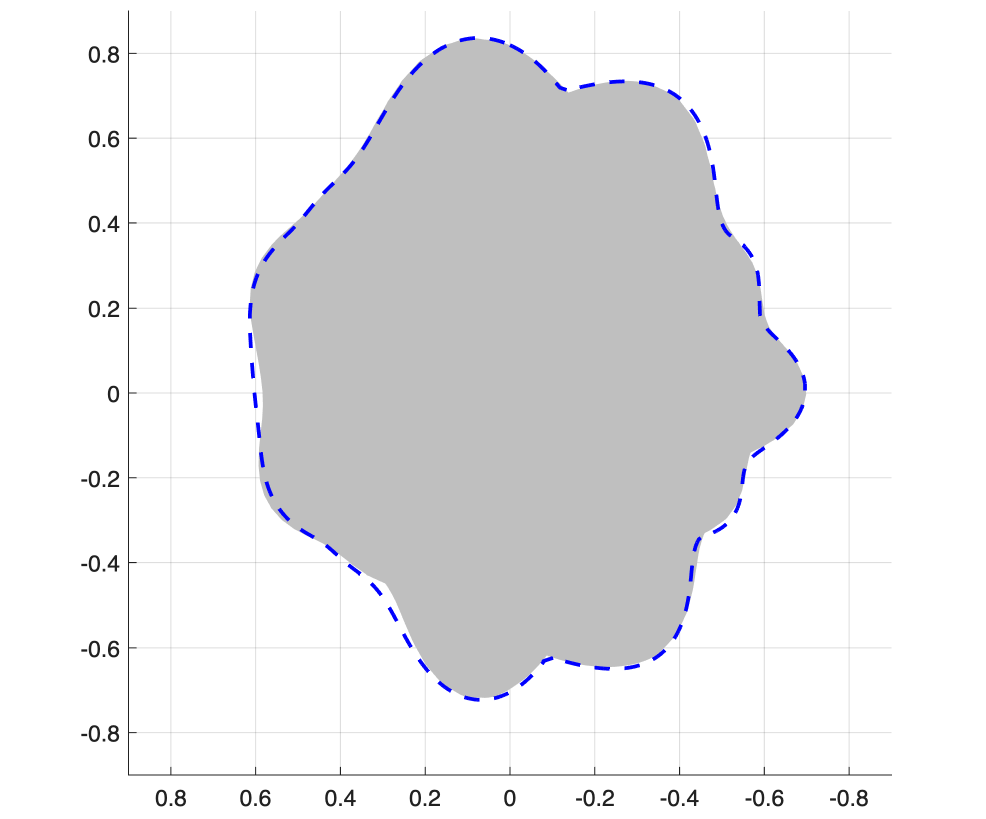}} &
		{\includegraphics[width=0.21\textwidth]{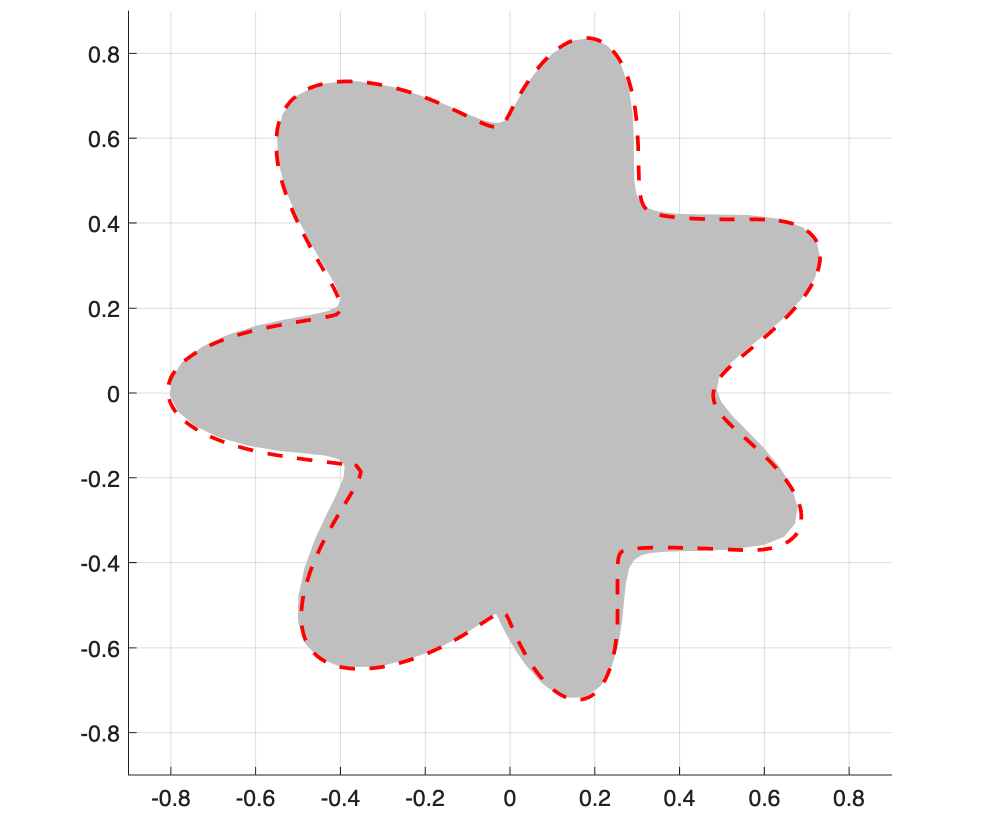}} &
		{\includegraphics[width=0.21\textwidth]{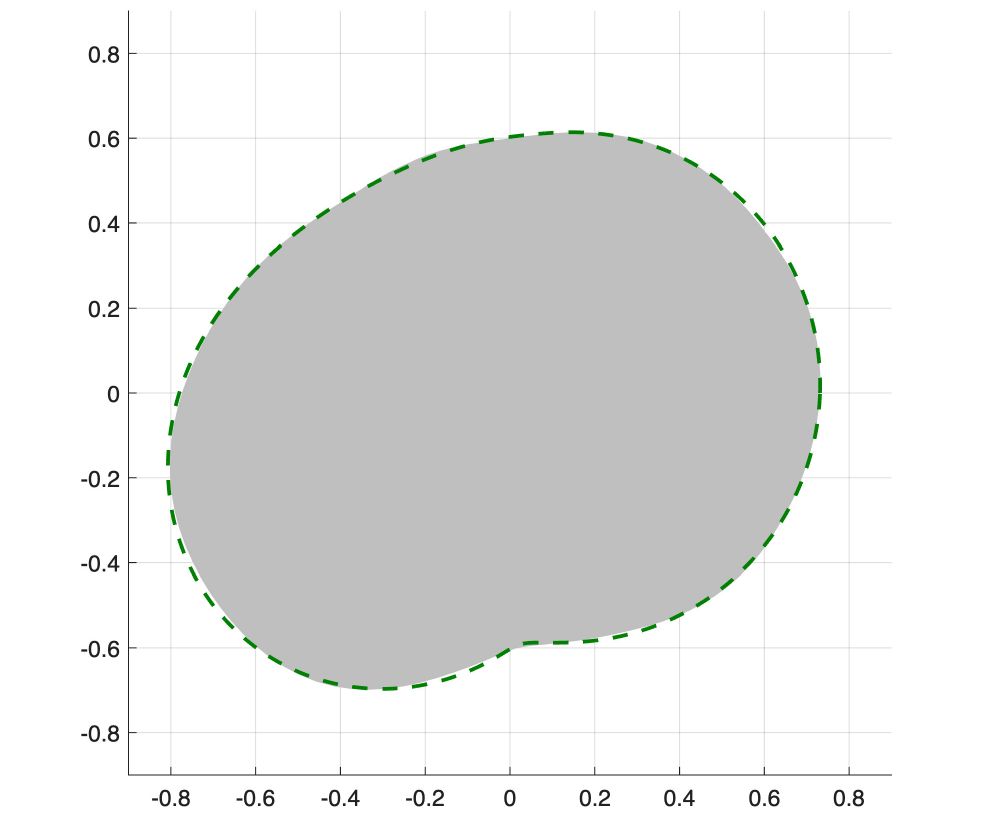}}
	\end{tabular} 
}\\	[-1.2ex]

\caption{Reconstruction of a complex obstacle from time-domain scattered field data generated by four point sources located at \((0,0,5)^\top\), \((0,0,-5)^\top\), \((5,0,0)^\top\), and \((-5,0,0)^\top\), under different noise levels. $\pmb c^{(0)} = (-0.3, 0.2, -0.5)^\top$, $r^{(0)} = 0.3$.}\label{fig_ex7}
\end{figure}

\begin{figure}[!htbp]
\centering 

\subfigure[Reconstruction with 1\% noise.]{
	\begin{tabular}{cccc}
		{\includegraphics[width=0.21\textwidth]{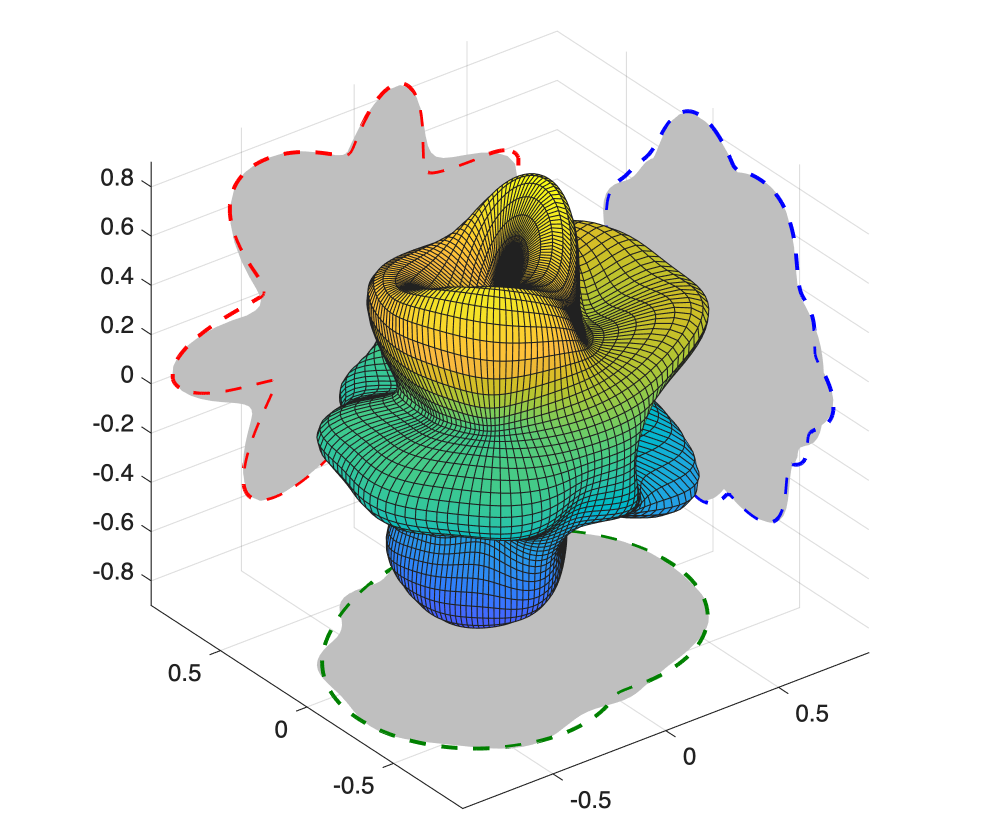}} &
		{\includegraphics[width=0.21\textwidth]{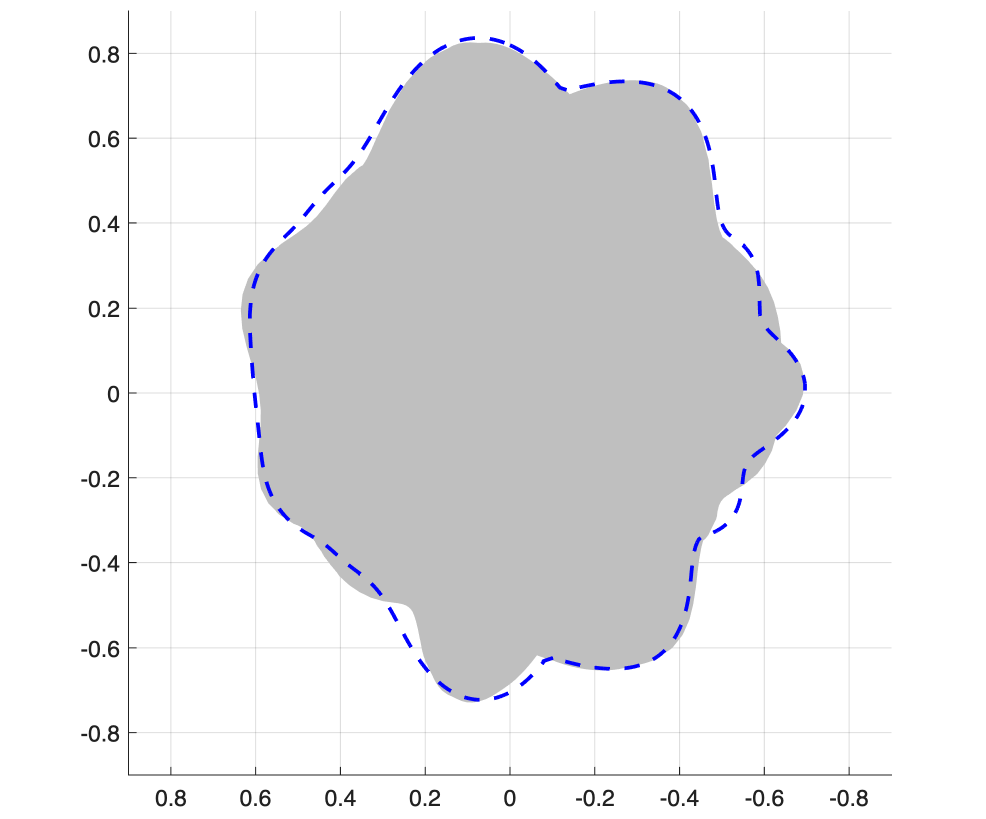}} &
		{\includegraphics[width=0.21\textwidth]{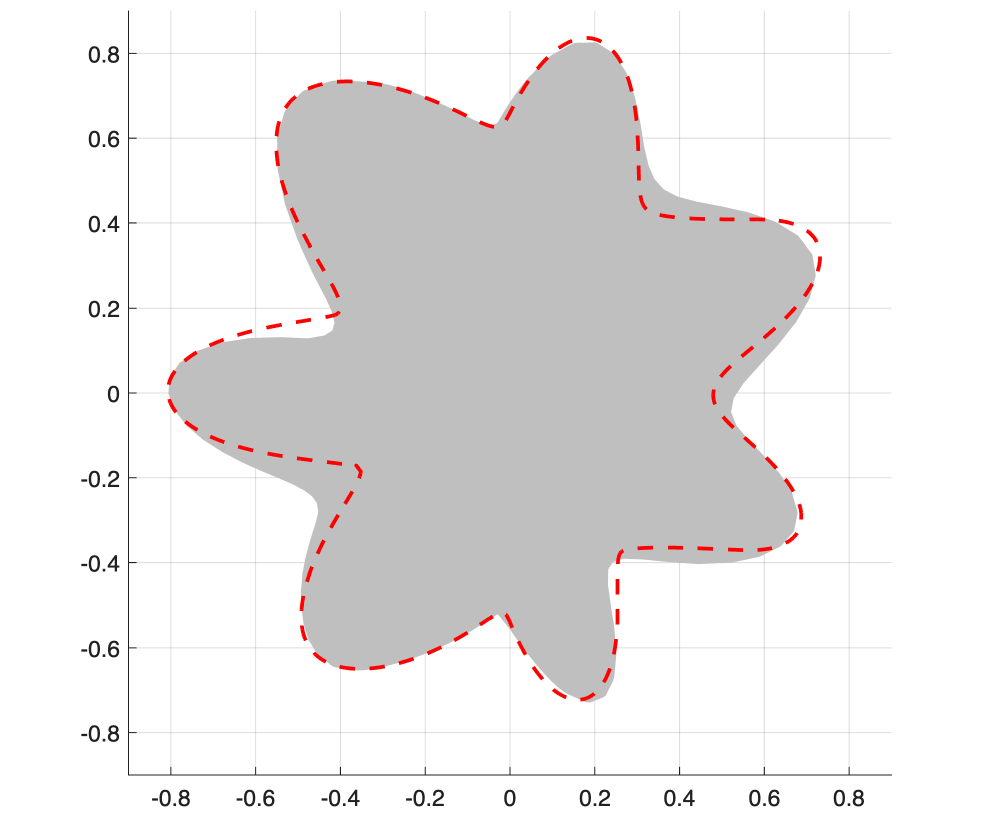}} &
		{\includegraphics[width=0.21\textwidth]{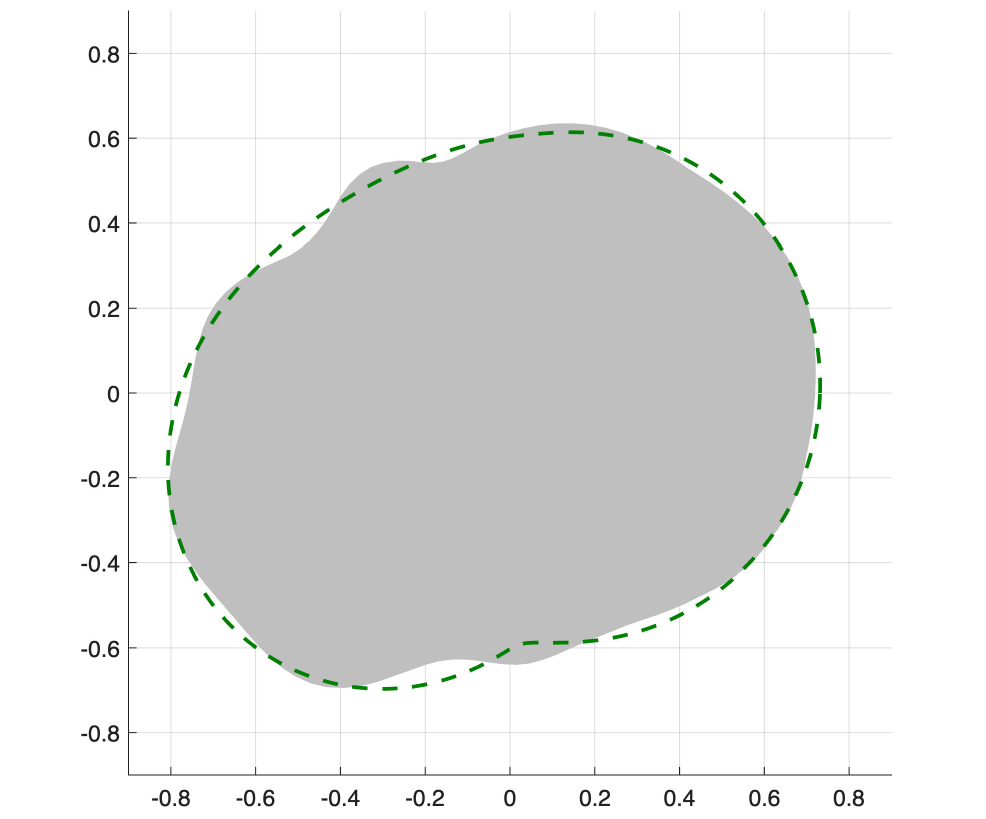}}
	\end{tabular} 
}\\	[-1.2ex]

\caption{Reconstruction of a complex obstacle obtained by directly increasing the truncation number from \(M=0\) to \(M=8\) without incremental truncation technique. The time-domain scattered field data are generated by four point sources located at \((0,0,5)^\top\), \((0,0,-5)^\top\), \((5,0,0)^\top\), and \((-5,0,0)^\top\). $\pmb c^{(0)} = (-0.3, 0.2, -0.5)^\top$, $r^{(0)} = 0.3$.}\label{fig_ex7M8}
\end{figure}

\begin{figure}[!htbp]
	\centering
	\begin{tabular}{ccc}
		\subfigure[Initial surface]{\includegraphics[width=0.3\textwidth]{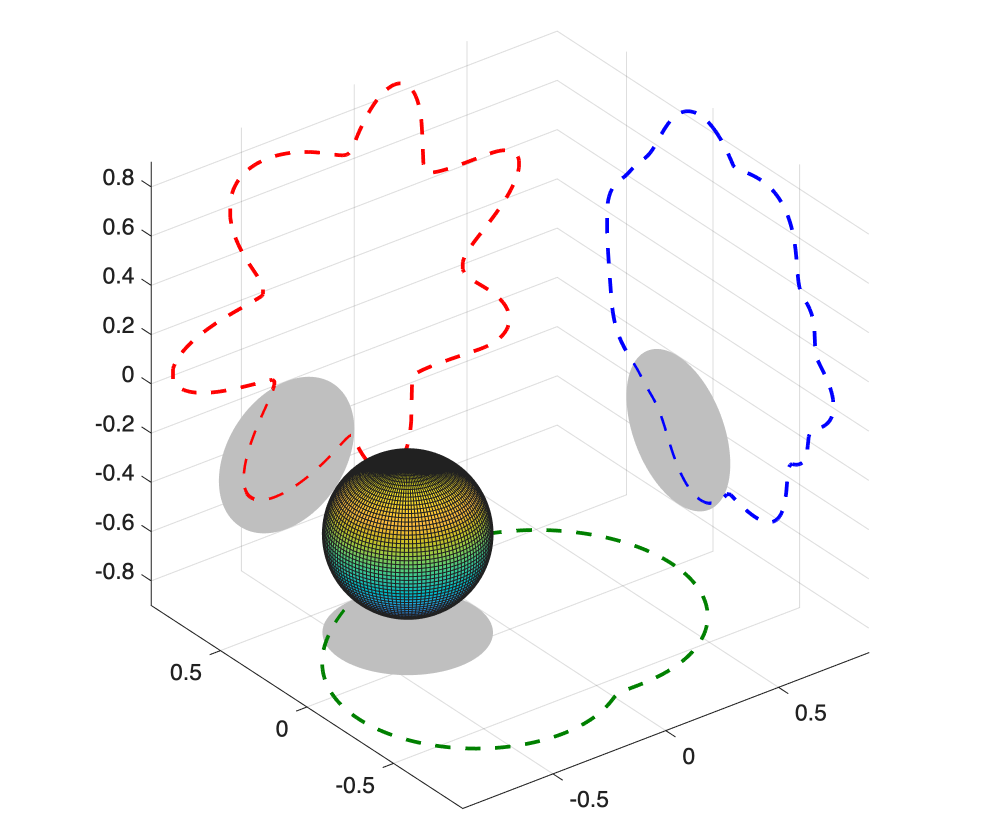}} &
		\subfigure[$M=0$]{\includegraphics[width=0.3\textwidth]{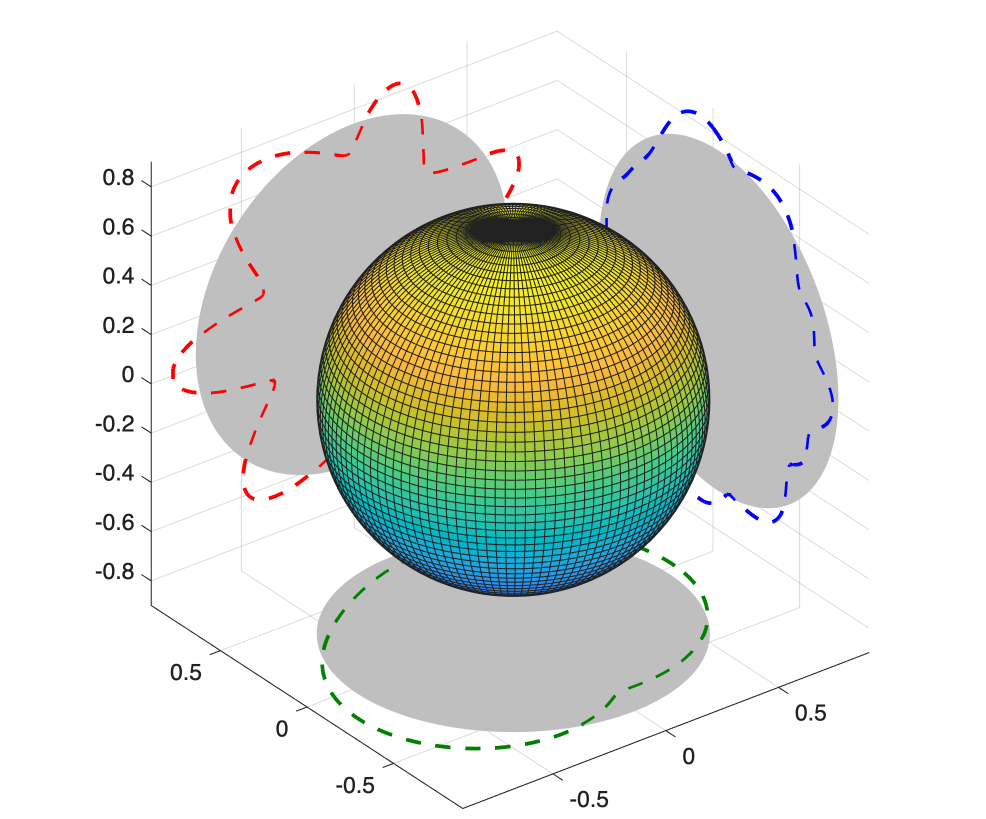}} &
		\subfigure[$M=2$]{\includegraphics[width=0.3\textwidth]{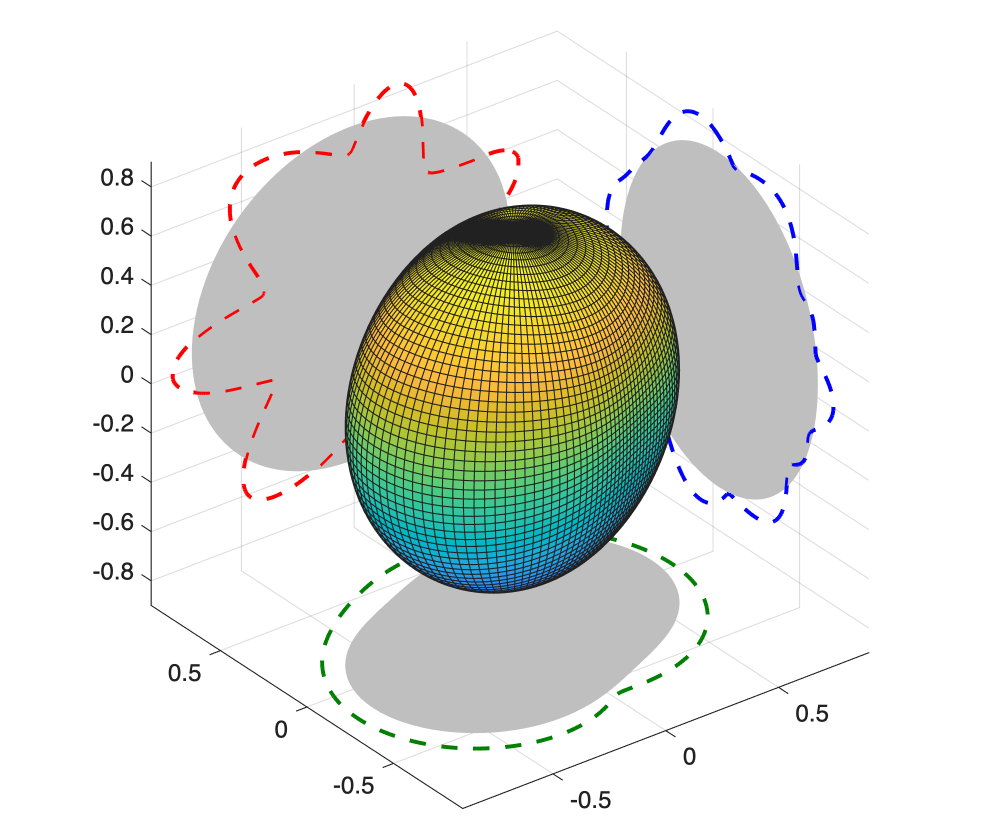}} 
	\end{tabular}
	
	\begin{tabular}{ccc}
		\subfigure[$M=4$]{\includegraphics[width=0.3\textwidth]{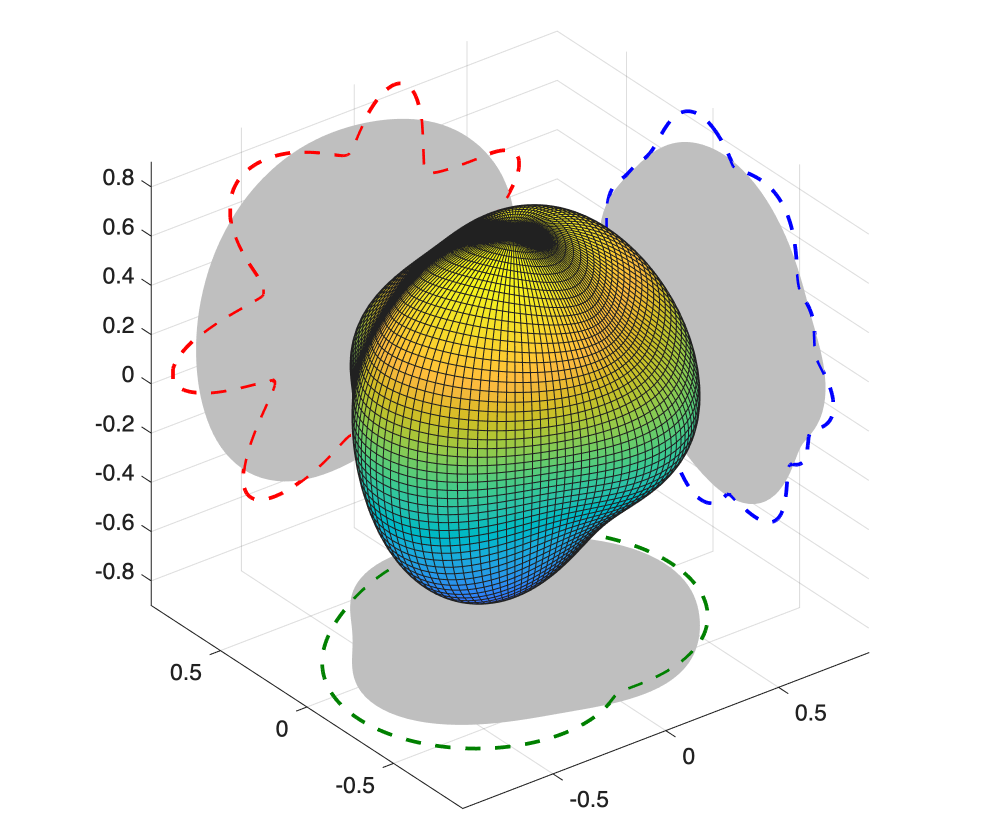}} &
		\subfigure[$M=6$]{\includegraphics[width=0.3\textwidth]{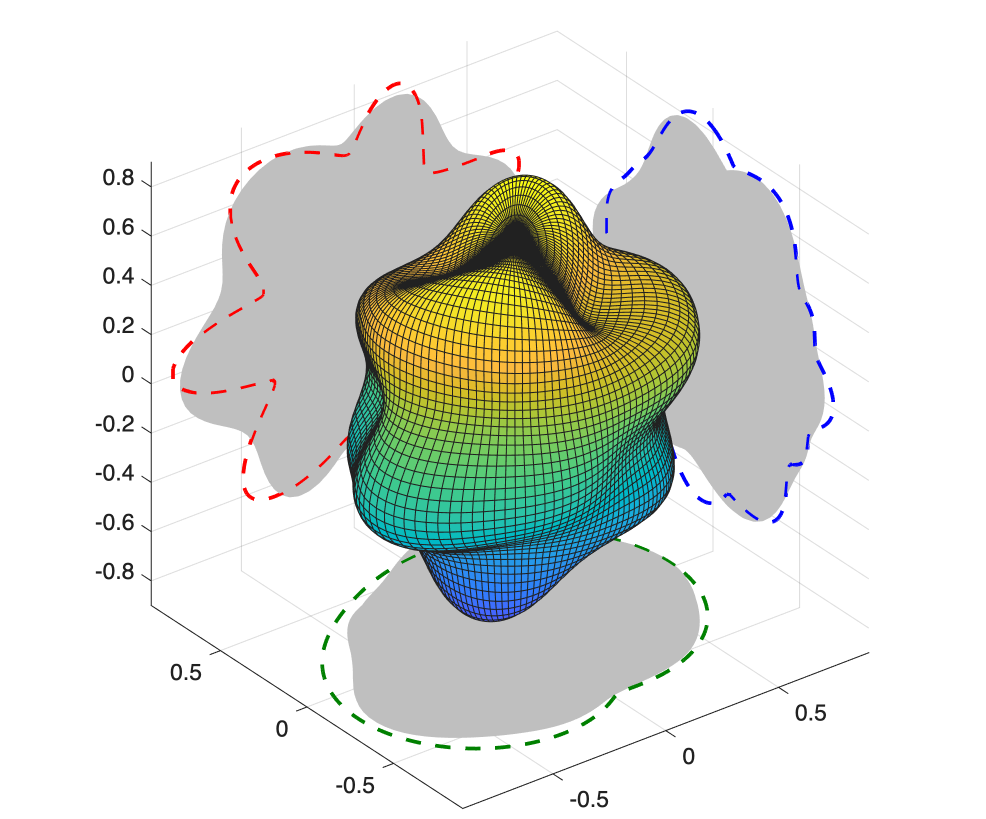}} &
		\subfigure[$M=8$]{\includegraphics[width=0.3\textwidth]{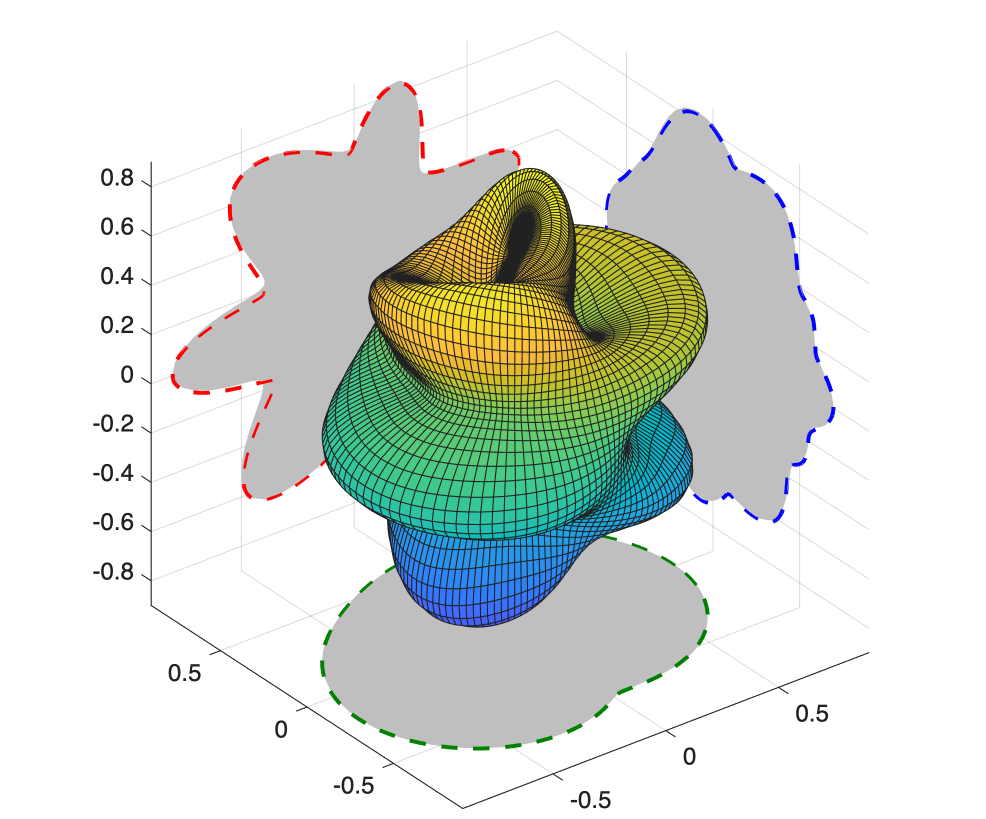}} 
	\end{tabular}

	\caption{Reconstruction of a complex obstacle from time-domain scattered field data generated by four point sources located at \((0,0,5)^\top\), \((0,0,-5)^\top\),\((5,0,0)^\top\) and \((-5,0,0)^\top\) with 1\% noise. $\pmb c^{(0)} = (-0.3, 0.2, -0.5)^\top$, $r^{(0)} = 0.3$. Subfigures (a)--(f) show the initial surface and the progressive reconstructions obtained as the truncation number $M$ increases.}
	\label{fig_ex72}
\end{figure}

\vspace{1.5ex}
{\noindent\bf Example 5: Reconstructions of a complex obstacle.}
\vspace{1ex}

To further assess the capability of the proposed approach for handling geometrically complicated three-dimensional obstacles, we consider a complex surface defined by 
\[
\pmb{p}_D(\theta,\phi)=0.6\Big(1
+0.3\sin(7\theta)\cos\phi
+0.2\sin^{2}(3\theta)\sin(2\phi)
+0.1\cos\theta\Big)\hat{\pmb{x}}(\theta,\phi).
\]
This surface contains multiple sharp undulations and nonconvex features, making the reconstruction problem considerably more challenging.

To ensure the effectiveness of the reconstruction for this complex geometry, the computational parameters are chosen as follows. 
The time-domain scattered field is collected at \(1800\) observation directions (i.e., \(\tilde{n}=30\)) using four incident waves. 
The boundary \(\Gamma_D\) is discretized with \(882\) nodes, and the boundary integrals on the homothetic surface \(\Gamma_{D'}\) are evaluated using the same number of quadrature points (i.e., \(n=20\)). 
Moreover, the maximum truncation order is set to \(M_{\max}=8\), and the scaling factor is chosen as \(\rho=0.1\). 

Figure~\ref{fig_ex7} presents the reconstruction results under different noise levels. 
The true obstacle is shown in subfigure~(a), while subfigures~(b)--(d) correspond to reconstructions with \(1\%\), \(5\%\), and \(10\%\) noise, respectively.
Even at \(10\%\) noise level, the main geometric features of the obstacle are still well recovered, demonstrating the robustness and noise-tolerance of the proposed approach.

To further illustrate the necessity of the incremental truncation technique, we present in Figure~\ref{fig_ex7M8} a reconstruction obtained by directly increasing the truncation number from $M=0$ to $M=8$. 
While the overall shape is reasonably captured, some fine-scale oscillatory features are not fully resolved, and slight deviations appear in several cross-sections compared with the progressive reconstruction.
This behavior suggests that a sudden jump to a high truncation order may reduce the robustness of the reconstruction and makes it more difficult to recover the geometric detail of the obstacle.

In addition, figure~\ref{fig_ex72} further illustrates the reconstruction process for the \(1\%\) noisy data. 
Starting from an initial spherical guess, the reconstructed surface is gradually refined as \(M\) increases, resulting in an increasingly accurate recovery of both the global shape and the highly oscillatory features of the obstacle. 

\section{Conclusions}\label{sec:conclusions}

This work addresses the three-dimensional time-domain forward and inverse acoustic obstacle scattering problem. By combining convolution quadrature technique for time discretization with a spherical harmonics Galerkin method for spatial discretization, we establish an efficient and stable numerical framework for solving the three-dimensional time-domain forward scattering problem. Moreover, we develop a novel time-domain iterative method for three-dimensional inverse acoustic obstacle scattering based on a retarded single-layer potential formulated on a homothetic surface. This strategy avoids the direct treatment of singular integrals on the unknown boundary and significantly improves computational efficiency. Meanwhile, the proposed approach is data-efficient and achieves accurate reconstructions using only one to four incident waves. To improve robustness for the inverse problem, we incorporate an incremental truncation technique that gradually increases the truncation level to reduce the sensitivity to the initial guess.
We also prove that the scattered field generated by the homothetic surface can approximate the exact scattered field arbitrarily, providing the theoretical foundation for the proposed iterative approach. 
Numerical experiments validate the accuracy of the forward scheme and demonstrate that the proposed inverse method can effectively reconstruct both the shape and location of three-dimensional obstacles from time-domain scattered field data.

Future work will focus on extending our method to electromagnetic and elastic wave scattering problems. We also plan to adapt the approach to more challenging geometries, including non-smooth obstacles with sharp corners and multi-scatterer scenarios.

\section*{Acknowledgements}
The work of H. Dong is supported in part by the NSFC Grant 12571454 and the National Key R\&D Program of China 2024YFA1012303. The work of L. Zhao is supported in part by the NSFC Grant 12401565.


\begin{thebibliography}{60}
	\bibitem{Abboud2011}
	T. Abboud, P. Joly, J. Rodr\'{\i}guez, I. Terrasse, Coupling discontinuous Galerkin methods and retarded potentials for transient wave propagation on unbounded domains, J. Comput. Phys. 230 (2011) 5877--5907.
	
	\bibitem{Ammari1}
	H. Ammari, An introduction to mathematics of emerging biomedical imaging, Math\'ematiques \& Applications, Springer, Berlin, 2008.

	\bibitem{Ammari2}
	H. Ammari, G. Bao, J.~L. Fleming, An inverse source problem for Maxwell's equations in magnetoencephalography, SIAM J. Appl. Math. {62} (2002) 1369--1382.
	
	\bibitem{Ammari2013}
	H. Ammari, E. Bretin, J. Garnier, A. Wahab, Time-reversal algorithms in viscoelastic media, European J. Appl. Math. {24} (2013) 565--600.
	
	\bibitem{Anderson2020}
	T. G. Anderson, O. P. Bruno, M. Lyon, High-order, dispersionless ``fast-hybrid'' wave equation solver. Part I: $\mathcal{O}$(1) sampling cost via incident-field windowing and recentering, SIAM J. Sci. Comput. {42} (2020) A1348--A1379.
	
	\bibitem{CQ}
	L. Banjai, S. Sauter, Rapid solution of the wave equation in unbounded domains, SIAM J. Numer. Anal. {47} (2008) 227--249.
	
	\bibitem{Banjai2014}
	L. Banjai, M. Kachanovska, Fast convolution quadrature for the wave equation in three dimensions, J. Comput. Phys. {279} (2014) 103--126.
	
	\bibitem{Barnett2020}
	A. Barnett, L. Greengard, T. Hagstrom. High-order discretization of a stable time-domain integral equation for 3D acoustic scattering, J. Comput. Phys. {402} (2020) 109047.
	
	\bibitem{Borden}
	B. Borden, Mathematical problems in radar inverse scattering, Inverse Problems {18} (2002) R1--R28.
	
	\bibitem{Bruno2024}
	O. P. Bruno, T. Yin, Multiple-scattering frequency-time hybrid solver for the wave equation in interior domains, Math. Comput. {93} (2024) 551-587.
	
	\bibitem{Bukard2009}
	C. Burkard, R. Potthast, A time-domain probe method for three-dimensional rough surface reconstructions, Inverse Probl. Imaging {3} (2009) 259--274.
	
	\bibitem{Cakoni2017}
	F. Cakoni, J. D. Rezac, Direct imaging of small scatterers using reduced time dependent data, J. Comput. Phys. {338} (2017) 371--387.
	
	\bibitem{Cakoni2019}
	F. Cakoni, H. Haddar, A. Lechleiter, On the factorization method for a far field inverse scattering problem in the time domain, SIAM J. Math. Anal. {51} (2019) 854--872.
	
	\bibitem{Cakoni2021}
	F. Cakoni, P. Monk, V. Selgas, Analysis of the linear sampling method for imaging penetrable obstacles in the time domain. Anal. PDE {14} (2021) 667--688.
	
	\bibitem{Castro2024}
	P. de~Castro, E. Silva, E. Fancello, A single level set function approach for multiple material-phases applied to full-waveform inversion in the time domain, Inverse Problems, {40} (2024) 055002.	
	
	\bibitem{Chen2}
	Q. Chen, H. Haddar, A. Lechleiter, A sampling method for inverse scattering in the time domain, Inverse Problems, {26} (2010) 085001.
	
	\bibitem{Colton}
	D. Colton, R. Kress, Inverse Acoustic and Electromagnetic Scattering Theory. 4nd edition, Springer, New York, 2019.
	
	\bibitem{Davies2013}
	P. J. Davies, D. B. Duncan, Convolution-in-time approximations of time domain boundary integral equations, SIAM J. Sci. Comput. {35} (2013) B43--B61.
	
	\bibitem{Epstein2016}
	C. Epstein, L. Greengard, T. Hagstrom, On the stability of time-domain integral equations for acoustic wave propagation, Discrete Contin. Dyn. Syst. {36} (2016) 4367--4382.
	
	\bibitem{Ganesh}
	M. Ganesh, I.G. Graham, A high-order algorithm for obstacle scattering in three dimensions, J. Comput. Phys. {198} (2004) 211--242.
	
	\bibitem{Ganesh2023}
	M. Ganesh, F. Le Lou\"er, A high-order algorithm for time-domain scattering in three dimensions, Adv. Comput. Math. {49} (2023) 46.
	
	\bibitem{Geng2025}
	C. Geng, M. Song, X. Wang, Y. Wang, Time-domain direct sampling method for inverse electromagnetic scattering with a single incident source, SIAM J. Imag. Sci. {18} (2025) 1208-1234.
	
	\bibitem{Guo2013}
	Y. Guo, P. Monk, D. Colton, Toward a time domain approach to the linear sampling method, Inverse Problems {29} (2013) 095016.
	
	\bibitem{Guo2016}
	Y. Guo, P. Monk, D. Colton, The linear sampling method for sparse small aperture data, Appl. Anal. {95} (2016) 1599--1615.
	
	\bibitem{Guo20162}
	Y. Guo, D. H$\ddot{\rm o}$mberg, G. Hu, J. Li, H. Liu, A time domain sampling method for inverse acoustic scattering problems, J. Comput. Phys. {314} (2016) 647--660.
	
	\bibitem{Guo2024}
	Y. Guo, H. Li, X. Wang, A novel time-domain direct sampling approach for inverse scattering problems in acoustics, SIAM J. Appl. Math. {84} (2024) 2152--2174.
	
	\bibitem{Haddar}
	H. Haddar, A. Lechleiter, S. Marmorat, An improved time domain linear sampling method for Robin and Neumann obstacles, Appl. Anal. {93} (2014) 369--390.
	
	\bibitem{Haddar2020}
	H. Haddar, X. Liu, A time domain factorization method for obstacles with impedance boundary conditions, Inverse Problems {36} (2020) 105011.
	
	\bibitem{Hou2024}
	S. Hou, H. Wang, An efficient algorithm for time-domain acoustic scattering in three dimensions by layer potentials, J. Comput. Phys. {514} (2024) 113258.
	
	\bibitem{Hou2025}
	S. Hou, J. Liu, H. Wang, On the computation of time-domain acoustic scattering by multiple obstacles using a decomposition method, SIAM J. Appl. Math. {85} (2025) 2683--2703.
	
	\bibitem{Ivanyshyn1}
	O. Ivanyshyn, R. Kress, Identification of sound-soft 3D obstacle from phaseless data, Inverse Probl. Imaging {4} (2010) 131--149.
	
	\bibitem{Kahlaoui2025}
	H. Kahlaoui, A non-iterative reconstruction method for the geometric inverse problem for the wave equation, J. Sci. Comput. {102} (2025) 61.
	
	\bibitem{Knoller2025}
	M. Kn$\ddot{\rm o}$ller, J. Nick, The temporal domain derivative in inverse acoustic obstacle scattering, Numer. Math. (2025), to appear. 
	
	\bibitem{Tikhonov}
	R. Kress, Newton's method for inverse obstacle scattering meets the method of least squares, Inverse Problems {19} (2003) S91--S104.
	
	\bibitem{Liu2023}
	X. Liu, J. Song, F. Pourahmadian, H. Haddar, Time-versus frequency-domain inverse elastic scattering: theory and experiment, SIAM J. Appl. Math. {83} (2023) 1296-1314.
	
	\bibitem{Lubich1994}
	C. Lubich, On the multistep time discretization of linear initial-boundary value problems and their boundary integral equations. Numer. Math. {67} (1994) 365--389.
	
	\bibitem{Luke2006}
	D. R. Luke, R. Potthast, The point source method for inverse scattering in the time domain, Math. Methods Appl. Sci. {29} (2006) 1501--1521.
	
	\bibitem{Polz2019}
	D. P\"olz, M. Schanz, Space-time discretized retarded potential boundary integral operators: quadrature for collocation methods, SIAM J. Sci. Comput. {41} (2019) A3860--A3886.
	
	\bibitem{Prunty2019}
	A. C. Prunty, R. K. Snieder, Theory of the linear sampling method for time-dependent fields, Inverse Problems {35} (2019) 055003.
	
	\bibitem{Sauter2013}
	S. Sauter, A. Veit, A Galerkin method for retarded boundary integral equations with smooth and compactly supported temporal basis functions, Numer. Math. {123} (2013) 145--176.
	
 	\bibitem{Sayag2022}
	A. Sayag, D. Givoli, Shape identification of scatterers using a time-dependent adjoint method, Comput. Methods Appl. Mech. Engrg. {394} (2022) 114923.
 
	\bibitem{sayas}
	F. Sayas, Retarded potentials and time domain boundary integral equations, Springer Series in Computational Mathematics, Switzerland, 2016.
	
	\bibitem{Wang2021}
	B. Wang, Z. Yang, L. Wang, S. Jiang, On time-domain NRBC for Maxwell's equations and its application in accurate simulation of electromagnetic invisibility cloaks, J. Sci. Comput. {86} (2021) 20.
	
	\bibitem{Zhao2021}
	L. Zhao, H. Dong, F. Ma, Inverse obstacle scattering for acoustic waves in the time domain, Inverse Probl. Imaging, {15} (2021) 1269--1286.
	
	\bibitem{Zhao2022}
	L. Zhao, H. Dong, F. Ma, Inverse obstacle scattering for elastic waves in the time domain, Inverse Problems {38} (2022) 045005.
	

	
\end{thebibliography}
\end{document}